\documentclass[11pt]{amsart}
\usepackage{amsthm,amsmath,amsfonts,amssymb,nicefrac}
\usepackage{color}
\usepackage{graphicx}
\usepackage{float,subcaption}
\usepackage{wrapfig}
\usepackage[dvips]{epsfig}
\usepackage[english]{babel}
\usepackage[utf8]{inputenc}
\usepackage[T1]{fontenc}
\usepackage{amsfonts}
\usepackage{enumerate}
\usepackage{eurosym}
\usepackage{xcolor}
\usepackage[bottom]{footmisc}
\usepackage{alltt}
\usepackage{algorithm}
\usepackage{algpseudocode}
\usepackage{booktabs}
\usepackage{mathtools}

\usepackage{caption}
\makeatletter 

\newcommand{\nosemic}{\renewcommand{\@endalgocfline}{\relax}}
\newcommand{\dosemic}{\renewcommand{\@endalgocfline}{\algocf@endline}}
\let\oldnl\nl
\newcommand{\nonl}{\renewcommand{\nl}{\let\nl\oldnl}}
\makeatother

\theoremstyle{plain}             
\newtheorem{theorem}{Theorem}[section]
\newtheorem{lemma}[theorem]{Lemma}     

\newtheorem{corollary}[theorem]{Corollary}

\newcommand{\spn}{\operatorname{span}}

\algdef{SE}[SUBALG]{Indent}{EndIndent}{}{\algorithmicend\ }%
\algtext*{Indent}
\algtext*{EndIndent}
\algnewcommand\AlgOn[1]{\State\textbf{on} #1 \textbf{do}\Indent}
\algnewcommand\AlgEndOn{\EndIndent\State\textbf{end on}}

\begin{document}
\title[A fault-tolerant DDM based on space-filling curves]
{A fault-tolerant domain decomposition method based on space-filling curves}
\author{Michael Griebel}
\address{Michael Griebel,
Institut f\"ur Numerische Simulation, Universit\"at Bonn, Friedrich-Hirzebruch-Allee 7, 53115 Bonn,
and Fraunhofer Institute for Algorithms and Scientific Computing (SCAI), Schloss Birlinghoven, 
53754 Sankt Augustin, Germany}
\email{griebel@ins.uni-bonn.de}
\author{Marc-Alexander Schweitzer}
\address{Marc-Alexander Schweitzer,
Institut f\"ur Numerische Simulation, Universit\"at Bonn, Friedrich-Hirzebruch-Allee 7, 53115 Bonn,
and Fraunhofer Institute for Algorithms and Scientific Computing (SCAI), Schloss Birlinghoven, 
53754 Sankt Augustin, Germany}
\email{schweitzer@ins.uni-bonn.de}
\author{Lukas Troska}
\address{Lukas Troska,
Institut f\"ur Numerische Simulation, Universit\"at Bonn, Friedrich-Hirzebruch-Allee 7, 53115 Bonn,
and Fraunhofer Institute for Algorithms and Scientific Computing (SCAI), Schloss Birlinghoven, 
53754 Sankt Augustin, Germany}
\email{lukas.troska@scai.fraunhofer.de}

\keywords{Domain decomposition, parallelization, space-filling curve, fault tolerance}

\begin{abstract}
We propose a simple domain decomposition method for $d$-dimensional elliptic PDEs which involves an overlapping decomposition into local subdomain problems and a global coarse problem. It relies on a space-filling curve to create equally sized subproblems and to determine a certain overlap based on the one-dimensional ordering of the space-filling curve.  
Furthermore we employ agglomeration and a purely algebraic Galerkin discretization in the construction of the coarse problem. 
This way, the use of $d$-dimensional geometric information is avoided.
The subproblems are dealt with in an additive, parallel way, which gives rise to a subspace correction type linear iteration and a preconditioner for the conjugate gradient method.
To make the algorithm fault-tolerant we store on each processor, besides the data of the associated subproblem, a copy of the coarse problem and also the data of a fixed amount of neighboring subproblems with respect to the one-dimensional ordering of the subproblems induced by the space-filling curve. This redundancy then allows to restore the necessary data if 
  processors fail during the computation.  Theory from \cite{Griebel.Oswald:2019} supports that the convergence rate of such a linear iteration method stays the same in expectation, and only its order constant deteriorates slightly due to the faults. We observe this in numerical experiments for the preconditioned conjugate gradient method in slightly weaker form as well.
Altogether, we obtain a fault-tolerant, parallel and efficient domain decomposition method based on space-filling curves which is especially suited for higher-dimensional elliptic problems. 
	
\end{abstract}

\maketitle
\section{Introduction}
Higher-dimensional problems beyond four dimensions appear in many mathematical models in medicine, finance, engineering, biology and physics. 
Their numerical treatment poses a challenge for modern high performance compute systems. 
With the advent of petascale compute systems in recent years and exascale computers to arrive in the near future, there is tremendous parallel compute power available for parallel simulations. 
But while higher-dimensional applications involve a huge number of degrees of freedom, their efficient parallelization by e.g. conventional domain decomposition approaches is difficult. Load balancing and scalability might be dimension-dependent, especially for geometry-based domain decomposition approaches, and necessary communication might grow strongly with increasing dimensionality. 
Thus one aim is to find parallel algorithms with good load balancing, good scaling properties and moderate communication costs especially for higher-dimensional problems. 
Moreover systems with hundreds of thousand or even millions of processor units will be increasingly prone to failures, which can corrupt the results of parallel solvers or renders them obsolete at all. 
It is predicted that large parallel applications may suffer from faults as frequently as once every 30 minutes on future exascale platforms \cite{snir14}.
Thus a second aim is to derive not just balanced, scalable and fast parallel algorithms, but to make them fault-tolerant as well. 
Besides hard errors, for which hardware mitigation techniques are under development, there is the issue of soft errors \cite{kavi13,snir14,trea05}. For further details and literature on resilience and fault tolerance, see \cite{Agullo2020}.
Altogether, the development of fault-tolerant and numerically efficient parallel algorithms are of utmost importance for the simulation of large-scale problems.

In this article we focus on algorithm-based fault tolerance. 
Indeed, a fault-tolerant, parallel, iterative domain decomposition method can be interpreted as an instance of the {\em stochastic} subspace correction approach. 
For such methods there exists a general theoretical foundation for convergence, which was developed in a series of papers \cite{Griebel.Oswald:2011, Griebel.Oswald:2016,Griebel.Oswald:2017*1}. 
Moreover, for a conventional geometry-based domain decomposition approach, this theory was already employed in \cite{Griebel.Oswald:2019} to show algorithm-based fault tolerance theoretically and, for the two-dimensional case, also practically under independence assumptions on the random failure of subproblem solves. 
We now propose a simple domain decomposition method for $d$-dimensional elliptic PDEs which works for higher dimensions. 
Besides an overlapping decomposition into local subdomain problems, it also involves a global coarse problem. 
To create nearly equally sized subproblems, we rely on {\em space-filling curves}. 
This way, the overall number $N$ of degrees of freedom is partitioned for $P$ processors into $N/P$-sized subproblems regardless of the dimension of the PDE and the number of available processors. 
This is in contrast to many geometry-based domain decomposition methods, where -- e.g. in the uniform grid case with mesh size $h$ and $N\approx h^{-d}$ depending exponentially on $d$  -- the number of processors is usually to be chosen as a power of $d$. 
Furthermore, in our method, the overlap is determined based on the one-dimensional ordering of the space-filling curve as well.  
Moreover we employ agglomeration and a purely algebraic Galerkin discretization to again avoid $d$-dimensional geometric information in the construction of the coarse space problem. 
The subproblems and the coarse space problem are dealt with in an additive, parallel way, which leads to an additive Schwarz subspace correction method that resembles a block-Richardson-type linear iteration. To speed up convergence, we also employ this approach in a preconditioner for the conjugate gradient (CG) method for the fine grid discretization of the PDE.
To this end, we store the global coarse space problem redundantly on each processor and also solve it redundantly on each processor in parallel.
Moreover, to gain fault tolerance, we store on each processor not just the data of the associated subproblem and the coarse problem but also the data of a fixed amount of neighboring subproblems with respect to the {\em one-dimensional} ordering of the subproblems, which is induced by the space-filling curve. 
This results in sufficient redundancy of the stored data, whereas the amount of stored data is just enlarged by a constant. 
If a processor now fails in the course of the computation, a new replacement processor is invoked from a reserve batch instead of the faulty one, the corresponding necessary data is transferred to it from (one of) the neighboring processors, and the iterative method proceeds with its computation on this processor as well.
Altogether, we obtain an algorithm-based fault-tolerant, parallel, iterative algorithm, which can be interpreted as an instance of the {\em stochastic} subspace correction approach. 
Again, the theory in \cite{Griebel.Oswald:2019} supports that the convergence rate of the overall linear Richardson-type iteration stays in expectation the same, and only the order constant deteriorates slightly due to the faults, provided that the number of faults stays bounded and their occurrence among the processors is sufficiently well distributed. For the preconditioned conjugate gradient approach we do not have such a theory, but a similar, though slightly weaker behavior can nevertheless be observed in practice. 
Altogether, we obtain a fault-tolerant, well-balanced, parallel domain decomposition method, which is based on space-filling curves and which is thus especially suited for higher-dimensional elliptic problems.

The remainder of this paper is organized as follows:
In section \ref{sec:DDM} we discuss our domain decomposition method, which is based on a space-filling curve. We first give a short overview on domain decomposition methods and their properties for elliptic PDEs. Then we discuss space-filling curves and their peculiarities. Finally, we present our algorithm and its features.
In section \ref{sec:fault} we deal with algorithmic fault tolerance. Here we recall its close relation to randomized subspace correction for our setting. Then we present a fault-tolerant variant of our domain decomposition method. 
In section \ref{sec:numer} we discuss the results of our numerical experiments. We first define the model problem which we employ. Then we give convergence and parallelization results. Furthermore we show the behavior of our method under failure of processors.
Finally we give some concluding remarks in section \ref{sec:conclude}.

\section{A domain decomposition method based on space-filling curves}\label{sec:DDM}
\subsection{Domain decomposition}
The domain decomposition approach is a simple method for the solution of discretized partial differential equations and is typically used as a preconditioner for the conjugate gradient method or other Krylov iterative methods. Its idea can be traced back to Schwarz \cite{Schwarz}.
Depending on the specific choice of the subdomains, one can distinguish between overlapping and non-overlapping domain decomposition methods, where the subdomains geometrically overlap either to a certain extent or intersect only at their common interfaces. The latter are often also called iterative substructuring methods in the engineering community.
It turned out that such simple domain decomposition methods can not possess fast convergence rates and thus, starting in the mid 80s, various techniques had been developed to introduce an additional coarse scale problem, which provides a certain amount of global transfer of information across the whole domain and thus substantially speeds up the iteration. 
For the overlapping case it could be shown in \cite{dryawidlund87} that the condition number of the fine grid system preconditioned by such a two-level additive Schwarz method is of the order
\begin{equation}\label{AAA}
O(1+H/\delta),
\end{equation}
where $\delta$ denotes the size of the overlap and $H$ denotes the coarse mesh size. This bound also holds for small overlap \cite{dryawidlund94} and can not be improved further \cite{brenner20}.
Thus, if the quotient of the coarse mesh size $H$ and the overlap $\delta$ stays constant, the method is indeed optimally preconditioned and weakly scalable.
For further details on domain decomposition methods see e.g. the books \cite{smithborstedgropp96, quateronivalli99, toselliwidlund04, doleanjolivetnataf15}.

We obtain a two-level additive Schwarz method as follows: Consider an elliptic differential equation 
\begin{equation}\label{ellprob}
\mathcal{L} u=f \end{equation}
in the domain $\Omega \subset \mathbb{R}^d$, e.g. the simple Poisson problem on a $d$-dimensional cube. Using a conforming finite element, a direct finite difference or a finite volume discretization involving $N$ degrees of freedom and mesh size $h\approx N^{-1/d}$, we arrive at the system of linear equations
\begin{equation}\label{LES}
	A x = b
\end{equation}
with sparse  stiffness matrix $A \in \mathbb{R}^{N \times N}$, right hand side vector $b \in \mathbb{R}^N$ and unknown coefficient vector $x \in \mathbb{R}^N$, which needs to be solved.
Suppose that 
$$\Omega=\bigcup_{i=1}^P \Omega_i$$ 
is covered by a finite number $P$ of well-shaped subdomains $\Omega_i$ of diameter $\approx H$ which might locally overlap. It is silently assumed that $h<<H$ and that the subdomains are aligned with the fine mesh. 
Now denote by $N_i$ the number of grid points associated to each $\Omega_i$, i.e. the degrees of freedom associated to the subdomains $\Omega_i, i=1,\ldots,P$.
Then denote by $R_i: \mathbb{R}^{N} \to \mathbb{R}^{N_i}$ the restriction operators, which map the entries of the coefficient vector $x$ corresponding to the full grid on $\Omega$ to the coefficient vectors $x_i$ corresponding to the local grids on the subdomains $\Omega_i$.
Analogously denote by $R^T_i: \mathbb{R}^{N_i} \to \mathbb{R}^{N}$ the extension operators, which map the coefficient vectors from the local grid on the subdomains $\Omega_i$ to that of the full grid on $\Omega$ via the natural extension by zero.
Then the local stiffness matrices associated to the subdomains $\Omega_i$ can be denoted as $A_i \in \mathbb{R}^{N_i \times N_i}$  with ${A_i := R_i A R^T_i}$.
Finally, we add a coarse space problem with dimension $N_0$ as a second level via the restriction operator $R_0: \mathbb{R}^{N} \to \mathbb{R}^{N_0}$, which maps from the full grid on $\Omega$ to the respective global coarse space. The associated coarse stiffness matrix then can be generated via the Galerkin approach as 
$A_0 := R_0 A R^T_0$. Altogether, with the one-level additive Schwarz operator
\begin{equation}\label{AS2a}
C_{(1)}^{-1}:=   \sum_{i=1}^P  R_i^T A_i^{-1} R_i,
\end{equation}
this gives the two-level additive Schwarz operator
\begin{equation}\label{AS2}
	C_{(2)}^{-1}:= R_0^T A_0^{-1} R_0 + C_{(1)}^{-1} =   \sum_{i=0}^P  R_i^T A_i^{-1} R_i,
\end{equation}
which can be used, with a properly chosen relaxation parameter, directly for  a linear iterative method or as preconditioner within a steepest descent or conjugate gradient solver for (\ref{LES}). A notational variant based on space splittings is given in \cite{Griebel.Oswald:2019}.
Note that there are more sophisticated space splittings which follow the Bank-Holst technique \cite{BH2003},
where the coarse problem is formally avoided by including a redundant copy of it into each of the subdomain problems with $i=1,\ldots,P$. We will indeed follow this approach later on.

Now, if the condition number $\kappa(C_{(2)}^{-1}A)= \lambda_{\max}(C_{(2)}^{-1}A)/\lambda_{\min}(C_{(2)}^{-1}A)$ of the preconditioned system is
independent of the number $P$ of subproblems for fixed $N$, we obtain
{\em strong scalability}. If it is independent of the quotient $N/P$, i.e. the problem size per subdomain and thus per processor stays fixed, we obtain {\em weak scalability}. 
Moreover, if it is independent of the number $N$ of degrees of freedom, we would have an optimally preconditioned method, which however still may depend on $P$ and might thus not be scalable.
Note furthermore that we employ here for reasons of simplicity a direct solver for $A_i^{-1}$ on all subproblems and for $A_0^{-1}$ on the coarse scale, which involves Gaussian elimination and comes with a certain cost. However, the corresponding matrix factorization needs to be performed just once at the beginning and, in the plain linear iteration or in the preconditioned conjugate gradient iteration, only the cheaper backward and forward steps need to be employed. Alternatively, approximate iterative methods might be used as well, like the multigrid or BPX-multilevel method, which would even results in optimal linear cost for the subproblem solves. This given, to achieve a mesh-independent condition number for the preconditioned system $C_{(2)}^{-1}A$ with $C_{(2)}^{-1}$ as in (\ref{AS2}), one usually chooses for the coarse problem a suitable FE space on the mesh of domain partitions, where a linear FE space will do for a second-order elliptic problem such as (\ref{ellprob}). Mild shape regularity assumptions on the overlapping subdomains  $\Omega_i$ then guarantee robust condition number estimates of the form  
$
\kappa\le c(1+\frac H \delta),
$
see \cite[Theorem 3.13]{toselliwidlund04}.
Dropping the coarse grid problem, i.e. considering a one-level preconditioner as in (\ref{AS2a}) without the coarse problem $ R_0^T A_0^{-1} R_0 $,
would lead to the worse bound 
$\kappa\le c H^{-2}(1+H/\delta)$.
Note that, even though these estimates imply a deterioration of the condition number proportional to $\delta^{-1}$ if $\delta\to h$, in practice good performance has been observed when only a few layers of fine mesh cells form the overlap at the surface of the $\Omega_i$. 
With the use of an additional coarse grid problem based on piecewise linears, an optimal order of the convergence rate is then guaranteed for elliptic problems. The additional coarse grid problem results in a certain communication bottleneck,  which however can by principle not be avoided and is inherent in all elliptic problems.  Similar issues arise for multigrid and multilevel algorithms as well, but these methods are more complicated to parallelize in an efficient way on large compute systems. Moreover their achieved convergence rate and cost complexity is not better than for the domain decomposition approach with coarse grid, at least in order, albeit the order constant might be.

In the above two-level method,  the coarse problem is mostly determined as a kind of {\em geometric} coarse grid or mesh, which corresponds to the specifically chosen set of subdomains and also involves a geometric coarse mesh size. Moreover, on this geometric coarse grid, usually a set of  piecewise $d$-linear basis functions is employed, which directly gives a geometric $d$-dimensional prolongation to the fine grid space by e.g. interpolation. Furthermore a conventional discretization by finite elements (FE) is invoked on this coarse grid either by direct discretization or by Galerkin coarsening. This is the basis for both the proofs of the theoretical convergence estimates and the practical implementation of the algorithms. Besides, the wire basket type constructions \cite{bramblepasciakschatz89, smith90} and related techniques are used in theory and practice with good success. They result in somewhat more general coarse spaces, mimicking discrete harmonics to some extent, but are also based on geometric $d$-dimensional information of the fine mesh, the respective domain partitions and their boundaries. Altogether, such geometric coarse spaces work well for problems with up to three dimensions. Their extension to higher dimension, however, is neither straightforward nor simple and involves local costs which in general scale exponentially in $d$. More information on the development of coarse spaces for domain decomposition can be found in \cite{widlund09}.

Besides a mesh-based geometric coarse grid problem, we can derive a coarse problem in a purely {\em algebraic} way. To this end,  let $V_H$ be a coarse space with $N_0:=\dim(V_H)$ and let $Z$ be a basis for it, i.e. $V_H= \spn Z$. The restriction $R_0:\mathbb{R}^{N}\to \mathbb{R}^{N_0}$ from the fine grid space $V_h$ to the coarse space $V_H$ can be (algebraically) defined as the matrix $R_0:=Z^T\in \mathbb{R}^{N_0 \times N}$ and the coarse space discretization is again given via the Galerkin approach as
$A_0 := R_0 A R^T_0.$

There are now different specific choices for $Z$ or $R_0^T$, respectively.
In \cite{nicolaides87} it was suggested to employ the kernel of the underlying differential operator as coarse space, i.e. the constant space for the Laplace operator. Thus, with $N_0=P$,
$$R_0^T:=(R_i^T D_iR_i {\bf 1})_{1 \leq i \leq P}$$
with $ {\bf 1}=(1,\ldots,1)^T$ and with diagonal matrices $D_i$ such that a partition of unity results, i.e.
$$\sum_{i=1}^{P} R_i^T D_iR_i = I.
$$ 
Indeed, it is observed that the associated two-level preconditioner gives good results \cite{mansfield88, mansfield90} and weak scaling is achieved in practice. Note that this approach can be easily generalized to the case $N_0 = q P$ such that $q$ degrees of freedom are associated on the coarse scale to each subdomain instead of just one.
Improved variants, namely the balanced method \cite{mandel93} and the deflation method \cite{vuiknabben06}, had been developed subsequently.
With the definitions
\begin{equation}\label{balan}
 F:=R_0^T A_0^{-1}R_0 \quad \mbox{ and } \quad G:=I-AF
\end{equation}
and the plain one-level additive Schwarz operator $C_{(1)}^{-1}$ from (\ref{AS2a}),
we get the additively corrected operator $C_{(1)}^{-1}+F$ due to Nicolaides \cite{nicolaides87}, the deflated approach $G^TC_{(1)}^{-1} +F$ (\cite{vuiknabben06}) and the balanced version \cite{mandel93}
\begin{equation}\label{bala}
	C_{(2),bal}^{-1}:= G^TC_{(1)}^{-1}G+F.
\end{equation}
Closely related are agglomeration techniques inspired by the algebraic multigrid method, volume agglomeration methods stemming from multigrid in the context of finite volume discretizations and so-called spectral coarse space constructions, see e.g. \cite{alcinallain11} and \cite{doleanjolivetnataf15}.

\subsection{Space-filling curves}
A main question for domain decomposition methods is how to construct the partition $\{ \Omega_i\}_{i=1}^P$ in the first place. To this end, for a fixed number $P$, one aim is surely to derive a set of subdomains which involves equal size and thus equal computational load for each of the subdomains. If we just consider uniform fine grid discretizations, a simple uniform and geometric decomposition of the mesh would need $P$ to be a power of $d$, i.e. $P=\bar P^d$ with $\bar P$ being the amount of subdomain splitting in each coordinate direction. This however prohibits the use of a slowly growing number of processors and merely allows numbers of processors which are to be chosen as a power of $d$.
Similar issues arise for anisotropic mesh sizes $h=(h_1,\ldots, h_d)$ with $h_i \neq h_j$, more general quasi uniform finite element meshes (for which often  mesh partitioners like METIS or similar methods are employed) or adaptive meshes, where a well-balanced domain decomposition can be complicated to derive.
In this article, we follow the lead of \cite{griebelzumbusch00,griebelzumbusch01} and rely on {\em space-filling curves} to create nearly equally sized subproblems. 
This way, the overall number $N$ of degrees of freedom is partitioned for $P$ processors into $N/P$-sized subproblems (up to round-off deviations, i.e. up to a value of just one) regardless of the dimension of the PDE and the number of available processors. For further details see also \cite{GriebelKnapekZumbusch2007, Bader2013, zumbusch03}.

In general, we consider the situation where a domain $X$ in a one-dimensional space $X\subset \mathbb{R}$ is mapped to a domain $Y$ in a $d$-dimensional space $Y \subset \mathbb{R}^d$, $d \geq  1$, i.e. we consider the mapping
\[ s:X\to Y, x\to y := s(x).
\]
The question is: Is there a one-dimensional curve that runs through all points of the $d$-dimensional domain, i.e. that is indeed filling the whole domain $Y$?
In 1878 Cantor \cite{Cantor1878} discovered that there is a {\em bijective} function between  any  two  finite-dimensional  smooth  manifolds  independent  of  their  dimension. This is a remarkable and surprising result: For example, it is possible to map a one-dimensional curve 1:1 onto a two-dimensional surface and, consequently, the surface contains the same amount of points as the curve. 
Especially the points of the unit interval $[0,1]$ can be mapped onto the unit square $[0,1]^2$ and vice versa. 
Thus such a mapping from the interval to the square is  indeed filling the whole square. 
The second question then is: Is there a {\em continuous} mapping between any  two  finite-dimensional  smooth  manifolds $X$ and $Y$ independent  of  their  dimension?
In 1879 Netto \cite{Netto1879} settled this issue by proving that, if the dimensions of the two manifolds are different, such a function is necessarily discontinuous.
But what happens if one sacrifices bijectivity? Does a {\em surjective}, continuous map from the one-dimensional interval to the higher-dimensional manifold exist? This question was positively answered in 1890 by Peano \cite{Peano1890}. He was the first to present such a map, which is now known as Peano's space-filling curve. 
It maps the unit interval into the plane where the image has positive Jordan content.
One year later, Hilbert \cite{Hilbert1891} gave another map with such properties. 
While Peano's curve was defined purely analytically in the first place, Hilbert's approach was given by means of a geometric construction, see below.
In subsequent years, many other space-filling curves were discovered, for example
the curves by Lebesgue, 
Moore, Sierpi\'nski, Polya and Schoenberg.
For a detailed overview on space-filling curves we refer to the book \cite{Sagan1994}.

The construction of the various space-filling curves in general follows a specific recursive process. They differ in
the dimension $d$ of the space to be filled,
the shape of the space to be filled, i.e. the initial object (for example, a square or a triangle  in the case  $d=2$),
	the refinement rule for the decomposition of the respective objects (for example, subdivision by a factor of $1/2$ or $1/3$ in each direction),
	and the rule how the resulting sub-objects are to be connected or ordered.
The recursive construction process can be explained most easily for the two-dimensional case and for the refinement into four subsquares, 
i.e. for subdivision by a factor of $1/2$ in each coordinate direction. 
Indeed, this is the refinement rule which is used for the Hilbert curve, the Moore curve and the Lebesque curve.
Then the associated space-filling curve $s:[0,1]\to[0,1]^2$ is given as the limit of a sequence of curves $s_n:[0,1]\to[0,1]^2,  n=1,2,3,\ldots$.
Every curve  $s_n$ connects the centers of the $4^n$ squares that are created by successive subdivision of the unit square by line
segments in a specific way. 
The curve $s_{n+1}$ results from the curve $s_n$ as follows:  Each square is subdivided, the centers
of the newly created four smaller squares are connected in a given
order, and all the $4^n$ groups of $4^{n+1}$ centers of smaller squares 
are connected in the order given by the curve $s_n$. In this sense, the 
curve $s_{n+1}$ refines the curve $s_n$.

The various curves differ in the ordering of the centers
in each refinement step. In the case of the Lebesgue curve, the same
ordering is used everywhere, as shown in Figure~\ref{fig:leb}. 
For the Hilbert curve, the ordering is chosen 
in such a way that, whenever two successive  centers are connected by a straight line, only the 
common edge of the two squares is crossed. The construction
is made clearer in Figure~\ref{fig:hil}. 
There, from each iteration to the next, all existing subsquares are subdivided into four
smaller subsquares. These four subsquares are then connected by a pattern that is obtained
by rotation and/or reflection of the fundamental pattern given in Figure~\ref{fig:hil} (left). 
As it is known where the current discrete curve will enter and exit the existing subsquare, 
one can determine how to orientate the local pattern and can ensure that each square's first subsquare touches the previous square's last subsquare. 
One can show that the sequence $s_n$ for Hilbert's curve converges uniformly to a
curve $s$, which implies that this limit curve is continuous. 
For the Lebesgue curve, the sequence only converges pointwise
and the limit is discontinuous.
\begin{figure}
\begin{center}
    \leavevmode
    \includegraphics[width=0.20 \textwidth]{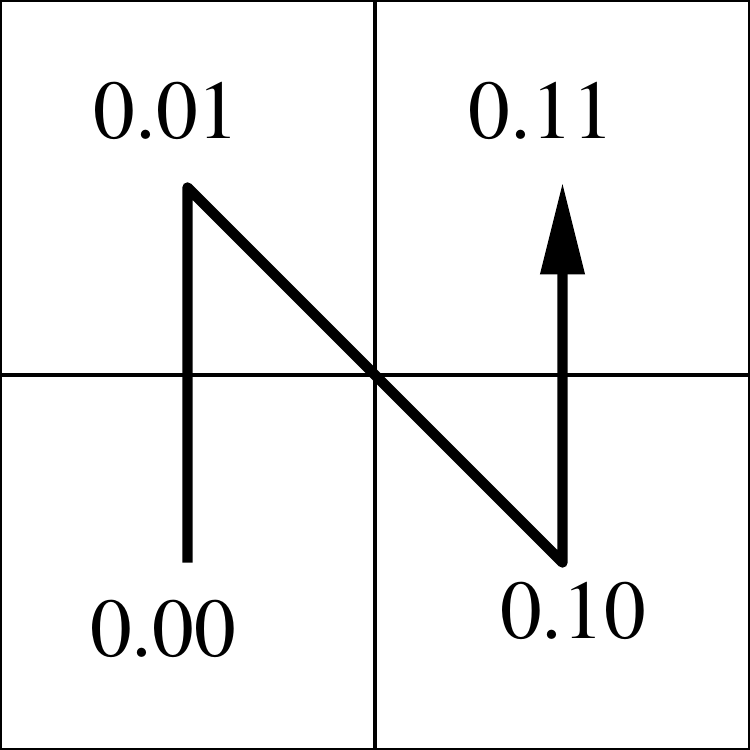} ~~~ 
    \includegraphics[width=0.20 \textwidth]{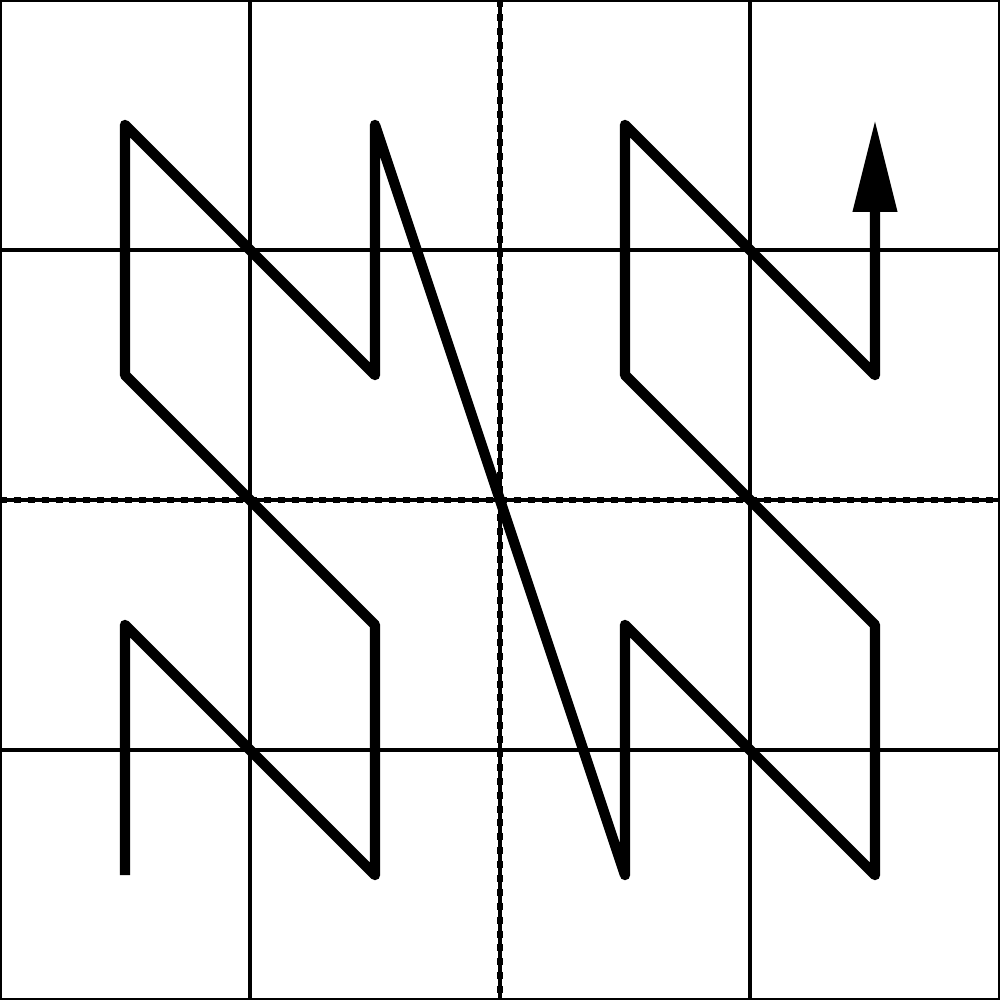} ~~~ 
    \includegraphics[width=0.20 \textwidth]{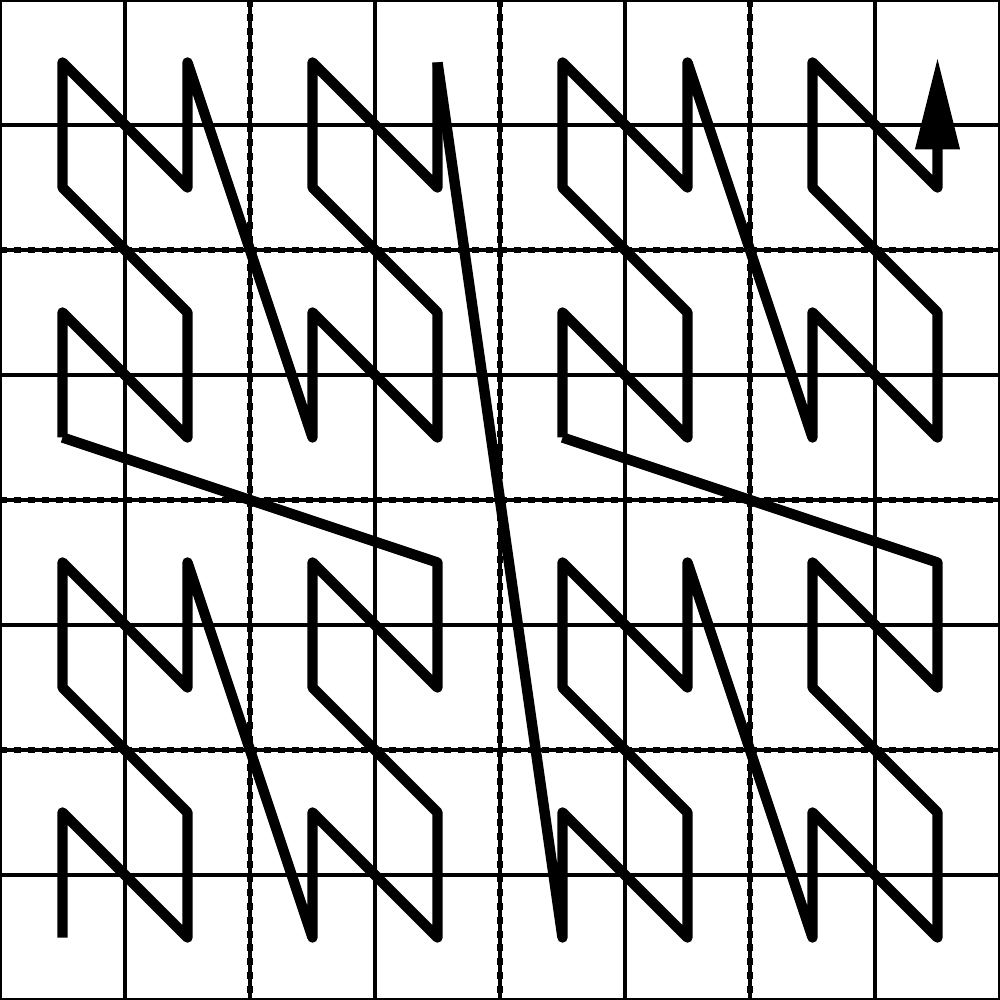}
    \caption{Three steps in the construction of the Lebesgue curve.
            \index{Lebesgue curve}}
    \label{fig:leb}
\end{center}
\end{figure}
\begin{figure}
\begin{center}
    \leavevmode
    \includegraphics[width=0.20 \textwidth]{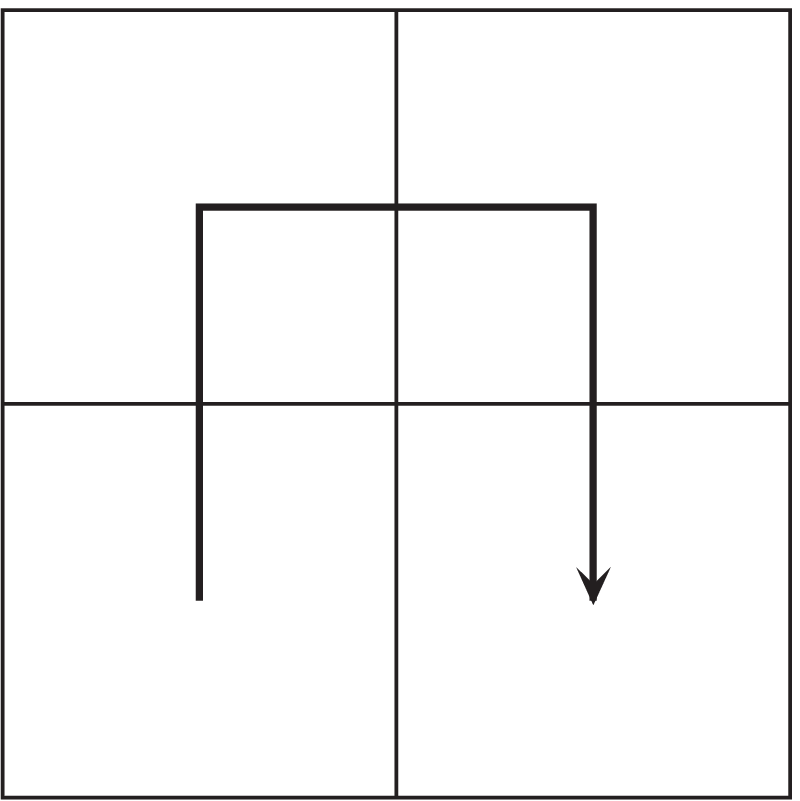} ~~~ 
    \includegraphics[width=0.20 \textwidth]{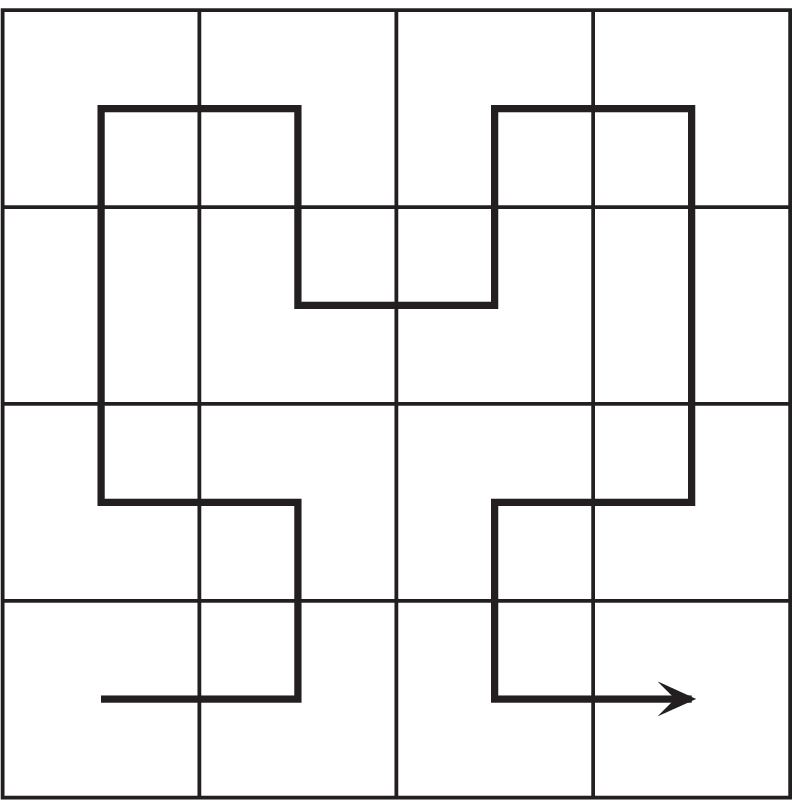} ~~~ 
    \includegraphics[width=0.20 \textwidth]{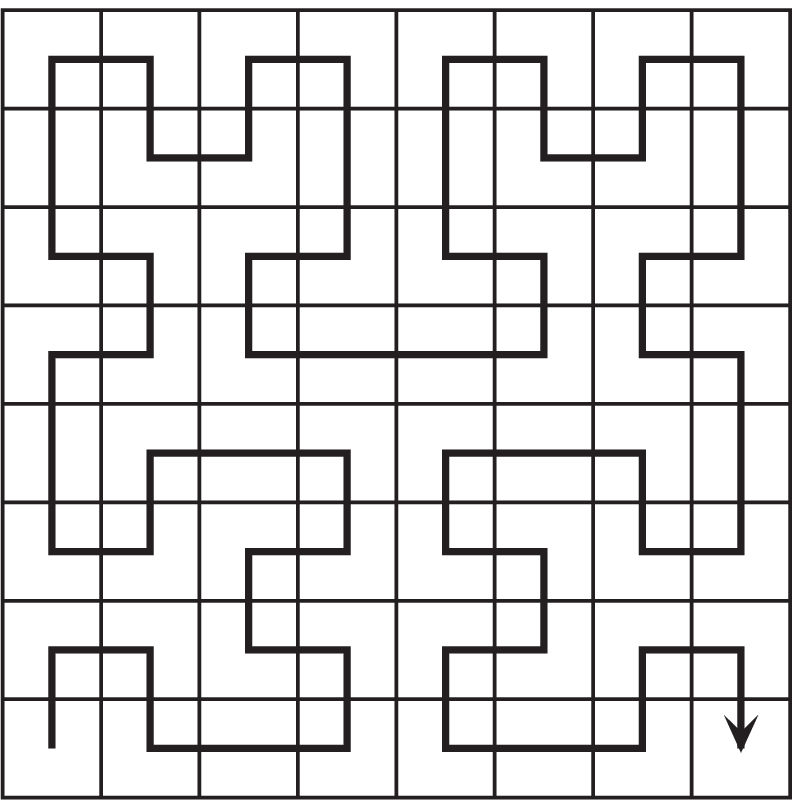}
    \caption{Three steps in the construction of the Hilbert curve.
            \index{Hilbert curve} }
    \label{fig:hil}
\end{center}
\end{figure}

An important aspect of any space-filling curve is its locality property, i.e. points that are close to each other in $[0,1]$ should tend to be close after mapping into the higher dimensional space. But how close or apart will they be?
To this end, let us consider the Hilbert curve. In \cite{zumbusch03} it was shown that the $d$-dimensional Hilbert curve $s(x)$ is H\"older continuous, but nowhere differentiable and that we have, for any $x,y \in [0,1]$, the H\"older condition
	\[	 |s(x)-s(y)|_2 \leq c_s |x-y|^{1/d} \quad \quad \mbox{ with } c_s = 2 \sqrt{d+3}. 	\]
The exponent $1/d$ of the H\"older condition is optimal for space-filling curves and cannot be improved. It also holds for most of the other curves (Peano, Sierpi\'nski, Morton,...), albeit with associated curve- and dimension-dependent constants.
It does not hold for the Lebesgue curve, since this curve linearly interpolates outside its defining Cantor set.
The Lebesgue curve is thus non-injective, almost everywhere differentiable and its H\"older exponent is merely $1/(d \log_2 3)$.

The two-dimensional recursive constructions can be generalized to higher space dimensions $d$, i.e. to curves  $s:[0,1]\to[0,1]^d$.
There exist codes for certain space-filling curves, and especially for the Hilbert curve. A recursive pseudo-code for $d$ dimensions is given in \cite{zumbusch03}. 
The approach in \cite{Butz1971} was further developed 
in \cite{Thomas1992, Moore2000, Lawder2000a} and \cite{Skilling2004}. 
Meanwhile, various code versions can be found on github.com \cite{chernoch, adishavit, pypi}. 
Our implementation of the $d$-dimensional Hilbert curve is based on \cite{Skilling2004}.

Space-filling curves were originally developed and studied for purely mathematical reasons. 
But since the 1980s they were successfully applied in numerics and computer science. 
This is due to the fact that in one dimension we have a {\em total ordering} for points whereas  
in higher dimensions there is only a partial ordering. 
Space-filling curves in principle allow to retract points from high dimensional space to  
one-dimensional space and to use the total order there. 
Moreover the implicit reduction of the dimension from $d$ to one is helpful in many situations. 
Since space-filling curves preserve locality and neighbor relations to some extent, partitions that contain neighboring points in one-dimensional space will (more or less) correspond to a local point set in the image space as well and it can be expected that a parallelization based on space-filling curves is able to satisfy the two main requirements for partitions, namely uniform load distribution and compact partitions.
Thus the space-filling approach is nowadays frequently used for partitioning and for static and adaptive mesh refinement, for the parallelization and dynamic load balancing of tree-based multilevel PDE-solvers and of tree-based fast integral transforms with applications in, e.g.,  molecular dynamics and computational astrophysics, compare \cite{GriebelKnapekZumbusch2007}. There, the underlying partitioning problem is NP-hard and can only be efficiently dealt with in an 
approximate heuristic manner. To this end, since space-filling curves provide a sequential order on a multidimensional computational domain and on the corresponding grid cells or particles used in simulations, the partitioning of the resulting list into parts of equal size is an easy task. 
A further advantage of the space-filling curve heuristic in this context is its simplicity: Parallel dynamic partitioning basically boils down to parallel bucket or radix sort, for further details see e.g. \cite{GriebelKnapekZumbusch2007,zumbusch03} and the references cited therein. 
Meanwhile, space-filling curves are also successfully applied in other areas of scientific computing where, besides parallelization,  cache-efficient implementation plays an increasingly important role. There exist now
cache-aware and cache-oblivious methods for matrix-matrix multiplication, for the handling of sparse matrices or for the stack\&stream-based traversal of space-trees using space-filling curves. For further details see, e.g., \cite{Bader2013} and the references given there.

\subsection{Our algorithm}\label{ssect}
Now we discuss the main features of our domain decomposition method for $d$-dimensional elliptic problems $\mathcal{L} u=f$.
For reasons of simplicity, we consider here the unit domain $\Omega=[0,1]^d$ and employ Dirichlet boundary conditions on $\delta \Omega$. 
The discretization is done with a uniform but in general anisotropic mesh size $h=(h_1,\ldots,h_d)$, where $h_j=2^{-l_j}$ with multivariate level parameter $l=(l_1,\ldots,l_d)$, which gives the global mesh $\Omega_h$. The number of interior grid points, and thus  the number of degrees of freedom is then 
\begin{equation}\label{dof}
	N:=\prod_{j=1}^d (2^{l_j}-1).
\end{equation}
The reason for considering the general anisotropic situation is the following: In a forthcoming paper, we will employ our specific domain decomposition method as a parallel and fault-tolerant inner solver within the combination method \cite{griebel92CombiTechnique, Griebel.Harbrecht:2014} for the sparse grid discretization \cite{Z, BG, nutshell} of high-dimensional elliptic PDEs. There, problems with in general anisotropic discretizations with the level parameters 
$$l=(l_1,\ldots,l_d)\in \mathbb{N}^d, \mbox{ where } |l|_1:= l_1 +\ldots +l_d = L+ (d-1)-i, i=0,\ldots, d-1,, l_j >0.$$
are to be solved.
The resulting solutions $u_l(x)$, $x=(x_1,\ldots,x_d)$, are then to be linearly combined as
\begin{equation}\label{eq:combi}
	u({ x}) \approx u^{(c)}_{L}({ x}):=
	\sum_{i=0}^{d-1} (-1)^i \left( \begin{array}{c} {d-1}\\{i}\end{array}\right) \sum_{|{l}|_1=L+(d-1)-i} u_{l}({x}).
\end{equation}
Figure~\ref{fig:combi_scheme} illustrates the construction of the combination method in the two-dimensional case with $ L=3$. Then five anisotropic coarse full grids 
$\Omega_{l}$ are generated on which we solve our problem. Once the problem is solved on each of the grids, we combine these solutions
with the appropriate weights (here either +1 or -1). This combination yields a sparse grid approximation (on the grid 
$\Omega_{L}^{(c)}$), which is our approximation to the problem on the (unfeasibly large) full grid $\Omega_{L}$. 
\begin{figure}[h]
	\begin{center}
		\includegraphics[width=0.99\textwidth]{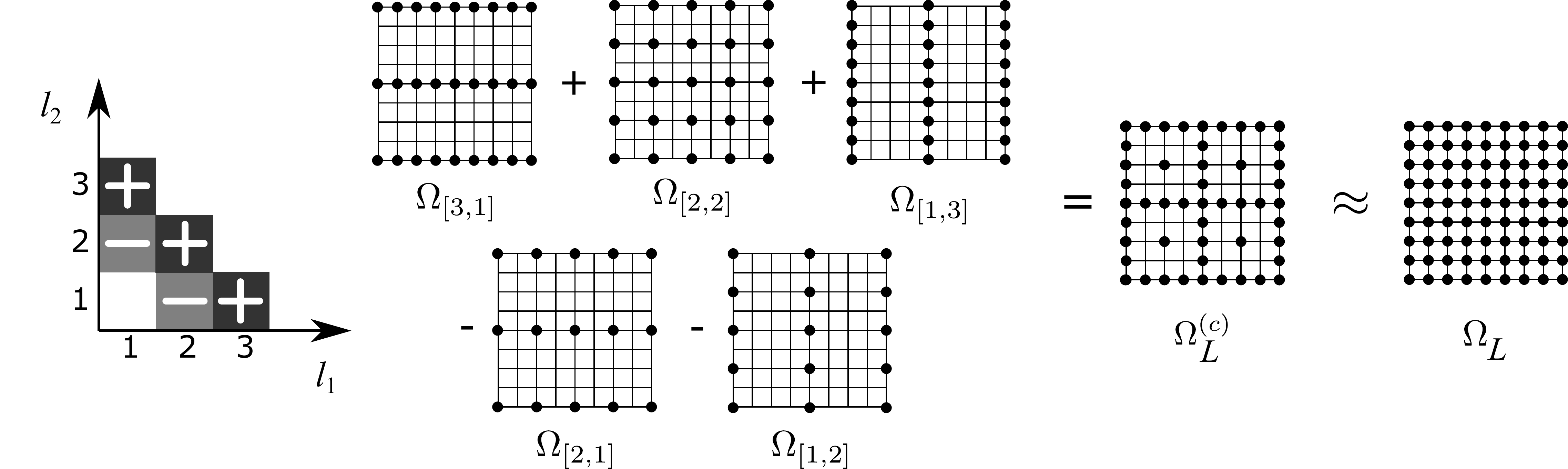}
	\end{center}
	\vspace{-1em}
	\caption{The combination method, two-dimensional case with $ L=3$.}
	\label{fig:combi_scheme}
\end{figure}

This results in a sparse grid approximation to the high-dimensional original problem thus breaking the curse of dimension of a conventional full grid discretization, provided that a certain bounded mixed regularity is present. For details see \cite{BG} and the references cited therein. We were able to show that the direct hierarchical sparse grid approach and the combination method indeed possess the same order of convergence, see also \cite{griebel92CombiTechnique,Bungartz.Griebel.Roschke.ea:1994*1,Griebel.Harbrecht:2014}. The combination method allows to reuse existing codes to a large extent. 
 Moreover the various discrete subproblems can be treated independently of each other \cite{Griebel:1992*2}, which introduces a second level of parallelization beyond the parallelism within each subproblem solver due to the domain decomposition. Furthermore, by means of a fault-tolerant domain decomposition method, also a fault-tolerant parallel combination  method for sparse grid problems is automatically induced. Thus it is necessary to be able to deal with the various anisotropic problems arising in the combination method in a simple and efficient manner with {\em one} single domain decomposition code, which runs straightforwardly and automatically for all these different problems and will not need tedious modifications and adaptions by hand.    

Now, for the discretization of $\mathcal{L}u=f$ on the grid $\Omega_h$, we employ the simple finite difference method (or the usual finite element method with piecewise $d$-linear basis functions) on $\Omega_h$ which results in the system of linear equations $A x = b$ with sparse system matrix $A \in \mathbb{R}^{N \times N}$ and right hand side vector $b \in \mathbb{R}^N$. 

Next, we consider the case of $P$ subdomains.  To generate a partition of $P$ overlapping subdomains $\{\Omega_i\}_{i=1}^P$ of equal size, we employ the space-filling curve approach and in principle just map our $d$-dimensional (interior) grid points $x_k \in \Omega_h$ by means of the inverse, discrete space-filling curve $s_n^{-1}$ with sufficiently large $n$ to the one-dimensional unit interval $[0,1]$. Then we simply partition the one-dimensional, totally ordered sequence of $N$ points into a consecutive one-dimensional set of disjoint subsets of approximate size $ N/P$ each. To this end, we first determine the remainder $r:=N- P \lfloor N/P\rfloor$. This gives us the number $r$ of subdomains which have to possess $\lfloor N/P\rfloor +1$ grid points, whereas the remaining $P-r$ subdomains possess just $\lfloor N/P\rfloor $ grid points. Thus, with 
\begin{equation}\label{sloc}
	\tilde N_i:=\lfloor N/P\rfloor +1, i=1,\ldots, r \mbox{ and } \tilde N_i=\lfloor N/P\rfloor, i=r+1, \ldots,P,
\end{equation}
we assign the first $\tilde N_1$ points to the set $\tilde \Omega_1$, the second $\tilde N_2$ points to the set $\tilde \Omega_2$, and so on. Since the $\tilde N_i$ differ at most by one, we obtain a perfectly balanced partition $\{\tilde \Omega_i\}_{i=1}^P$. The basic partitioning approach by means of the Hilbert curve is shown for the two-dimensional case in Figures \ref{sfcmap} and \ref{sfcmap2}. Note here that the resolution of the discrete isotropic space-filling curve is chosen as the one which belongs to the largest value $\max_{j=1,\ldots,d} l_j$ of the entries of the level parameter $l=(l_1,\ldots,l_d)$, i.e. to the finest resolution in case of an anisotropic grid. 
\begin{figure}[htb]
\centering
\begin{subfigure}{.47\textwidth}
  \centering
  \includegraphics[width=0.5\linewidth,height=0.5\linewidth]{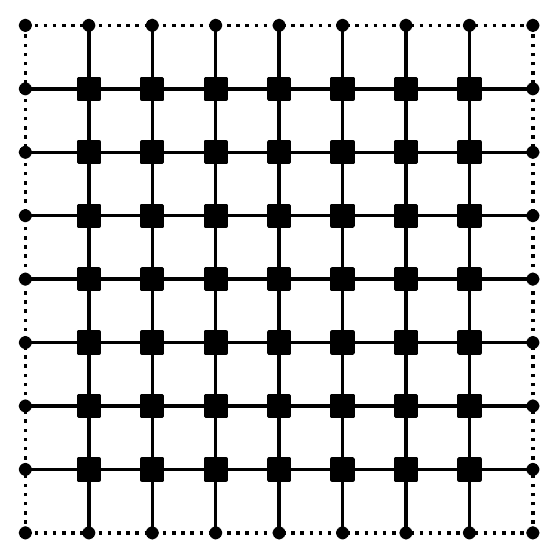}
\end{subfigure}%
{$\rightarrow{}$}%
\begin{subfigure}{.47\textwidth}
  \centering
  \includegraphics[width=0.5\linewidth,height=0.5\linewidth]{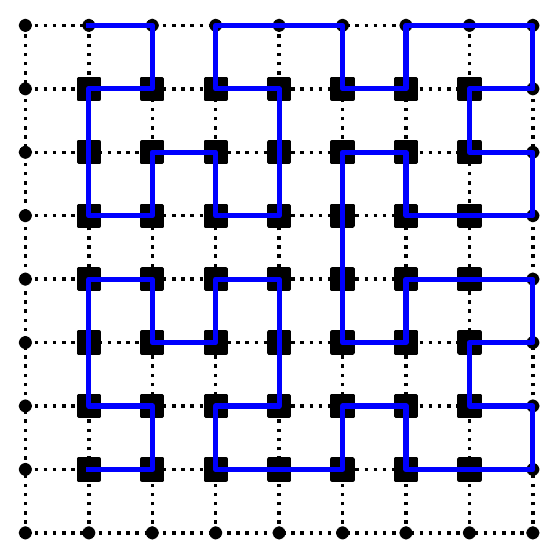}\\
\end{subfigure}
\begin{subfigure}{.47\textwidth}
  \hfill
\end{subfigure}%
\begin{subfigure}{.47\textwidth}
  \centering
\hspace{0.55cm}
  {$\downarrow{}$}%
\end{subfigure}
\vspace{-1.7cm}
\begin{subfigure}{.47\textwidth}
  \centering
  \includegraphics[width=0.5\linewidth,height=0.5\linewidth]{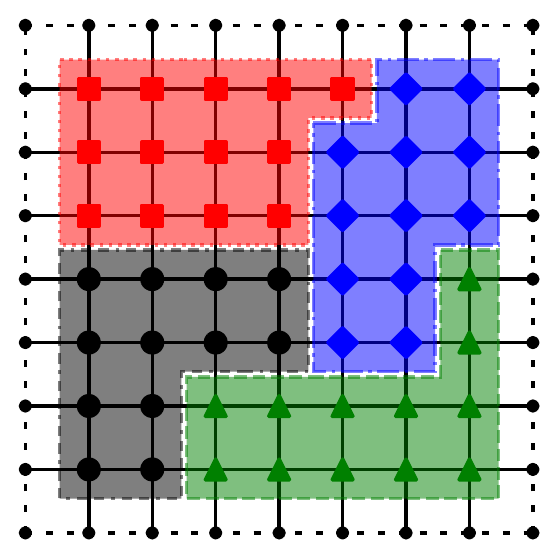}
\end{subfigure}%
{$\leftarrow{}$}%
\begin{subfigure}{.47\textwidth}
  \centering
  \includegraphics[width=0.6\linewidth,height=0.6\linewidth]{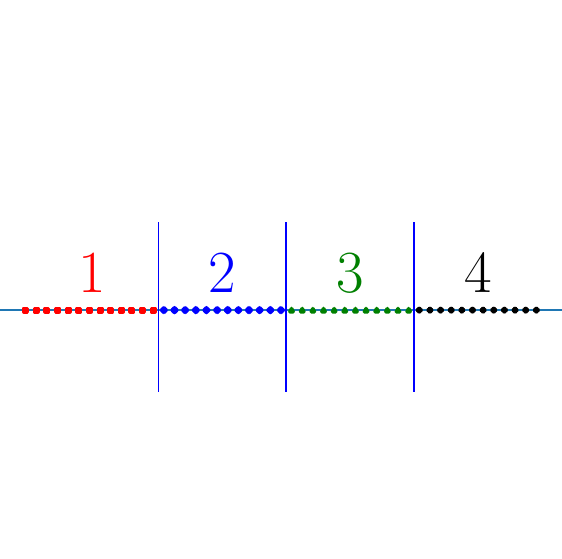}
\end{subfigure}
\vspace{1.7cm}
	\caption{Decomposition of an isotropic grid with ${l}=(3,3)$ by the Hilbert curve approach.}
\label{sfcmap}
\end{figure}
\begin{figure}[htb]
\centering
\begin{subfigure}{.47\textwidth}
  \centering
  \includegraphics[width=0.5\linewidth,height=0.5\linewidth]{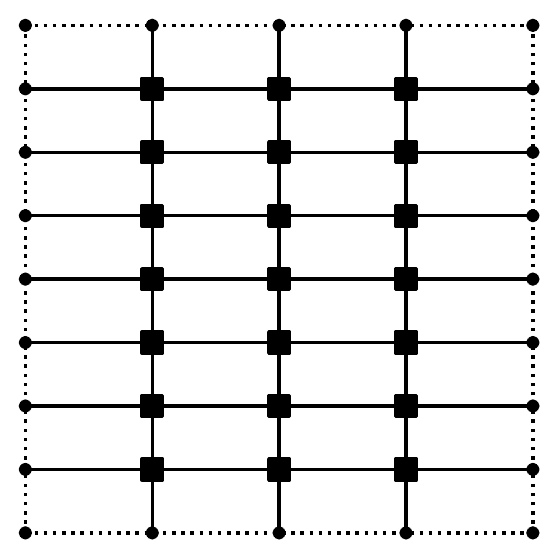}
\end{subfigure}%
{$\rightarrow{}$}%
\begin{subfigure}{.47\textwidth}
  \centering
  \includegraphics[width=0.5\linewidth,height=0.5\linewidth]{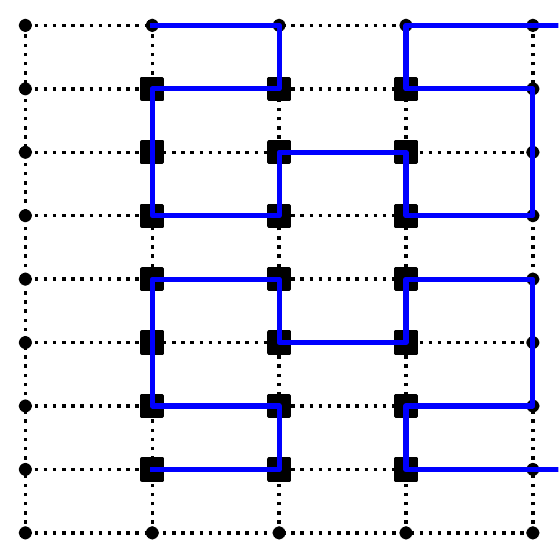}\\
\end{subfigure}
\begin{subfigure}{.47\textwidth}
  \hfill
\end{subfigure}%
\begin{subfigure}{.47\textwidth}
  \centering
\hspace{0.55cm}
  {$\downarrow{}$}%
\end{subfigure}
\vspace{-1.7cm}
\begin{subfigure}{.47\textwidth}
  \centering
  \includegraphics[width=0.5\linewidth,height=0.5\linewidth]{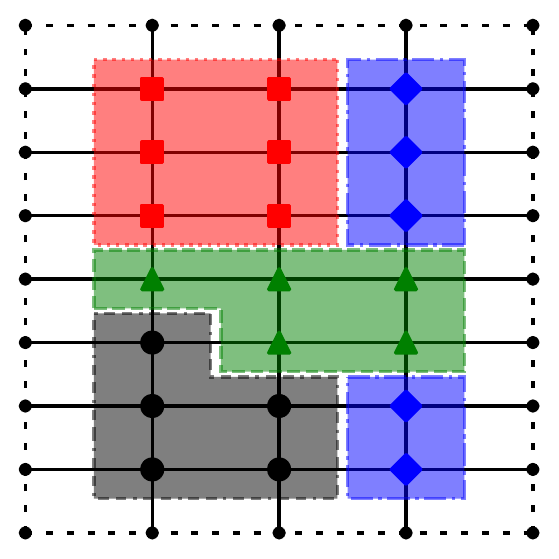}
\end{subfigure}%
{$\leftarrow{}$}%
\begin{subfigure}{.47\textwidth}
  \centering
  \includegraphics[width=0.6\linewidth,height=0.6\linewidth]{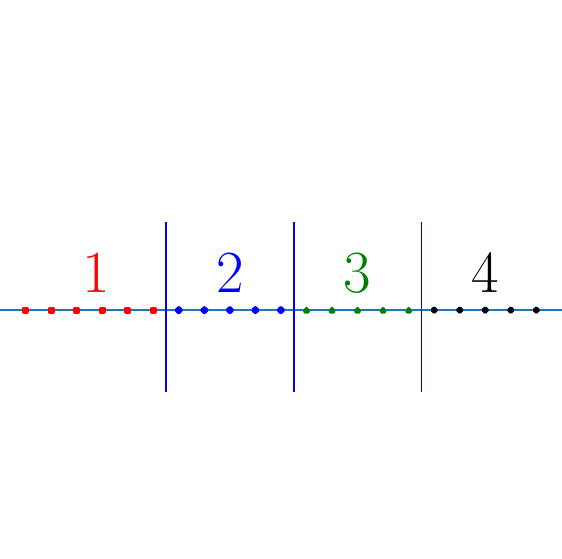}
\end{subfigure}
\vspace{1.7cm}
\caption{Decomposition of an anisotropic grid with ${l}=(2,3)$ by the Hilbert curve approach.}
\label{sfcmap2}
\end{figure}

In practice, we do not use the grid points $x_k=(x_{k_1},\ldots,x_{k_d}),\in \mathbb{R}^d$ with $ x_{k_j}=k_j 2^{-l_j}$ in the partitioning process, but merely their indices $k:=(k_1,\ldots, k_d) \in \mathbb{N}^d$, $k_j=1,\ldots 2^{l_j}-1$, which we store in an $N$-sized vector, e.g. in lexicographical order. The space-filling curve mapping then boils down to a simple {\em sorting} of the entries of this vector with respect to the order relation of the indices induced by the space-filling curve mapping. To this end, we just need the one-dimensional relation $<$ of the space-filling curve ordering of the mapped indices.
This is realized by means of a comparison function {\em cmp}$((k_1, ..., k_d), (k'_1, ..., k'_d))$ for any two indices $(k_1, ..., k_d)$ and $(k'_1, ..., k'_d)$,
which returns true if the index $(k_1, ..., k_d)$ is situated before the index $(k'_1, ..., k'_d)$ on the space-filling curve and which returns false in the else clause. For the case of the Hilbert curve, the implementation follows mainly \cite{Butz1971, Moore2000, spencer20, adishavit}. Our variant is based on bit-operations on the indices  $(k_1, ..., k_d)$ and $(k'_1, ..., k'_d)$ only and is non-recursive. It corresponds to a $2^d$-tree traversal with one rotation per depth of the tree on the bit level and a bit comparison. This avoids first explicitly calculating and then comparing the associated keys. Note here that, if we would first compute the full two keys, we would have to completely traverse the tree down to the finest scale of the discrete Hilbert curve and could only then compare the associated numbers. In our direct comparison we however descend iteratively down the tree and stop the traversal as soon as we detect on the current level that the two considered indices belong to different orthants. This is much faster. Moreover it avoids a problem which might occur for strongly anisotropic grids with the isotropic Hilbert curve:  For example, in the case of a grid with the multivariate level parameter $l=(L,1,\ldots,1)$ we still would have to deal with the universe of $2^{dL}$ possible indices and keys due to the {\em isotropy} of the $d$-dimensional Hilbert curve, but we employ for our anisotropic grid only $2^L$ indices altogether. This universe of keys for the Hilbert curve becomes, for larger $d$, with rising $L$ quickly too large for any conventional data type of the associated keys. Furthermore the keys would contain large gaps and voids in this universe, since only $2^L$ keys are present anyway. Our approach of using just a comparison relation without explicitly computing the two keys for the two indices still allows sorting of the indices. A key is then just given as the position in the sorted index vector of length $N=2^L$, i.e. we only need $2^L$ keys in our position-universe which now contains no voids or gaps at all. 
The sorting is done by introsort which has an average and worst-case complexity of $ O(N \log(N))$. We store the full vector on each processor redundantly and perform the sorting redundantly as well.
We will mainly consider the Hilbert curve ordering in our experiments. 

In a next step, we enlarge the subdomains $\tilde \Omega_j$, i.e. the corresponding subsets of grid point indices, in a specific way to create overlap.
This is not done as in conventional, geometry-based overlapping domain decomposition methods by adding a $d$-dimensional mesh-stripe with diameter $\delta$ of grid cells at the boundary of the $d$-dimensional subdomains that are associated to the sets $\tilde \Omega_i$. Instead, we deliberately stick to the one-dimensional setting which is induced by the space-filling curve: We choose an overlap parameter ${\gamma \in \mathbb{R}, \gamma >0,}$ and enlarge the index sets $\tilde \Omega_i$ as
\begin{equation}\label{enlarge}
	\Omega_i := \tilde \Omega_i \cup \bigcup_{k=1}^{\lfloor \gamma \rfloor} \left(\tilde \Omega_{i-k} \cup \tilde \Omega_{i+k} \right) \cup \tilde \Omega_{i-\lfloor \gamma \rfloor-1}^{\eta,+} \cup \tilde \Omega_{i+\lfloor \gamma \rfloor+1}^{\eta,-}.
\end{equation}
Here, with $\eta:=\gamma-\lfloor \gamma\rfloor$, the set $\tilde \Omega_{k}^{\eta,+}$ is the subset of $\tilde \Omega_{k}$ which contains its last $\lceil \eta N_{k}\rceil$ indices with respect to the space-filling curve ordering, while the set $\tilde \Omega_{k}^{\eta,-}$ is the subset of $\tilde \Omega_{k}$ which contains its first $\lfloor \eta N_{k}\rfloor$ indices.
For example, for $\gamma=1$, we add to $\tilde \Omega_i$ exactly the two neighboring index sets $\tilde \Omega_{i+1}$ and $\tilde \Omega_{i-1}$, for $\gamma=2$,  we add the four sets $\tilde \Omega_{i+1},\tilde \Omega_{i+2} $ and  $\tilde \Omega_{i-1},\tilde\Omega_{i-2}  $. 
For $\gamma=0.5$ we would add those halves of the indices of $\tilde \Omega_{i+1}$ and $\tilde \Omega_{i-1}$ which are the ones closer to $\tilde \Omega_i$, etc. 
Moreover, to avoid any special treatment for the first and last few  $\tilde \Omega_i, i=1,2,..  $ and $i=P,P-1,..$, we cyclically close the enumeration of the subsets, i.e. the left neighbor of $\tilde \Omega_1$ is set as $\tilde \Omega_P$ and the right neighbor of $\tilde \Omega_P$ is set as $\tilde \Omega_1$. Note that, besides $\gamma$, also the specific values of the $\tilde N_i$ enter here. Examples of the enlargement of the index sets from $\tilde \Omega_i$ to $\Omega_i$ with $\gamma=0.25$ are given in Figures \ref{sfccmap_overlap} and \ref{sfccmap_overlap2}. Here we show the induced domains $\Omega_1$ and $\Omega_3$ only.
\begin{figure}[htb]
\centering
\begin{subfigure}{.47\textwidth}
  \centering
  \includegraphics[width=0.5\linewidth,height=0.5\linewidth]{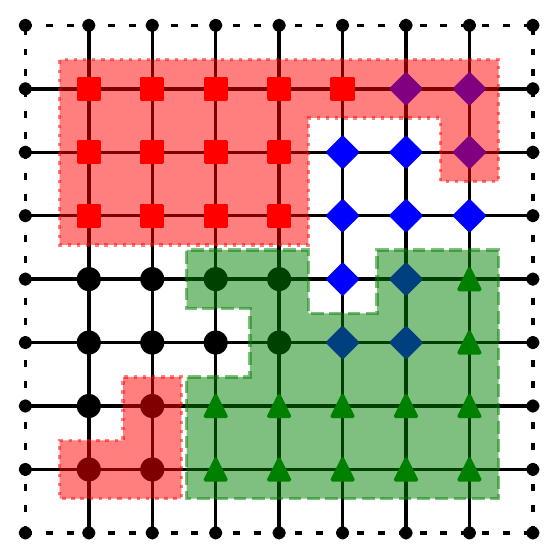}
\end{subfigure}%
{$\leftarrow{}$}%
\begin{subfigure}{.47\textwidth}
  \centering
	\includegraphics[width=0.6\linewidth,height=0.6\linewidth]{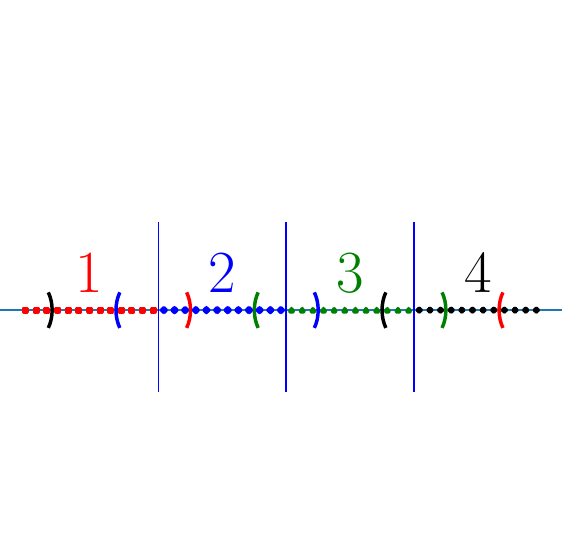}
\end{subfigure}
\caption{Enlargement of the subdomains with the Hilbert space-filling curve and associated overlapping domain decomposition, isotropic grid with ${l}=(3,3)$.}
	\label{sfccmap_overlap}
\end{figure}
\begin{figure}[htb]
\centering
\begin{subfigure}{.47\textwidth}
  \centering
  \includegraphics[width=0.5\linewidth,height=0.5\linewidth]{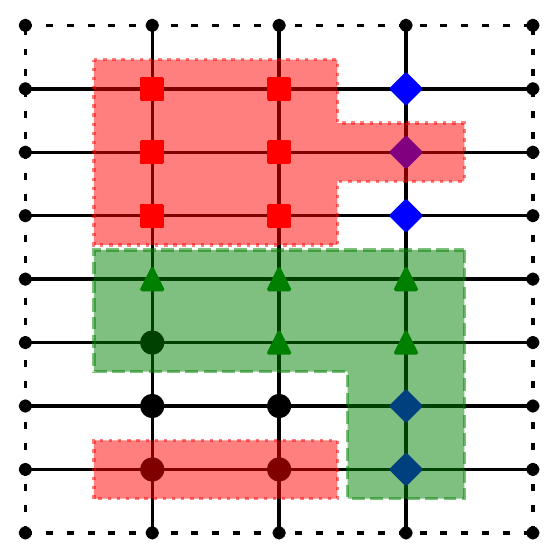}
\end{subfigure}%
{$\leftarrow{}$}%
\begin{subfigure}{.47\textwidth}
  \centering
	\includegraphics[width=0.6\linewidth,height=0.6\linewidth]{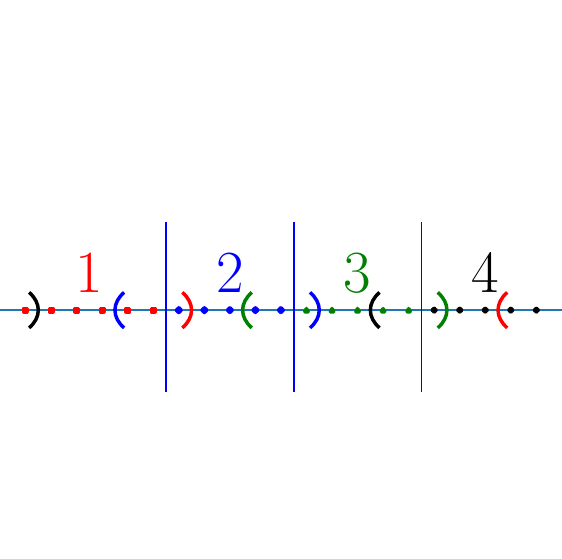}
\end{subfigure}
\caption{Enlargement of the subdomains with the Hilbert space-filling curve and associated overlapping domain decomposition,  anisotropic grid with ${l}=(2,3)$.}
	\label{sfccmap_overlap2}
\end{figure}

This way, an overlapping partition $\{\Omega_i\}_{i=1}^P$ is finally constructed. Note at this point that, depending on $N,P, \gamma$ and the respective space-filling curve type, each subdomain of the associated decomposition is not necessarily fully connected, i.e. there can exist some subdomains with point distributions that geometrically indicate some further separation of the subdomain, see e.g. $\Omega_1 $ in Figure \ref{sfccmap_overlap}. But since we merely deal with index sets and not with geometric subdomains, this causes no practical issues. 
Note furthermore that there is not necessarily always a complete overlap but sometimes just a partial overlap between two adjacent subdomains being created by our space-filling curve approach.
But this also causes no practical issues. In contrast to many other domain decomposition methods, where a goal is to make the overlap as small as possible, our approach rather aims at larger overlaps along the space-filling curve, which later gives the sufficient redundancy of stored data that will be needed for fault tolerance.
 
Finally, we set up our coarse space problem. To this end, the size $N_0$  of the coarse problem is given via the number $P$ of subdomains and the number of degrees of freedom $n_i$ considered on the coarse scale per subdomain, i.e.
$N_0 := \sum_{i=1}^P n_i.$
If we let all $n_i$ be equal to a fixed integer $q\in \{1,\ldots, \lfloor \frac N P \rfloor\}$, i.e. $n_i=q$, then $N_0= q \cdot P$. The mapping from the fine grid to the coarse space is now given by means of the restriction matrix $R_0 \in  \mathbb{R}^{N_0 \times N}$ and its entries. Again, we avoid any geometric information here, i.e. we do not follow the conventional approach of a coarse grid with associated coarse piecewise linear basis functions. Instead we derive the coarse scale problem in a purely {\em algebraic} way.
For that, we resort to the non-overlapping partition $\{\tilde \Omega_i\}_{i=1}^P$ and assign the values of zero and one to the entries of the restriction matrix as follows:
Let $q$ be the constant number of coarse level degrees of freedom per subdomain. With $\tilde N_i=\lfloor N/P\rfloor+1$ if $i \leq (N \mod P) $ and $\tilde N_i=\lfloor N/P\rfloor$ otherwise, which denotes the size of $\tilde \Omega_i$, we have the index sets $\tilde \Omega_i =\left\{ \sum_{j = 1}^{i-1} \tilde N_j + 1, \ldots, \sum_{j=1}^i \tilde N_j\right\} $. Now let $\tilde \Omega_{i, m}$ be the $m$-th subset of $\tilde \Omega_i$ with respect to the size $q$ in the space-filling curve ordering, 
i.e.
$$\tilde \Omega_{i, m} = \left\{\sum_{j=1}^{i-1} \tilde N_j + \sum_{n=1}^{m-1} \tilde N_{i, n} + 1, ..., \sum_{j=1}^{i-1} \tilde N_j + \sum_{n=1}^m \tilde N_{i, n}\right\},
$$
with associated size $\tilde N_{i,m}$ for which, with $q$ coarse points per domain, we have
\begin{equation*}
	\tilde N_{i, m} = \left( 
	\begin{array}{ll}
		\lfloor \tilde N_i/q \rfloor + 1  & \mbox{if } m \leq (\tilde N_i \mod q) \\ 
		\lfloor \tilde N_i/q \rfloor 	 & \mbox{otherwise}.
	\end{array}
\right.
\end{equation*}
Then,  for $i=1,...,P, m=1,...,q, j=1,...,N$, we have
\begin{equation}\label{restr}
	(R_{0})_{((i-1)q+m,j)}= \left( 
	\begin{array}{ll}
		1 & \mbox{if } j \in \tilde \Omega_{i,m} \\ 	0 & \mbox{otherwise}.
	\end{array}
\right.
\end{equation}
This way a basis is implicitly generated on the coarse scale: Each basis function is constant in the part of each subdomain of the non-overlapping partition which belongs to the
$\tilde \Omega_{i,m}$, where $q$ piecewise constant basis functions with support on $\tilde \Omega_{i,m}$ are associated to each $\tilde \Omega_i$.

The coarse scale problem is then set up via the Galerkin approach as 
\begin{equation}\label{csm}
A_0 = R_0 A R^T_0.
\end{equation}
Here we follow the Bank-Holst technique \cite{BH2003} and store a redundant copy of it on each of the $P$ processors together with the respective subdomain problem.  This way, the coarse problem is formally avoided. Moreover the coarse problem is redundantly solved on each processor. It interacts with the respective subproblem solver in an additive, parallel way, i.e. we solve the global coarse scale problem and the local fine subdomain problem independently of each other on the processor.

Finally, we have to deal with the overcounting of unknowns due to the overlapping of fine grid subdomains. To this end, we resort to the partition of unity  on the fine level
\begin{equation}\label{pum} 
	I = \sum_{i=1}^P R_i^T D_i R_i 
\end{equation}
with properly chosen diagonal matrices $D_i \in \mathbb{R}^{N_i \times N_i}$. 
This leads to the two-level domain decomposition operator
\begin{equation}\label{addSW}
	C_{(2),D}^{-1} :=  R_0^T A_0^{-1} R_0 + C_{(1),D}^{-1} = \sum_{i=0}^P  R_i^T D_i A_i^{-1} R_i, \quad D_0=I,
\end{equation}
with weighted one-level operator 
\begin{equation}\label{w1levprec}
C_{(1),D}^{-1} := \sum_{i=1}^P  R_i^T D_i A_i^{-1} R_i,
\end{equation}
and to the iteration in the $k$-th step
$$
x^{k+1} =x^k +C_{(2),D}^{-1}(b-Ax^k).
$$
Up to the coarse scale part, this resembles just the restricted additive Schwarz method of \cite{caisarkis99}, which is basically a weighted overlapping block-Jacobi/Richardson method. 
With the choice $D_0=I$, the associated weighted balanced variant is then
$$	C_{(2),D,bal}^{-1}:= G^T C_{(1),D}^{-1} G + F,
$$
with $G$ and $F$ from (\ref{balan}). The modifications of $F, G$ and thus $C_{(2),D,bal}^{-1}$ for general $D_0$ are obvious.

There are different choices of the $D_i$'s since the above condition (\ref{pum}) does not have a unique solution.
A natural choice is here
\begin{equation}\label{Di}
	(D_i)_{{j,j}} = 1/ |\left\{\Omega_{i'}, i'=1,\ldots P : j \in \Omega_{i'}\right\} |,
\end{equation}
which locally takes for each index $j\in \Omega_i$ the number of domains that overlap this index into account. Note that, for such general diagonal matrices $D_i$,  the associated 
$C_{(2),D}$ is not a symmetric operator. 
If however the $D_i$ are chosen as $\omega_i I_i$ with scalar positive weights $\omega_i$, then, on the one hand, the partition of unity property is lost, but, on the other hand, symmetry is regained and we have a weighted, two-level overlapping additive Schwarz method, which can also used as a preconditioner for the conjugate gradient method. To this end, a sensible choice is
\begin{equation}\label{omegai}
	\omega_i =\max_j (D_i)_{(j,j)}
\end{equation}
with $D_i$ from (\ref{Di}).

At this point a closer inspection of the relation of the overlap parameter $\gamma$ to the entries of $D_i$ or $\omega_i$, respectively, is in order. Of course we may choose any value for $\gamma$ and we may select quite freely the $D_i$ or $\omega_i$. However, if we want to achieve fault tolerance and thus employ a larger overlap of the subdomains, which are created as described above via our space-filling curve construction to allow for proper redundant storage and for data recovery, it is sensible to restrict ourselves 
to values of $\gamma$ that are integer multiples of $0.5$. In this case, every fine grid point is overlapped by exactly $2\gamma + 1$ subdomains, whereas, if $\gamma$ is not an integer multiple of 0.5, the number of subdomains that overlap a point can indeed be different for different points of the same subdomain. Additionally, integer multiples of 0.5 for $\gamma$ are the natural redundancy thresholds of our fault-tolerant space-filling curve algorithm. In particular for $\frac 1 2 n\leq\gamma < \frac 1 2 (n+1), n \in \mathbb{N}$, our fault-tolerant algorithm can recover from faults occurring for at most $n$ neighboring processors in the same iteration. Thus, with these considerations, overlap parameter values of the form $\gamma=\frac 1 2 n, n \in  \mathbb{N},$ are the ones that are most relevant for proper redundant storage, for data recovery and thus for fault tolerance in practice.
Additionally, such a specific choice of $\gamma$ has a direct consequence on the resulting $\omega_i$ and $D_i$. We have the following  Lemma.

\begin{lemma}\label{lemweig}
Let $d\ge1$ be arbitrary and let $\gamma = \frac{1}{2}n$, where $n\in\mathbb{N}, n\leq P - 1$. Then, with $c\coloneqq 2\gamma + 1$, there holds
\begin{equation*}
 D_i = \omega_i I = \frac{1}{c} I
\end{equation*}
for all $i=1,\ldots, P$ and any type of space filling curve employed.
\end{lemma}
\begin{proof}
With the notation from above we need to show that we have $$|\left\{\Omega_{i}, i=1,\ldots P : j \in \Omega_{i}\right\}| =c=2\gamma + 1$$ for all $j=1,\ldots,N$, if $\gamma = \frac{1}{2}n$. 
To this end, let $j$ be arbitrary and $m\in\{1,\ldots,P\}$ such that $j\in\tilde\Omega_m$. Since the $\tilde\Omega_i$ are disjoint, there is exactly one such $m$. 
Now recall (\ref{enlarge}), i.e.
\begin{equation}\label{enlarge2}
	\Omega_i := \tilde \Omega_i \cup \bigcup_{k=1}^{\lfloor \gamma \rfloor} \left(\tilde \Omega_{i-k} \cup \tilde \Omega_{i+k} \right) \cup \tilde \Omega_{i-\lfloor \gamma \rfloor-1}^{\eta,+} \cup \tilde \Omega_{i+\lfloor \gamma \rfloor+1}^{\eta,-},
\end{equation}
	where $\eta = \gamma - \lfloor \gamma \rfloor$. 

	First consider the case where $n$ is even.
This implies that $\gamma$ is an integer and $\lfloor\gamma\rfloor=\gamma$. We obtain $\eta = 0$ and therefore $\tilde \Omega_{i-\lfloor \gamma \rfloor-1}^{\eta,+} = \tilde \Omega_{i+\lfloor \gamma \rfloor+1}^{\eta,-} = \emptyset$ for all $i$. Furthermore, due to the assumption $n\leq P - 1$, we have $\gamma \leq \frac{P-1}{2}$ and, since the $\tilde\Omega_i$ are disjoint by construction, the sets of the union in (\ref{enlarge2}) are disjoint. This implies that $j\in\Omega_i$ if and only if $j\in \tilde\Omega_i$ or $j \in \tilde\Omega_{i-k}$ or $j \in \tilde\Omega_{i-k}$ for some $k=1,\ldots,\gamma$. Equivalently, $j\in\tilde\Omega_i$ if and only if $i=m, i-k=m$ or $i+k=m$ for $k=1,\ldots,\gamma$. Hence there are exactly $1 + \gamma + \gamma = 2\gamma + 1$ values for $i$, namely $m$ and $m+1, \ldots, m+\gamma$ and $m-1, \ldots, m-\gamma$. Therefore, we obtain
\begin{equation*}
|\left\{\Omega_{i}, i=1,\ldots P : j \in \Omega_{i}\right\}| =2\gamma + 1.
\end{equation*}

	Now consider the case where $n$ is odd. Then $\eta=\frac{1}{2}$ and we have $\tilde \Omega_i^{\eta,-}\cup \Omega_i^{\eta,+} = \tilde\Omega_i$ and $\tilde\Omega_i^{\eta,-}\cap\tilde\Omega_i^{\eta,+}=\emptyset$ for all $i$, since $\tilde \Omega_i^{\eta,-}$ contains the first $\lfloor \eta N_i\rfloor=\lfloor \frac{N_i}{2}\rfloor$ and $\tilde \Omega_i^{\eta,+}$ contains the last $\lceil \eta N_i\rceil=\lceil \frac{N_i}{2}\rceil$ indices of $\tilde\Omega_i$, respectively. Thus there is exactly one $i$ such that $j\in\left(\tilde \Omega_{i-\lfloor \gamma \rfloor-1}^{\eta,+} \cup \tilde \Omega_{i+\lfloor \gamma \rfloor+1}^{\eta,-}\right)$. Furthermore, since $\gamma \leq \frac{P-1}{2}$, the sets of the union in (\ref{enlarge2}) are again disjoint. Using the same argument as for the previous case where $\gamma$ was an integer, there are now exactly $2\lfloor\gamma\rfloor + 1$ choices for $i$ such that $j\in\left(\tilde \Omega_i \cup \bigcup_{k=1}^{\lfloor \gamma \rfloor} \left(\tilde \Omega_{i-k} \cup \tilde \Omega_{i+k} \right)\right)$. Consequently, we also obtain for the case where $n$ is odd that 
\begin{equation*}
\begin{aligned}
|\left\{\Omega_{i}, i=1,\ldots P : j \in \Omega_{i}\right\}| &=2\lfloor\gamma\rfloor + 1 + 1 = 2(\lfloor\gamma\rfloor + 0.5) + 1\\
	&= 2(\lfloor\gamma\rfloor +\eta) + 1 = 2\gamma + 1. 
\end{aligned}
\end{equation*}
Altogether, this proves the assertion.
\end{proof}

We thus have seen that the general weightings $D_i$ and $\omega_i$ for the fine scale subdomains are the {\em same} and even boil down to a simple constant global scaling with the factor $1/c = 1/(2\gamma + 1) = 1/(n+1) $ uniformly for all subdomains $i=1,\ldots,P$, if we a priorily choose $\gamma = \frac{1}{2}n$, $n\in\mathbb{N}$. Recall that the single weight for the coarse scale problem was set to one in the first place. Thus we regain symmetry of the corresponding operator also for the $D_i$-choice, as long as we employ values for $\gamma$ that are integer multiples of $0.5$. This is a direct consequence of our one-dimensional domain decomposition construction due to the space filling curve approach involving the proper rounding in the definition of the overlapping subdomains $\Omega_i$ and using the value of $\gamma$ not absolutely but {\em relatively} with respect to the size of the non-overlapping subdomains $\tilde \Omega_i$, which are still allowed to be of different size. Moreover this property is independent of the respective type of space-filling curve and merely a consequence of our construction of the overlaps. Such symmetry of the operator can not be obtained so easily for the general $D_i$-weighting within the conventional geometric domain decomposition approach for $d>1$. Note also that the result of Lemma \ref{lemweig} holds analogously for more general geometries of $\Omega$ beyond the $d$-dimensional tensor product domain.
Note furthermore that, for the choice $\gamma = \frac{1}{2}n$, $n\in\mathbb{N}$, the weighted one-level operator (\ref{w1levprec}) becomes just a scaled version of the conventional 
one-level operator (\ref{AS2a}), i.e.
$$
C_{(1),D}^{-1}  =\frac 1 {2 \gamma+1}  C_{(1)}^{-1}.
$$
Consequently, if we choose $D_0=I$, we obtain
$$
C_{(2),D}^{-1} =  R_0^T A_0^{-1} R_0 + C_{(1),D}^{-1} = R_0^T A_0^{-1} R_0 + \frac 1 {2 \gamma+1}  C_{(1)}^{-1}.
$$
and 
$$
C_{(2),D,bal}^{-1}:= G^T C_{(1),D}^{-1} G + F =  \frac 1 {2 \gamma+1} G^T C_{(1)}^{-1} G + F,
$$
which now just resemble a fine-level-rescaled variant of the conventional two-level operator and of its balanced version, respectively.

In any case, we obtain the damped linear two-level Schwarz/Richardson-type iteration as given in Algorithm \ref{Alg:DDMAlgorithm2}, where the setup phase, which is executed just once before the iteration starts, is described in Algorithm \ref{Alg:DDMAlgorithm}. Its convergence property is well known, since it is a special case of the additive Schwarz iteration as studied in \cite{Griebel.Oswald:1995*1}. There, in Theorem 2, it is stated that the damped additive method associated to, e.g., our decomposition (\ref{addSW}) indeed converges in the energy norm for $0 < \xi  <2/\lambda_{\max}$ with convergence rate 
$$\rho =\max\{|1-\xi \lambda_{\min}|,|1-\xi \lambda_{\max}|\},
$$
where $\lambda_{\min}$ and $\lambda_{\max}$ are the smallest and largest eigenvalues of $C_{(2),D}^{-1}A$, provided that $C_{(2),D}^{-1}$ is a symmetric operator, which for example is the case for the choice (\ref{omegai}). Moreover the convergence rate takes the minimum
\begin{equation}\label{redrate}
\rho^* = 1- \frac 2 {1+\kappa} \quad \mbox{for} \quad \xi^* = \frac 2 {\lambda_{\min}+\lambda_{\max}}
\end{equation}
with $\kappa =\lambda_{\max}/\lambda_{\min}$. The proof is exactly the same as for the conventional Richardson iteration. 
To this end, the numbers $\lambda_{\max}$ and $\lambda_{\min}$ need to be explicitly known to have the optimal damping parameter $\xi^*$, which is of course not the case in most practical applications. Then good estimates, especially for $\lambda_{\max}$, are needed to still obtain a convergent iteration scheme, albeit with a somewhat reduced convergence rate. Note at this point that for the general non-symmetric case, i.e. for the general choice (\ref{Di}), this convergence theory does not apply. In practice however, convergence can still be observed.
\begin{algorithm}[hbt!]
\caption{Overlapping two-level additive Schwarz iteration with space-filling curve: Setup phase.}
\label{Alg:DDMAlgorithm}

\begin{algorithmic}[1]
\AlgOn{every processor $i=1,\ldots,P$}
    \State \parbox[t]{\dimexpr\linewidth-\algorithmicindent}{Set input parameters: $d$, $l=(l_1,\ldots,l_d)$, $P$, $\gamma$, $q$, $\xi$, type of space-filling curve.\strut}
    \State \parbox[t]{\dimexpr\linewidth-\algorithmicindent}{Derive $N$ from (\ref{dof}), set $\tilde N_i, i=1,\ldots,P$ as in (\ref{sloc}) and set $N_0=q\cdot P$.\strut}
    \State \parbox[t]{\dimexpr\linewidth-\algorithmicindent}{Compute the index vector $sfc\_index$ of length $N$ from the $d$-dimensional grid point indices $k=(k_1,\ldots,k_d)$, $k_j=1,\ldots,2^{l_j}-1$, $j=1,\ldots d,$ according to the space-filling curve by means of {\em cmp}$((k_1, ..., k_d), (k'_1, ..., k'_d))$  and introsort.\strut}
    \State \parbox[t]{\dimexpr\linewidth-\algorithmicindent}{Derive the disjoint subdomain index sets $\{\tilde \Omega_i\}_{i=1}^P$ by splitting the overall index set into $P$ subsets $\tilde\Omega_i$ of consecutive indices, each of size $\tilde N_i$. This is simply done by storing two integers $\tilde ta_i,\tilde tb_i$, which indicate where the local index sequence of $\tilde \Omega_i$ starts and ends in $sfc\_index$.\strut}
    \State \parbox[t]{\dimexpr\linewidth-\algorithmicindent}{Derive the overlapping subdomain index sets $\{\Omega_i\}_{i=1}^P$ by enlarging the $\tilde\Omega_i$ with $\gamma$ as in (\ref{enlarge}). Again, this is simply done by storing two integers  $ ta_i,tb_i$, which indicate where the local index sequence of $\Omega_i$ starts and ends in $sfc\_index$.\strut}
    \State \parbox[t]{\dimexpr\linewidth-\algorithmicindent}{Set up a map to neighboring grid points that are not in $\Omega_i$, i.e. store their global indices, to later determine the column entries of the stiffness matrix that are situated outside of $\Omega_i$.\strut}
    \State \parbox[t]{\dimexpr\linewidth-\algorithmicindent}{Set the rows of $A$ that belong to $\Omega_i$, i.e. store the rows of $A$ with indices $j \in [ta_i,tb_i]$ in CRS format.\strut}
    \State \parbox[t]{\dimexpr\linewidth-\algorithmicindent}{Initialize the part of the starting iterate $x^0$ and the part of $b$ that belong to $\Omega_i$.\strut}
    \State \parbox[t]{\dimexpr\linewidth-\algorithmicindent}{Derive the rows of the matrix $R_0$ from (\ref{restr}) with indices $j \in [(i-1)q+1, \dots,iq] $ and store them in CRS format.\strut}
    \State \parbox[t]{\dimexpr\linewidth-\algorithmicindent}{Compute the rows of the coarse scale matrix $A_0$ as in (\ref{csm}) that belong to $\Omega_i$, i.e. with the indices $j \in [(i-1)q+1, \dots,iq] $, and store them in CRS format.\strut}
\AlgEndOn
\end{algorithmic}

\end{algorithm}

\begin{algorithm}[hbt!]
\caption{Overlapping two-level additive Schwarz iteration with space-filling curve: Linear iteration.}
\label{Alg:DDMAlgorithm2} 

\begin{algorithmic}[1]
\State Set k=1.
\While{not converged}
\AlgOn{every processor $i=1,\ldots,P$}
\State \parbox[t]{\dimexpr\linewidth-\algorithmicindent}{Compute the part $r_i^k=R_i r^k$ of the residual $r^k=b-Ax^{k}$ that belongs to $\Omega_i$.\strut}
\State Solve the local subproblems
    \begin{equation*}
  A_i d_i^k=r_i^k. 
    \end{equation*}
\State Solve redundantly the coarse scale problem
                \begin{equation*}
                    A_0 d_0^k = R_0 r^k .
                \end{equation*}
\State \parbox[t]{\dimexpr\linewidth-\algorithmicindent}{Compute the part $x_i^{k+1}=R_ix^{k+1}$ of the correction
            \begin{equation*}
                   x^{k+1}= x^k+\sum_{i=0}^P \xi R_i^T D_i d_i^k
            \end{equation*}
            that belongs to $\Omega_i$.\strut}
\AlgEndOn
\EndWhile
\end{algorithmic}

\end{algorithm}

Here the following remarks are in order: 
We store and sort on each processor redundantly the  full index vector $sfc\_index$ of size $N$. 
The sorting could be done in parallel and non-redundantly, but the redundant sorting is very cheap anyway. 
The two numbers $ta_i,tb_i$ need, depending on the overlap factor $\gamma$, to be interpreted properly for the first and last few $i$, since we cyclically close the enumeration of the subsets $\Omega_i$ modulo $P$
to avoid any special treatment there.
The relevant parts of $A$ that belong to $\Omega_i$ are the full rows of $A$ with indices $j \in [ta_i,tb_i]$. They are stored in compressed row storage (CRS) format. But these full rows can also contain entries whose {\em column} index might belong to another processor, i.e. the column indices of such entries of the matrix need to be known to processor $i$ to be able to set up the full rows in the first place. Therefore, we store on each processor a map of the global indices which this processor does not own but which are geometric neighbors tied to an index on this processor due to the non-zero matrix entries of the respective rows. This information is then used to determine the non-zero entries of the complete rows associated to $\Omega_i$ in the next step and can be deleted after that.
It is not necessary to store the relevant parts of $R_i, i=1,\ldots,P$. The matrices $R^T_i,i=1,\ldots,P$, and the the corresponding restriction matrices $R_i, i=1,\ldots, P$, are implicitly available using the local part of $\Omega_i$ of the $sfc\_index$ vector, since $R_i^T$ is just the extension-by-zero map and $R_i$ the corresponding restriction.
Note here that the product of two matrices in CRS format can easily be stored in CRS format as well. 
This facilitates the setup of the coarse scale matrix. In this step, the $q$ rows of $A_0$ that belong to $\Omega_i$, i.e. those with the indices $j \in [(i-1)q+1, \dots,iq] $, are first locally generated. Then, to create the complete $A_0$ redundantly on all processors $i$, an all-to-all communication step is necessary.
We additionally need two types of communication: An overlap-based data exchange between processors $i$ and $i\pm 1, \ldots, i \pm \lceil \gamma \rceil$,
which are neighbors with respect to the space-filling curve ordering, for the update of the $R^T x^{k+1}$, and an $A$-based data exchange between processors 
for the parallel matrix-vector product $Ax^{k}$.
Note finally that we store on each processor the local part of $x$ and $b$, which belongs to $\Omega_i$ and not just to $\tilde \Omega_i$. This avoids one communication step of the overlap data, which would be necessary otherwise. Thus, in contrast to most conventional parallel domain decomposition implementations, we trade communication for (moderate) storage costs.

This linear two-level additive Schwarz iteration can also be used as a preconditioner for the conjugate gradient iteration, which results in a substantial improvement in convergence. In the symmetric case, an error reduction factor of $2(1-2/(1+\sqrt \kappa))$ per iteration step is then obtained in contrast to the reduction factor of $1-2/(1+\kappa)$ from (\ref{redrate}). Moreover a damping parameter is no longer necessary, since the respective two-level additive Schwarz operator from (\ref{addSW}) with the weights (\ref{omegai}) is merely used to improve the condition number and no longer needs to be convergent by itself.

Then the basic conjugate gradient iteration must additionally be implemented in parallel. This can easily be done in a way analogous to the overlapping Schwarz iteration above, involving a further  parallel vector product based on the overlapping decomposition $\{\Omega_i\}_{i=1}^P$. The details are obvious and are left to the reader. Note here again that, for general diagonal matrices $D_i$, the associated preconditioner is no longer symmetric, while it is in the case $D_i = \omega_i I$. This can cause both theoretical and practical problems for the conjugate gradient iteration. Then, instead of the conventional conjugate gradient method, we could resort to the {\em flexible} conjugate gradient method, which provable works also in the non-symmetric case, see  \cite{BDK15} and the references cited therein. But, as already shown in Lemma \ref{lemweig}, this issue is completely avoided 
with the choice $\gamma = \frac 1 2 n, n \in \mathbb{N}$.

\section{A fault-tolerant domain decomposition method}\label{sec:fault}

\subsection{Fault tolerance and randomized subspace correction}\label{sec:thft}
Now we will focus on algorithm-based fault tolerance. Usually, to obtain fault-tolerant parallel methods, techniques are employed which are  based on functional repetition and on check pointing. They are however prone to scaling issues, i.e. naive versions of conventional resilience techniques will not scale to the peta- or exascale regime, and more advanced techniques need to be devised. An overview of the state of the art  for resiliency in numerical algorithm design for extreme scale simulations is given in \cite{Agullo2020}. Depending on the PDE-problem to be treated, certain error-oblivious algorithms had been developed that can recover from errors without assistance as long as these errors do not occur too frequently. Besides time stepping methods with check pointing for parabolic and hyperbolic problems, various iterative linear methods and fixed-point solvers are able to execute to completion, especially in the setting of elliptic PDEs. To this end, the view of asynchronous iterative methods can be taken.
But to make, for example, a parallel multigrid or multilevel solver fault-tolerant is still a challenging task, see \cite{HuberGmeinerRuedeWohlmuth16, Stals2019AlgorithmbasedFR} for the two- and three-dimensional case. Furthermore, albeit numerical experiments often show good convergence and impressive fault-tolerance properties,
most asynchronous iterative methods are not just simple fixed-point iterations anymore and the development of a sound convergence theory for such algorithms is an issue.

For higher-dimensional time-dependent problems, fault mitigation on the level of the combination method has been experimentally tried for a linear advection equation in \cite{Harding14}, for the Lattice Boltzmann method and a solid fuel ignition problem in \cite{Ali2016ComplexSA}, and for a gyrokinetic electromagnetic plasma application in \cite{Pflueger.Bungartz.Griebel.ea.2014, ober17, Rentrop.Griebel.2020}, where existing (parallel) codes were used as black box solver for each of the subproblems and adaptive time stepping methods were employed. But again, the considered numerical approaches 
(Lax-Wendroff finite differences, the method of lines and Runge-Kutta 4th order in the GENE code, the two-step splitting procedure in the Taxila Lattice Boltzmann code) in general do not allow for a simple and clean convergence theory, nor do they allow for a Hilbert space structure due to the involved time and advection operators.  

However, if there is a direct Hilbert space structure as for elliptic problems, algorithm-based fault tolerance can be interpreted in the framework of
{\em stochastic subspace correction} algorithms for which in \cite{Griebel.Oswald:1995*1,Griebel.Oswald:2011, Griebel.Oswald:2016,Griebel.Oswald:2017*1} we recently developed a general theoretical foundation for their convergence rates in expectation. 
Indeed, for a conventional domain decomposition approach, we employed our stochastic subspace correction theory to show algorithm-based fault tolerance theoretically and in practice under independence assumptions for the random failure of subdomain solves in \cite{Griebel.Oswald:2019}.
The main idea is to switch from deterministic error reduction estimates of additive and multiplicative Schwarz methods as subspace splitting techniques to error contraction and thus to convergence {\em in expectation}. This way, convergence behavior and convergence rates can indeed be shown for certain iterative methods in a faulty environment, provided that specific assumptions on the occurrence and the distribution of faults and errors are fulfilled. 

To be precise, in \cite{Griebel.Oswald:2019}, we considered linear iterative, additive Schwarz methods for general stable space splittings, which we considered as subspace correction methods. For the setting of the two-level domain decomposition of (\ref{AS2}), the linear iteration reads as follows in our notation: In the $k$-th iteration step, a certain index set $I_k \subset \{0,1,\ldots,P\}$ is selected.  For each $i \in I_k$ the corresponding subproblem 
\begin{equation}\label{VPo}
	A_i d_i^{k} = R_i (b-A x^{k})
\end{equation} 
is solved for $d_i^{k}$, and an update of the form 
\begin{equation}\label{Rec0}
	x^{k+1}=x^{k}+\sum_{i \in I_{k} } \xi_{k,i} R_i^T d_i^{k} \quad \quad k=0,1,\ldots
\end{equation}
is performed. At the beginning we may set $x^{0}$ to zero without loss of generality. In the simplest case the relaxation parameters $ \xi_{k,i} $ are chosen independently of $x^{(k)}$ and we then obtain an, in general, non-stationary but linear iteration scheme. 
 The iteration (\ref{Rec0}) subsumes different standard algorithms such as the multiplicative (or sequential) Schwarz method,
where in each step a single subproblem (\ref{VPo}) is solved ($|I_k| =1$), the additive (or parallel) Schwarz method, where all $P+1$ subproblems are solved simultaneously ($I_k=\{0,1,\ldots,P\}$), and intermediate block-iterative schemes ($1<|I_k| <P+1$). Here and in the following, $|I|$ denotes
the cardinality of the index set $I$. The recursion (\ref{Rec0}) therefore represents a one-step iterative
method, i.e. only the current iterate $x^{k}$ needs to be available for the update step. 

In \cite{Griebel.Oswald:2019}, we specifically discussed stochastic versions of (\ref{Rec0}), where the sets $I_k$ are chosen {\em randomly}. To this end, we assumed that
\begin{itemize}
\item[{\bf A}] $I_k$ is a uniformly at random
chosen subset of size $p_k$ in $\{0,1,\ldots,P\}$, i.e., $|I_k|=p_k$ and  $\mathbb{P}(i\in I_m)=\mathbb{P}(i'\in I_k)$ for all $i,i'\in \{0,1,\ldots,P\}$. 
\item[{\bf B}] The choice of $I_k$ is independent for different $k$.
\end{itemize}
We then considered expectations of squared energy error norms for iterations with any fixed but arbitrary sequence $\{p_k\}$ and derived the following result:
\begin{theorem}\label{theo1}
	Let the relaxation parameters in (\ref{Rec0}) be given by $\xi_{k,i}:=\xi \omega_i$, ${i=0,1,\ldots,P,}$
where $0<\xi < 2/ \lambda_{\max}$. Furthermore let the random sets $I_k$ of size $p_k$ be selected in agreement with {\bf A} and let
{\bf B} hold. Then the algorithm (\ref{Rec0}) converges in expectation for any $x\in \mathbb{R}^N$ and
	\begin{equation}\label{ECC}
\mathbb{E}(\|x-x^{k}\|^2) \le \prod_{s=0}^{k-1}\left(1-\frac{ \lambda_{\max}~ \xi ~ (2-  \lambda_{\max}~ \xi) ~ p_s}{\kappa ~ (P+1)}\right) \|x-x^0\|^2,\qquad k=1,2,\ldots
	\end{equation}
where $\kappa:=\lambda_{\max}/\lambda_{\min}$ is the condition number of the underlying $\omega_i$-weighted decomposition.
\end{theorem}
Here $\|x\|:= \sqrt{x^T A x}$ denotes the discrete energy norm which is associated to the FEM-matrix $A$ of the discretized Laplace operator.
For a detailed proof in the case of general space splittings see \cite{Griebel.Oswald:2019}. 
Here the norm equivalency
\begin{equation}\label{NE}
	\lambda_{\min} ~ v^T C_{(2),\omega}  A v  \le v^TAv \le \lambda_{\max} ~ v^T C_{(2),\omega}  A v,\qquad v\in  \mathbb{R}^N
\end{equation}
has to hold with $0<\lambda_{\min}\le \lambda_{\max}\le \infty$ and positive
weights $\omega=\{\omega_i>0\}_{i=0,1,\ldots,P}$, where
we employ the $\omega$-weighted, symmetric,  additive Schwarz operator 
\begin{equation}\label{weigschwarz}
	C_{(2),\omega}^{-1}:=\sum_{i=0}^P \omega_i R_i^T A_i^{-1} R_i,
\end{equation}
compare (\ref{addSW}) with $D_i=\omega_i I$. Now any application of the linear iteration (\ref{Rec0}) with theoretical guarantees according to Theorem \ref{theo1} requires
knowledge of suitable weights $\omega_i$, and  an upper bound $\bar{\lambda}$ for the stability constant $\lambda_{\max}$ in order to choose the value of $\xi$, whereas information about $p_k$, i.e. the size of $I_k$, is not crucial. Numerical experiments for model problems with different values  $\xi\in  (0,2/\lambda_{\max})$ suggest that the iteration count is sensitive to the choice of $\xi$ and that overrelaxation often gives better results. The weights $\omega_i$ can be considered as scaling parameters that can be used to
improve the stability constants $\lambda_{\max}$, $\lambda_{\min}$, and thus the condition number $\kappa$ of the splitting.

The application of the above convergence estimate to fault tolerance will focus on the situation  
$$
1 << p_k\le p\le  P+1,
$$
where $p$ is the number of processors available in the compute network, and $p_k$ is a sequence of random integers denoting the number of correctly working processors in iteration $k$. 
For such a setting, the average reduction of the expectation of the squared error per iteration is approximately given by
\begin{equation}\label{faultrate}
\left(\prod_{s=0}^{k-1} \left(1-\frac{p_s}{(P+1) ~ \kappa}\right)\right)^{1/k}
\approx 1-\frac{\sum_{s=0}^{k-1} p_s}{k~(P+1)~\kappa} \approx
1-\frac{r_p}{\kappa},\qquad r_p:=\mathbb{E}(p_k)/(P+1),
\end{equation}
if we set $\xi=1/{\lambda}_{\max}$ and take a sufficiently large $k$.
The number $r_p$ can be interpreted as the average rate of subproblem solves
per iteration (\ref{Rec0}) and the linear dependence on $\kappa$ is what can be expected from a linear iteration in the best case.
Altogether, this convergence theory covers a stochastic version of Schwarz methods based on generic splittings, where
in each iteration a random subset of subproblem solves is used. On the one hand, this theory shows that randomized Schwarz iterative
methods are competitive with their deterministic counterparts. On the other hand, there are situations where randomness in the subproblem selection is naturally occurring and is not a matter of choice in the numerical method. An important example is given by
algorithm-based methods for achieving fault tolerance in large-scale distributed and parallel computing applications. 

This theory was already successfully applied to two different domain decomposition methods in a faulty environment, see \cite{Griebel.Oswald:2019} for further details. There, concerning the nature of faults, we assumed that faults are detectable and represent subproblem solves that were unreturned or declared as incorrect, i.e., we ignored soft errors such as bit flips in floating point numbers even if they were detectable. Moreover we assumed that the occurrence of a fault is not related to some load imbalance, i.e., slightly longer execution or communication times for a particular subproblem solve do not increase the chance of declaring such a process as faulty.  
Furthermore it should not matter if faults are due to node crashes or communication failures, nor did we pose any restrictions on spatial patterns (which and how many nodes fail) or temporal correlations of faults (distribution of starting points and idle times
of failing nodes). 
Then meaningful convergence results under such a weak fault model follow almost directly from the results above, as long as we can select, uniformly at random and independently of previous iterations, $p$ subproblems out of the $P+1$ available ones at the start of each iteration, assign them in a one-to-one fashion to the $p$ nodes,
and send the necessary data and instructions for processing the assigned subproblem solve to each of the $p$ nodes.
Indeed, if the time available for a solve step is tuned such that there is no correlation between faults and individual subproblem solves,
then one can safely assume that, with $f_k$ denoting the number of faulty subproblem solves, the index set $I_k$ corresponding to the $p_k=p-f_k$ subproblem solutions detected as non-faulty at the end of a cycle is still a uniformly at random chosen subset of $\{0,1,\ldots,P\}$ that is independent of the index sets $I_0,\ldots,I_{k-1}$ used in the updates of the previous iterations. The latter independence property is the consequence of our scheme of randomly assigning subproblems to processor nodes, and not an assumption on the fault model.
Thus Theorem \ref{theo1} applies and yields the estimate 
\begin{equation}\label{ECm}
\mathbb{E}(\|x-x^{k}\|^2) \le \prod_{s=0}^{k-1}\left(1-\frac{p_s}{\kappa (P+1)}\right) \|x-x^0\|^2,\qquad k=1,2,\ldots,
\end{equation}
for the expected squared error if we formally set $\xi={\lambda}^{-1}_{\max}$. 

In \cite{Griebel.Oswald:2019}, this approach was first applied for a simple manager-worker network example with a reliable manager node that possesses enough storage capacity to keep all necessary data, and a fixed number $p$ of unreliable worker nodes, which perform the calculations. Here communication only takes place between the manager and worker nodes, but not between worker nodes. There, the $P$ subproblems and also the coarse scale problem are randomly assigned to one of the available worker nodes for each iteration, the worker node receives the necessary data from the manager node for the corresponding subproblem, solves it and sends the solution back to the manager node. For simplicity $p$ was set to $P+1$. For example, a constant failure rate $r_f$ then results in a constant $p^*=p_k=\lfloor (1-r_f)(P+1)\rfloor$ and in each iteration a fixed number $f^* =P+1-p^*$ of compute nodes fail to return correct subproblem solutions, which finally means that the index set $I_k$ in the iteration was selected as a random subset of size $p^*$ from $\{0,1,\ldots,P\}$ and our assumptions in Theorem \ref{theo1} are indeed fulfilled.

Subsequently, our theory was also employed to a more general local communication network with fixed assignment of the subproblems $i=1,\ldots,P$ to $P$ unreliable compute nodes, decentralized parallel data storage with local redundancy, and treatment of the coarse scale problem on an additional reliable server node. Moreover communication was possible between the server node and the compute nodes but now also between geometrically neighboring compute nodes. Altogether, there were thus $p=P+1$ nodes available for the computation. Moreover the data of the subdomain problem associated to a processor were also redundantly stored on a fixed number of neighboring processors. This allowed, in case of a fault occurring on a processor, to proceed with the associated computations on one of these neighboring processors while delaying the computations of the neighboring processor. This behavior of the iterative method then could again be matched to our convergence theory with minor modifications, compare Corollary 1 in \cite{Griebel.Oswald:2019}. 

The deterioration of the convergence rate with the condition number $\kappa$ of the associated weighted splitting in (\ref{weigschwarz}) is typical for one-step iterations such as (\ref{Rec0}). Note that the convergence rate can be improved to a dependence on only $\sqrt \kappa$ rather than on $\kappa$ by using multi-step strategies. This was indeed shown in \cite{Griebel.Oswald:2019} for a two-step variant of the basic linear iteration. Note however that there is presently no similar theory for the conjugate gradient method with a stochastic two-level Schwarz preconditioner. Moreover we are presently not aware of simple, tight estimates for the values ${\lambda}_{\min}$ and ${\lambda}_{\max}$ that are associated to our specific splitting from Subsection \ref{ssect} of overlapping subdomains based on space-filling curves in the case of discretizations for general level parameters $l=(l_1,\ldots,l_d)$ and a coarse scale problem associated to agglomeration like (\ref{restr}). This is future work. 
In the following, we nevertheless apply our algorithm and report its behavior. 

\subsection{Our fault-tolerant algorithm}\label{ssfault}
Now we discuss the main features of our two-level domain decomposition algorithm based on space-filling curves in a faulty environment. The main idea is again to exploit redundancy to recover data in case of a fault and to resume computation. This redundancy is now provided by means of the $\gamma$-overlap, which is present in our construction of overlapping domains $\Omega_i$ based on space-filling curves. Furthermore recall that we employ the Bank-Holst paradigm, i.e. the coarse scale problem is stored and treated redundantly on each of the $P$ processors together with the respective subdomain problem. In the linear iteration, it interacts with the respective subproblem in an additive, parallel way, i.e. we solve the global coarse scale problem {\em and} the local fine subdomain problem independently of each other on one processor. The coarse scale problem is thus completely protected against faults due to this redundancy, i.e. as long as there is at least one processor running, it is always solved.
Thus, for reasons of simplicity, we focus on failures in the computation of the local subproblems and just record if a processor fails within its local subproblem solver in Algorithm \ref{Alg:DDMAlgorithm2}. In case of a fault of a certain processor we have to recover the necessary data of the associated subproblem from some other processor, which is a $\gamma$-dependent neighbor with respect to the one-dimensional space-filling curve ordering of the subdomains. This way, we can proceed with the iteration albeit delaying the corrections of faulty processors to a certain extent. Again, our above theory can be modified to cover this situation for the linear iteration case provided that faults are detectable, represent subproblem solves that were declared as unsuccessful or incorrect, and are happening independently of each other. We then have the following result:
\begin{corollary}\label{cor1} 
Let $I^1=\{0\}$ and $I^2=\{1,\ldots,P\}$. 
For fixed $x^{k}$, let  $x^{k+1}$ be given by (\ref{Rec0}) with $I_k=I^1\cup I_k^2$, where the
	$I^2_k$ are uniformly at random selected subsets of $I^2$ of size $p_k \le |I^2|$. Let furthermore $\xi_{k,i} := \xi_k \omega_i, i=0,\ldots,P$.
Then taking $\xi_k=p_k/(P~ \lambda_{\max})$ yields the estimate
	\begin{equation}\label{CorEneu}
		\mathbb{E}(\|x-x^{k+1}\|^2|x^{k}) \le \left(1-\frac{p_k^2}{\kappa~ P^2} \right) \|x-x^{k}\|^2.
	\end{equation}
\end{corollary}
This is a simple instance of the result in Corollary 1 in \cite{Griebel.Oswald:2019} for the local communication network adapted to our specific situation, see also the bound (32) in \cite{Griebel.Oswald:2019}. Here 
$\mathbb{E}(\|x-x^{k+1}\|^2|x^{k})$ denotes the expectation of $\|x-x^{k+1}\|^2$ conditioned on $x^{k}$.
If no faults occur in the $k$-th iteration, then the contraction rate stays $1-1/\kappa$ of course.

If we now assume, as it is often done in the literature, that a faulty processor in iteration $k$ comes back to life (or is replaced by a new processor) in a relatively short time and thus is active again in the next iteration, and if we assume that faults occur independently of each other, we obtain the following result for our algorithm involving the Bank-Holst paradigm:
\begin{theorem}\label{theo2}
	Let $\xi_{k,i} := \xi_k \omega_i, i=0,\ldots,P$ where $0< \xi_k <2/\lambda_{\max}$. Furthermore let $I^1=\{0\}$ and $I^2=\{1,\ldots,P\}$ and let $I_k=I^1\cup I_k^2$, where in each iteration $k$ of (\ref{Rec0}) the
$I^2_k$ are uniformly at random selected subsets of $I^2$ of size $p_k \le |I^2|$, i.e. let  $I^2_k$ be selected in agreement with {\bf A} and {\bf B}. 
Then the algorithm (\ref{Rec0}) converges in expectation for any $x\in \mathbb{R}^N$ with
	\begin{equation}\label{ECCBH}
\mathbb{E}(\|x-x^{k}\|^2) \le \prod_{s=0}^{k-1}\left(1-\frac{ p_s^2 }{\kappa ~ P^2}\right) \|x-x^0\|^2,\qquad k=1,2,\ldots
	\end{equation}
where $\kappa:=\lambda_{\max}/\lambda_{\min}$ is the condition number of the underlying $\omega$-weighted decomposition.
\end{theorem}
Now, for our setting, the average reduction of the expectation of the squared error per iteration is approximately given by
$$
\left(\prod_{s=0}^{k-1} \left(1-\frac{p_s^2}{(  \kappa ~ P^2)}\right)\right)^{1/k}
\approx 1-\frac{\sum_{s=0}^{k-1} p_s^2}{k~P^2~\kappa} \approx
1-\frac{\hat r_p}{\kappa},\qquad \hat r_p:=\mathbb{E}(p_k^2)/P^2,
$$
if we take $k$ to be sufficiently large. 
Compared to $r_p$ in (\ref{faultrate}) we have lost the simple interpretation as the average rate of subproblem solves per iteration. The squaring of the value $p_k/P$ in the expectation is due to the imbalance of fault probabilities between the coarse problem (never faulty) and fine level local problems (potentially faulty)
in comparison to those in Corollary 1 of \cite{Griebel.Oswald:2019}, which no longer allows to reduce the fraction and thus gives a slightly weaker upper bound. 

Altogether, we now have an estimate for the conditional expectation of the squared error in one iteration with possible faults for some of the $P$ processors and thus an asymptotic convergence theory for our specific additive, two-level domain decomposition method based on space filling curves for the faulty situation which involves the Bank-Holst approach. This is under the assumption that we have a good upper estimate of the value $\lambda_{\max}$ of the associated Schwarz operator, i.e. the associated splitting, and that we have independent occurrences of faults in each iteration and also independent occurrences between two different iterations. Note here that this gives -- as already discussed in Theorem \ref{theo1} -- not necessarily an optimal rate as for example for the non-faulty iteration method in (\ref{redrate}) with optimal damping parameter, but merely an upper bound. It nevertheless shows that the asymptotic convergence rate depends on $\kappa$ in a linear fashion, which is as good as we can hope for with a linear iterative method after all. 

We now describe the details of the fault-tolerant version of our domain decomposition method of Algorithm \ref{Alg:DDMAlgorithm} and \ref{Alg:DDMAlgorithm2}.
At the beginning, we assume to have $P$ processors available in our compute system, where each processor will be assigned to treat one of the $P$ subdomain problems (plus the redundant coarse scale problem) in our parallel domain decomposition method. This way, subdomain $i$ is uniquely associated to the processor with number $proc_i$.
Furthermore we have two different components in our approach, namely the failure process and the reconstruction process, which can be treated independently of each other.

In the {\em failure process} we decide which of the $P$ processors fail and which stay active. 
Here, for reasons of simplicity, we focus on failures in the computation of the local subproblems. 
Note that if a fault would occur during another part of the overall algorithm, for example in the coarse scale problem or in the computation of the residual, the data can be recovered directly. This is due to the redundancy of the coarse scale problem with the Bank-Holst approach in the first case and due to the redundancy induced by the overlap and the global nature of the involved operations in the second case. 
This is different for the treatment of the local subproblems in our domain decomposition approach.
In the following we measure the failures in terms of cycles. To this end, we define one cycle as one application of the additive Schwarz operator. 
For each processor we assume that the processor fails in a {\em random} fashion or successfully completes this cycle.  Once a processor is discovered to be faulty, its local subproblem will not contribute to the iteration for that cycle. 

Furthermore we assume that a faulty processor will be instantly repaired and is available again for the next iteration, i.e. a processor only stays faulty for the iteration in which the fault was detected. Indeed, that faulty processors come back to life relatively quickly or are quickly replaced is a common assumption in many algorithm-based fault-tolerant methods, see e.g. \cite{Harding14,Rentrop.Griebel.2020,ober17, Pflueger.Bungartz.Griebel.ea.2014}. Note here that this does in principle not exclude longer fault times of a processor since it may happen that the same processor is faulty again in the next iteration. But this happens with the same probability $p_\text{fault}$ in each iteration, which results in the much smaller overall probability $p_\text{fault}^s$ for a processor to be faulty for $s$ successive iterations. This way, we ensure the independence of faults between iterations as required for Corollary \ref{cor1} and Theorem \ref{theo2}. The more realistic scenario of a longer lasting fault is not directly treated by our theory, but, nevertheless, our assumptions and thus our theory still give a crude worst case bound for such cases by adjusting the failure rate $p_\text{fault}$ accordingly.

To simulate the failures of processors described above in our real compute system, we do the following:
In each cycle we determine the faulty processors via $P$ independently drawn 
Bernoulli-distributed random variables
with fault probability $p_\text{fault}$, which is the same across all processors.

Once a failed processor $proc_i$ returns to the computation in the next cycle, we first employ the {\em reconstruction process} to restore its local data. To this end, we reassign the interval limits of processor $proc_i$. This can be done without problem since, even as $proc_i$ has failed and its data is lost, its interval information can be taken from any other running processor since we stored the index vector $sfc\_index$ and the interval limits $ta_i,tb_i, \tilde ta_i,\tilde tb_i$ redundantly on all processors.
The processor $proc_i$ then locally recalculates his part of the index vector, 
i.e.\ it executes steps (4)-(6) of Algorithm \ref{Alg:DDMAlgorithm} with the input data of processor $proc_i$.
Consequently, processor $proc_i$ is aware which other processor(s) possess(es) indices which correspond to its lost local data. There exists at least one such healthy processor if $\gamma \geq 0.5$, since this ensures that at least two processors cover each index. These other processors then also possess, due to the overlapping nature of the local subproblem, the rows of $A$ as well as the entries of all vectors (iterate, residual, etc.) that correspond to the data lost by processor $proc_i$. 
Here we always choose the first processor that is not $proc_i$ along the space-filling curve ordering on which the required data is available.
Additionally, the redundant coarse level problem, in particular the matrices $R_0, P_0$ and $A_0$, can be fetched from any living processor since it is stored redundantly on all of them. Alternatively, it can be locally recalculated as in step (10) of Algorithm \ref{Alg:DDMAlgorithm} to avoid communication.

Altogether, we obtain the failure and recovery process as given in Algorithms \ref{Alg:DDMAlgorithm3} and \ref{Alg:DDMAlgorithm4}.
\begin{algorithm}[hbt!]
\caption{Overlapping two-level additive Schwarz using space-filling curves: Failure Cycle.}
\label{Alg:DDMAlgorithm3}
\begin{algorithmic}[1]
\Statex \textit{Input:} number of subproblems $P$, cycle number $k$
\Statex
\AlgOn{all processors $proc_i$ independently}
\State Determine if $proc_i$ has failed in cycle $k$.
\If{$proc_i$ has failed in cycle $k$}
\State Remove $proc_i$ from the computation for cycle $k$.
\ElsIf{$proc_i$ has not failed in cycle $k$ but has failed in cycle $k-1$}
\State \textit{Reconstruction} of the data of processor $proc_i$.
\EndIf
\AlgEndOn
\end{algorithmic}
\end{algorithm}

The determination of failures in Line 2 of Algorithm~\ref{Alg:DDMAlgorithm3} will simply be done by drawing independent random numbers from a Bernoulli distribution, i.e. $X \sim Bernoulli(p_\text{fault})$
with parameter $p_\text{fault}\in[0,1]$.
Moreover line 4 in Algorithm~\ref{Alg:DDMAlgorithm3} in the context of our DDM Algorithm~\ref{Alg:DDMAlgorithm2} essentially boils down to setting the local subproblem update $d^k_i$ to zero if the processor corresponding to subproblem $i$ failed in the current iteration $k$.
\begin{algorithm}[hbt!]
\caption{Overlapping two-level additive Schwarz using space-filling curves: Reconstruction.}
\label{Alg:DDMAlgorithm4}
\begin{algorithmic}[1]
\Statex \textit{Input:}  processor $proc_i$
\Statex
\AlgOn{processor $proc_i$}
	\State \parbox[t] {\dimexpr\linewidth-\algorithmicindent}{Recover interval bounds from a processor $proc_s \ne proc_i$ and \\ 
	recompute locally owned indices.\strut}
\For{locally owned indices $j=ta_i,\ldots,tb_i$}
    \State Find some processor $proc_t \ne proc_i$ which also owns index $j$.
    \State \parbox[t]{\dimexpr\linewidth-\algorithmicindent}{Copy necessary data (vector entries, matrix rows) corresponding to $j$ from \\
	processor $proc_t$ to processor $proc_i$.\strut}
\EndFor
\AlgEndOn
\end{algorithmic}
\end{algorithm}

Here the following remarks are in order: 
We employ the Bernoulli distribution which directly gives the independence of faults between two different iterations and thus fulfills the prerequisites of Corollary \ref{cor1} and Theorem \ref{theo2}. 
This is in contrast to the usual assumption on the temporal distribution of faults  
where a Weibull distribution for the failure arrival times is used instead, see \cite{Harding14,Ali2016ComplexSA,PAS2014,Rentrop.Griebel.2020}.
The usage of a Weibull distribution stems from 
the observations in \cite{SG10}. There, in 2010, failure data had been collected over 9 years at Los Alamos National Laboratory, which include about 23.000 failures recorded on more than 20 different compute systems. The subsequent analysis of inter-arrival times of failures then gave a good agreement with the Weibull distribution, especially when the system was already in production for some time and the processors had ''burnt in''. Not so much is known about the length of failures of a processor (it could be Weibull-, Gamma- or even lognormally-distributed). More recent studies can be found in \cite{YWWYZ12, JYS19}. There however, the picture is much less clear. Furthermore the data of fault distributions for current, substantially larger compute systems are not publicly known. 

Note that a Weibull distribution is not memoryless and involves dependencies of the fault time points. Thus the prerequisites of our Corollary \ref{cor1} and our Theorem \ref{theo2} are not fulfilled for faults with inter-arrival times according to a Weibull distribution. In contrast, the use of a Bernoulli distribution for the failures adheres to the prerequisites of our Corollary \ref{cor1} and our Theorem \ref{theo2} and allows proven convergence rates in the faulty situation. If the true fault distribution would indeed be Weibull-distributed, our approach using a Bernoulli distribution nevertheless gives a crude upper bound on the convergence rate of a fault-tolerant method, since it can be seen as a worst case scenario (but with existing convergence theory at least for the linear iteration case, while for fault-tolerant conjugate gradient iterations no convergence theory is presently available anyway). Moreover the time scale for the occurrence of faults is unrealistically large in most Weibull-based fault models. For more realistic fault values (like one fault a day on a big compute system) we would not see much difference at all, as the occurrence of faults during the relatively short computation time needed in our parallel domain decomposition method for any subdomain problem within the sparse grid combination method is scarce in case of an elliptic problem. The benefit of the fault-tolerant repair mechanism will rather be visible for time-dependent parabolic problems with long time horizons and long compute times, where elliptic DDM solvers for the subproblems arising in the combination method are invoked in each time step of an implicit time discretization method. But even then the number of faults during a long time simulation is rather small and there is thus not much difference between the choice of a Bernoulli or a Weibull distribution for the fault model. In any case, the Bernoulli choice gives an upper bound for the convergence and the usage of a Weibull distribution can only lead to better results in practice.

\section{Numerical experiments}\label{sec:numer}
\subsection{Model problem}
We will consider the elliptic diffusion-type model problem 
$$
- \nabla \alpha ({\bf x}) \nabla u({\bf x}) = f({\bf x}) \quad \mbox{ in } \Omega=[0,1]^{d} 
$$
with right hand side $f({\bf x})$ and appropriate boundary conditions on $\partial \Omega$.
Since we are merely interested in the convergence behavior, the scaling properties and the fault tolerance quality of our approach and not so much in the solution itself, we resort to the simple
Laplace problem, i.e. we set $\alpha= I$, $f=0$, and employ zero Dirichlet boundary conditions. Consequently, the solution is zero as well. For the discretization we employ finite differences on a grid with level parameter $l=(l_1,\ldots,l_d)$, which leads to $N$ interior grid points and thus $N$ degrees of freedom, compare (\ref{dof}), and which results in the associated matrix $A$.  Now any approximation $x^k$ during an iterative process directly gives the respective error in each iteration. We measure the error in the discrete energy norm associated to the matrix $A$ that stems from the finite difference discretization , i.e. we track 
$$\lVert x^k\rVert_A := \sqrt{(x^k)^TAx^k}$$
for each iteration of the considered methods. Note here that we run the iterative algorithms for the symmetrically transformed linear system ${\hat A \hat x= \hat b}$ with ${\hat A=T^TAT}$, ${\hat b= T^Tb}$, ${\hat x=T^{-1}x}$ and ${T= diag(A)^{-1/2}}$,  whereas we measure the error in the untransformed representation, i.e. for $x-x^k$. 
For the initial iterate $x^0$ we uniformly at random select the entries ${\tilde x_i^0, i=1\ldots,N,}$ of $\tilde x^0 $ from $[-1,1]$ and rescale them via $x_i^0 := \tilde x_i^0/\lVert\tilde x^0\rVert_A$ such that $\lVert x^0\rVert_A=1$ holds. To this end, we employed the routine $uniform\_real\_distribution$ of the C++ STL (Standard Template Library).
We then run our different methods until an error reduction of the initial error $x^0$ by at least a factor $10^{-8}$ is obtained and record the necessary number $K$ of iterations. 
We employ two types of convergence measures: First, we consider the average convergence rate $\rho^{ave} =( \lVert x^K\rVert_A/ \lVert x^0\rVert_A)^{1/K}$, which contains both a fast initial error reduction due to smoothing effects of the employed domain decomposition iteration on the highly oscillating random initial error in the first iterations
and the asymptotic linear error reduction later on. Secondly, we consider the asymptotic convergence rate $\rho^{asy}$. To this end, we use the maximum of the last 5 iterations and the last 5 percent of the iterations, i.e. we set $\tilde K := \max( 5, \lceil 0.05 \cdot K\rceil)$ and define $\rho^{asy}:=( \lVert x^K\rVert_A/ \lVert x^{K-\tilde K}\rVert_A)^{1/\tilde K}$ which gives the average convergence rate over the last $\tilde K$ iterations. 
Note that we usually have $\rho^{ave} \leq \rho^{asy}$. The quotient $\rho^{ave} / \rho^{asy}$ reflects the influence of the preasymptotic regime on the convergence. If the quotient is close to one, we have an almost linear decay of the iteration error starting from the very beginning, whereas, if the quotient is much smaller than one, we have a large and fast preasymptotic regime before the linear asymptotic error reduction sets in. Note furthermore that instead of $\rho^{ave}$ we may just give the necessary number $K$ of iterations. It bears the same amount of information as $\rho^{ave}$, since ${K=\lceil \log(10^{-8})/\log(\rho^{ave})\rceil}$. Note finally that, instead of the random initial iterate and zero right hand side, we could have chosen zero as initial guess and a right hand side which results from the application of the discrete Laplacian to an a priorily given non-zero solution. 
We then would converge to the discrete solution instead of zero. 
This approach might be more suitable from a practical point of view, since it eliminates the randomness of the initial guess altogether and also eliminates certain smoothing effects of our algorithms in the first few iterations (as discussed in more detail later on). Therefore, it almost completely eliminates the corresponding preasymptotic regime in the convergence and directly gives the asymptotic properties of our algorithms. 
However, picking a sufficiently general smooth solution in an arbitrary dimension is not an easy task and it may even happen that, for a too simple solution candidate, 
the convergence may look deceptively good. For our choice of a random initial guess we will observe the strongest influence of the preasymptotic regime in any case. 

In the following, we present the results of our numerical experiments. First, in Subsection \ref{subcp}, we study the convergence and parallel scaling behavior of both the 
linear two-level Schwarz/Richardson-type iteration as given in Algorithm \ref{Alg:DDMAlgorithm2} and the associated preconditioned conjugate gradient method. Then, in Subsection \ref{subft}, we report on the behavior of the fault-tolerant version of these two iterative methods in the presence of faults. 

All calculations have been performed on the parallel system 
{\em Drachenfels} of the Fraunhofer Institute for Algorithms and Scientific Computing (SCAI).
It provides, among others, $1.824$ Intel Sandy Bridge cores on $114$ compute nodes, each one with $2$ Xeon E5-2660 processors ($8$ cores per CPU, disabled SMT) at $2.20$ GHz and $32$ GB RAM, i.e. 2GB RAM per core,
and  $2.272$ Ivy Bridge (Xeon E5-2650 v2) cores on $142$ compute nodes, each one with $2$ Xeon E5-2650 processors ($8$ cores per CPU, disabled SMT) at $2.60$ GHz and $64$ GB RAM, i.e. 4GB RAM per core. Drachenfels is equipped with a Mellanox Technologies MT27500 Family [ConnectX-3] 4x (56Gbps) connection. Altogether, we may use $4.096$ cores on this system but we restricted ourselves to $256$ cores and thus subdomains in our experiments for practical reasons.

\subsection{Convergence and parallelization results}\label{subcp}

We are initially interested in the weak scaling situation, i.e. we set 
$$N:= 2^S P,$$
where $S$ denotes the weak scaling parameter. This results in the size $2^S$ for each subproblem, whereas $P$ denotes the number of subproblems. Thus the size of each subproblem stays fixed for rising $P$ and the overall number $N$ of unknowns grows linearly with $P$.
Before we turn to the results of our numerical experiments, let us shortly put this weak scaling situation into the conventional geometric perspective:
There, with the isotropic fine scale mesh width $h\approx N^{-1/d}$ and the isotropic coarse grid mesh size $H\approx (q\cdot P)^{-1/d}$, we get
\begin{equation}\label{Hh}
\frac H h \approx \left( \frac {q\cdot P}{N}\right)^{-1/d}= \left(q \cdot 2^{-S} \right)^{-1/d},
\end{equation}
which is independent of $P$.
From conventional domain decomposition theory we have (for geometric coarse grid problems) a condition number of the order $O(1+H/h)$, compare (\ref{AAA}) with $\delta=c h$. Now, with (\ref{Hh}), we obtain a condition number and a convergence rate which is independent of $N$ and $P$, i.e. we achieve weak scaling of the respective domain decomposition method. Moreover (\ref{Hh}) suggests the choice 
\begin{equation}\label{qchoice}
	q := c \cdot 2^S,
\end{equation}
which then gives $H/h \approx c^{-1/d}$ independent of $N$, $P$ and $S$.
We are interested in also observing such a weak scaling behavior for our non-geometric overlapping domain decomposition method, which is based on the space-filling curve approach and which involves a coarse scale problem set up purely algebraically via agglomeration. 

We study the convergence behavior of the two-level additive Schwarz/Richardson-type iteration from Algorithm \ref{Alg:DDMAlgorithm2} 
with optimal damping parameter $\xi^*=2/(\lambda_{min}+\lambda_{max})$ as a solver. To this end, we numerically determine the two eigenvalues by the Krylov-Schur method beforehand. Note that the eigenvalues correspond to their specific subspace splitting, i.e. to the specific two-level domain decomposition operator under consideration. The splitting depends on the dimension $d$, on the discretization level $ l=(l_1,\ldots,l_d)$, on the number $P$ of subdomains, on the weak scaling parameter $S$, on the coarse level parameter $q$, on the overlap parameter $\gamma$, on the respective scalings $D_i$, on the specific space-filling curve, and finally on the specific type of Schwarz operator (plain additive, balanced). 

First, we consider the one-dimen\-sio\-nal case in our numerical experiments. Surely there is no need to employ a parallel iterative domain decomposition solver and a direct sequential Gaussian elimination would be sufficient. However, this is a good starting point to study the convergence and parallel scaling properties of the various algorithmic variants. 
Moreover it will turn out that the one-dimensional case is indeed the most difficult one for good convergence and scaling results. 
This behavior stems from the relative ''distance'' of the fine scale to the coarse scale, which is maximal for $d=1$.
Here we study the convergence and scaleup behavior of the standard additive two-level Schwarz/Richardson-type approach (\ref{addSW}) as well as its balanced variant (\ref{bala}). Furthermore we consider the conjugate gradient method with these two associated variants of the respective two-level domain decomposition operators as preconditioner. 
We set $N:= 2^S P$ where $S:=8$. Thus the size of each subproblem stays fixed at $2^8$ with growing $P$, whereas the overall number $N$ of unknowns grows linearly with $P$. Moreover we fix $q=16$ and $\gamma=0.5$. We compare the different methods for the three scalings $D_i=I$ (no scaling at all), $D_i=\omega_i I$ with $\omega_i$ according to (\ref{omegai}), and $D_i$ according to (\ref{Di}), $i=1,\ldots,P$.
For our special choice of $\gamma$ we know from Lemma \ref{lemweig} that the weighting with $\omega_i$ and the weighting with $D_i$ are indeed the same and differ from the unweighted case by just the constant scaling $\frac 1 c I$ with $c=2 \gamma+1$. For the coarse scale problem, we always set $D_0=I$. 
The results for the two types of convergence rates of the standard additive two-level Schwarz/Richardson-type approach (\ref{addSW}) and the associated conjugate gradient method are shown in Figure \ref{compare1_rho}. The results for the two types of convergence rates of the balanced variants according to (\ref{bala}) are given in Figure \ref{compare2_rho}. 
\begin{figure}[ht]
\centering
	\includegraphics[width=0.82\textwidth,height=0.33\textheight]{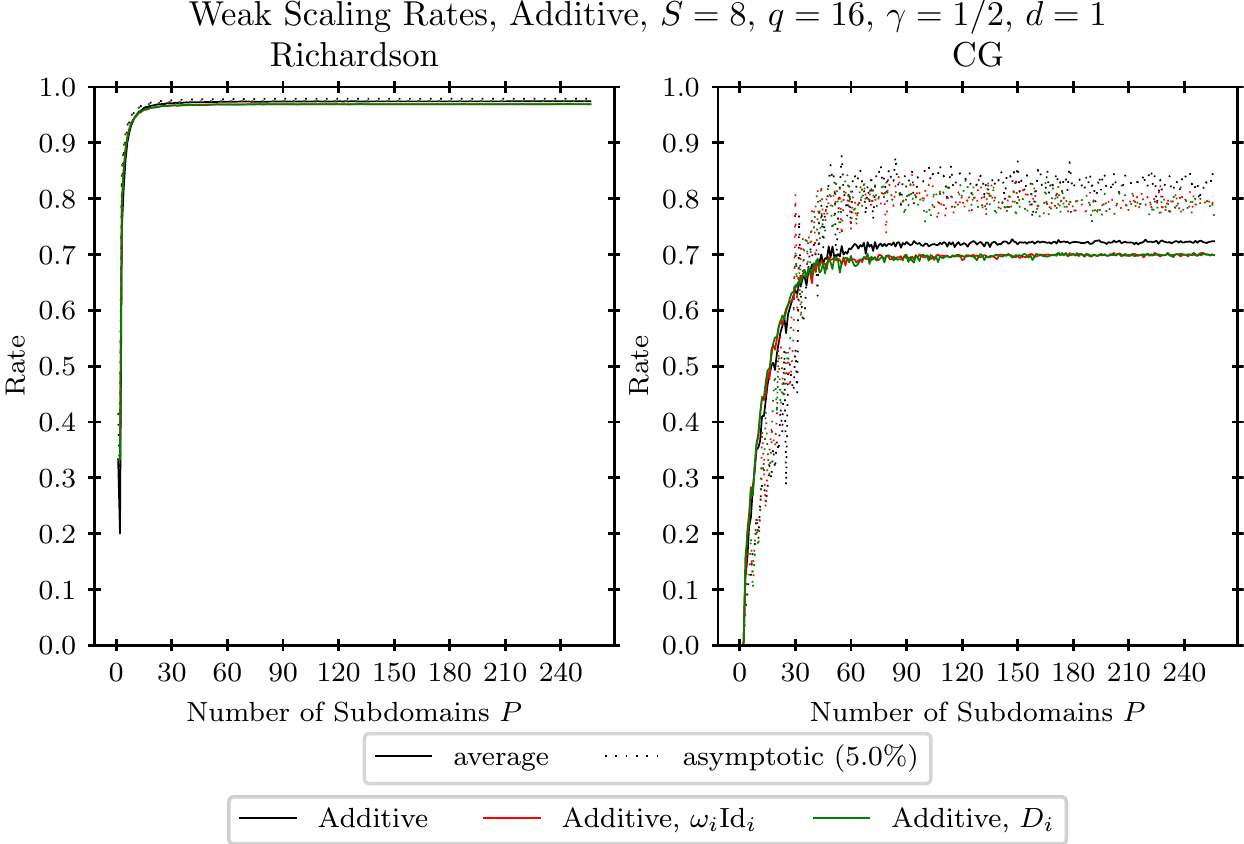}
	\caption{Weak scaling: Convergence rates $\rho^{ave}$ and $\rho^{asy}$ versus number of subdomains for the three weighted versions of the {\em plain} two-level additive Schwarz/Richardson iteration (left) and the associated preconditioned conjugate gradient method (right), $d=1$, $q=16$, $\gamma=0.5$.}
\label{compare1_rho}
\end{figure}
\begin{figure}[ht]
\centering
	\includegraphics[width=0.82\textwidth,height=0.33\textheight]{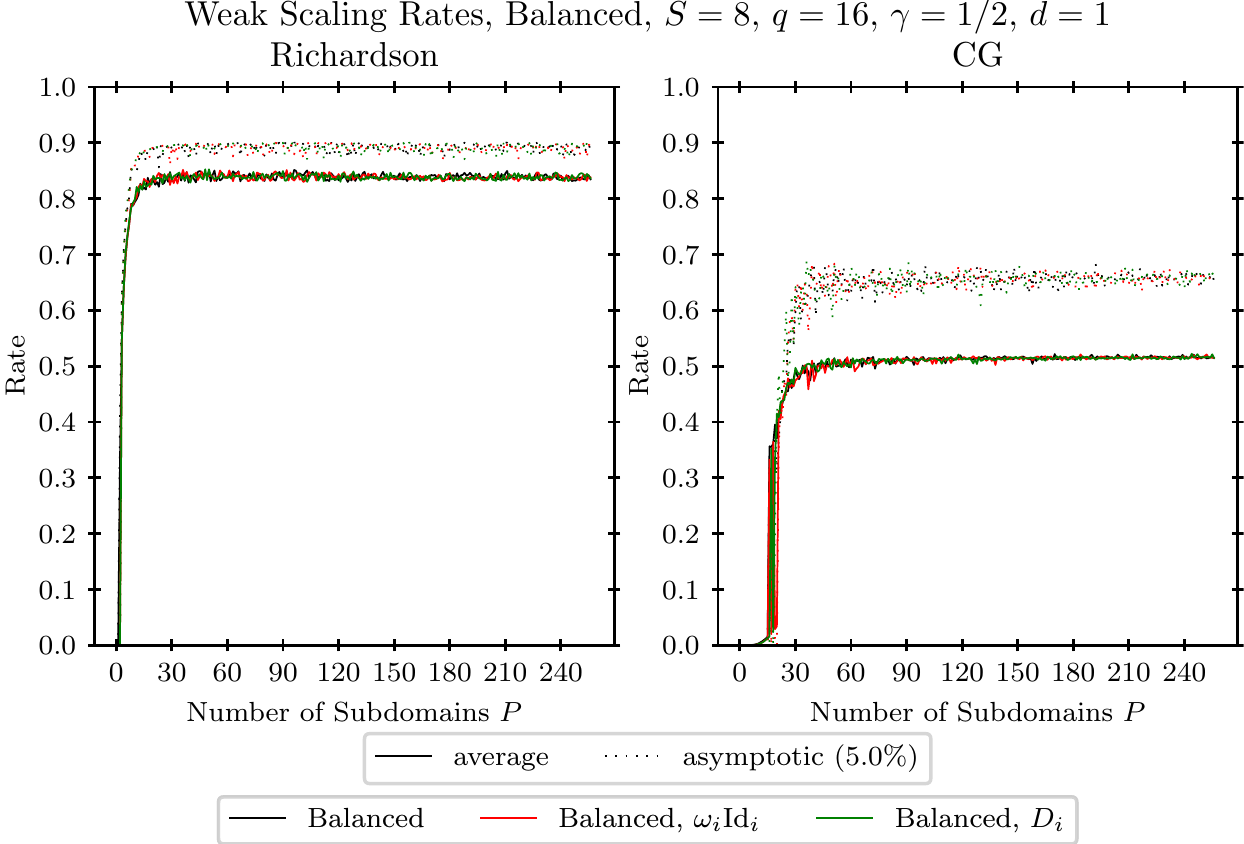}
	\caption{Weak scaling: Convergence rates $\rho^{ave}$ and $\rho^{asy}$ versus number of subdomains for the three weighted versions of the {\em balanced} two-level additive Schwarz/Richardson iteration (left) and the associated preconditioned conjugate gradient method (right), $d=1$, $q=16$, $\gamma=0.5$.}
\label{compare2_rho}
\end{figure}

We see that the scaling with $\omega_i$ and the scaling with $D_i$ indeed give the same results, as was expected. 
Moreover the case of no scaling at all only gives slightly worse results  for the plain additive variant and even gives the same results for the balanced variant for this one-dimensional situation. Furthermore we observe that $\rho^{ave}$ and $\rho^{asy}$ are approximately the same in the Richardson case, i.e. there is not much difference to be seen.
This is surely due to the large number of iterations needed to reach an error reduction of $10^{-8}$ and a moderate preasymptotic regime.  
The rates are quite high with a value of about 0.98 for larger values of $P$.
Moreover, in the conjugate gradient case, we see a slight difference between $\rho^{ave}$ and $\rho^{asy}$ of about 0.08, which is due to the preasymptotic regime incorporated in $\rho^{ave}$ and the lower number of iterations involved.
For the balanced variants, there is still not much difference for the Richardson iteration, i.e. there is just a difference of about 0.04 between the $\omega_i$-scaling and no-scaling at all, but the rates are a improved with values of about 0.9. In the conjugate gradient case, we now observe an improvement to a rate of 0.6 for $\rho^{ave}$ due to balancing. Thus balancing indeed improves convergence substantially. Moreover we obtain a more distinct difference between $\rho^{ave}$ and $\rho^{asy}$ with a value of 0.13, which is again due to the preasymptotic regime incorporated in $\rho^{ave}$ and the lower number of iterations involved. Nevertheless, for rising values of $P$, constant convergence rates are obtained in any case, which indicates weak scaling.

We conclude that the average convergence rate $\rho^{ave}$ is a sufficient measure to unveil the weak scaling properties of our various algorithms, even though the asymptotic convergence rate $\rho^{asy}$ is a bit larger.
Since $\rho^{ave}$ is related via $K=\lceil \log(10^{-8})/\log(\rho^{ave})\rceil$ to the number $K$ of required iterations, we will report the number $K$ of iterations instead of $\rho^{ave}$ for the remainder of this subsection. As we are interested in weak scaling properties and since a fixed reduction of the initial error is considered anyway, the necessary number $K$ of iterations is sufficient for our purpose and we may neglect the influence of the preasymptotic regime. This will be different in the next Subsection \ref{subft}, where we will consider the behavior of our algorithms for different fault probabilities. Then the size of the preasymptotic regime can be influenced by faulty subproblem solves and plays a certain role, as we will see later on.  

Thus, from now on in this subsection, the number $K$ of iterations necessary to reduce the initial error by a factor of $10^{-8}$ is shown. For the plain additive two-level Richardson-type approach (\ref{addSW}) and the associated conjugate gradient method, they are given in Figure \ref{compare1_it}. For the balanced variants according to (\ref{bala}), they are given in Figure \ref{compare2_it}. 
\begin{figure}[ht]
\centering
	\includegraphics[width=0.82\textwidth,height=0.33\textheight]{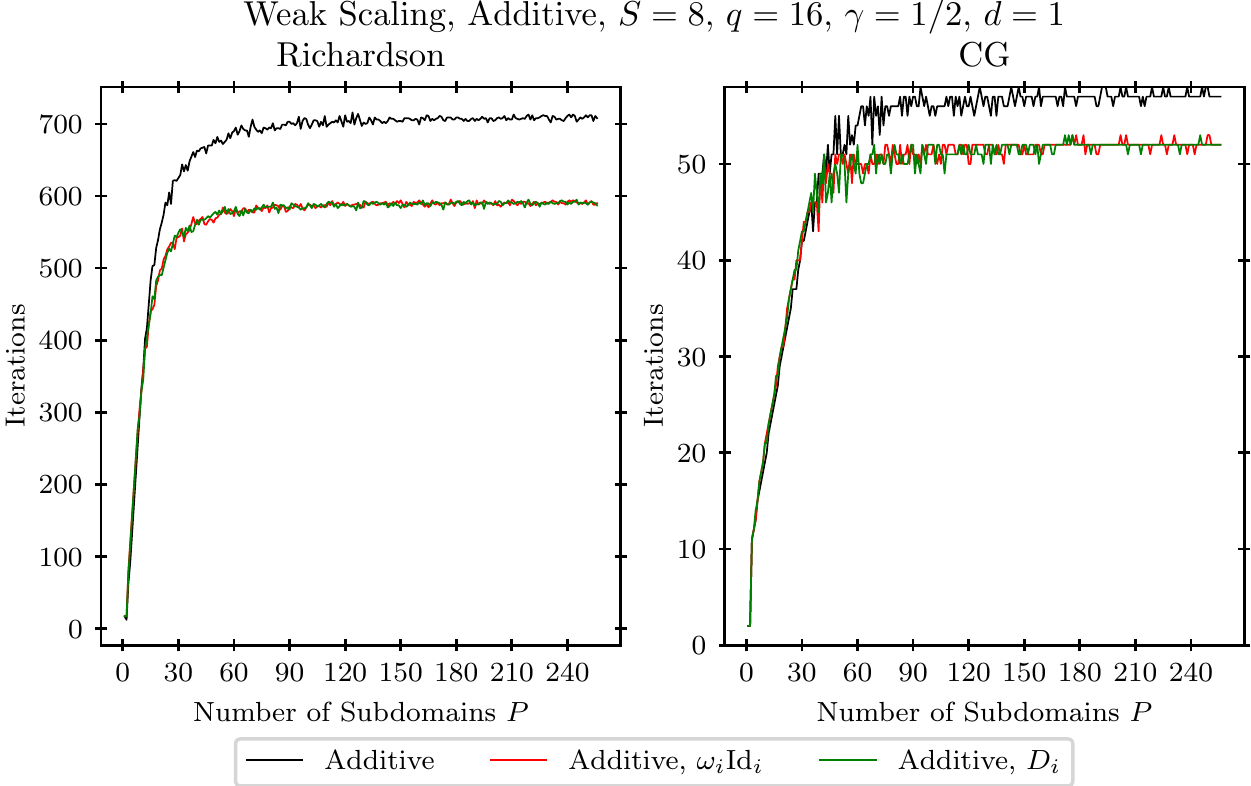}
	\caption{Weak scaling: Number of iterations $K$ versus number of subdomains for the three differently weighted versions of the {\em plain} two-level additive Schwarz/Richardson iteration (left) and the associated preconditioned conjugate gradient method (right), $d=1$, $q=16$, $\gamma=0.5$.}
\label{compare1_it}
\end{figure}
\begin{figure}[ht]
\centering
	\includegraphics[width=0.82\textwidth,height=0.33\textheight]{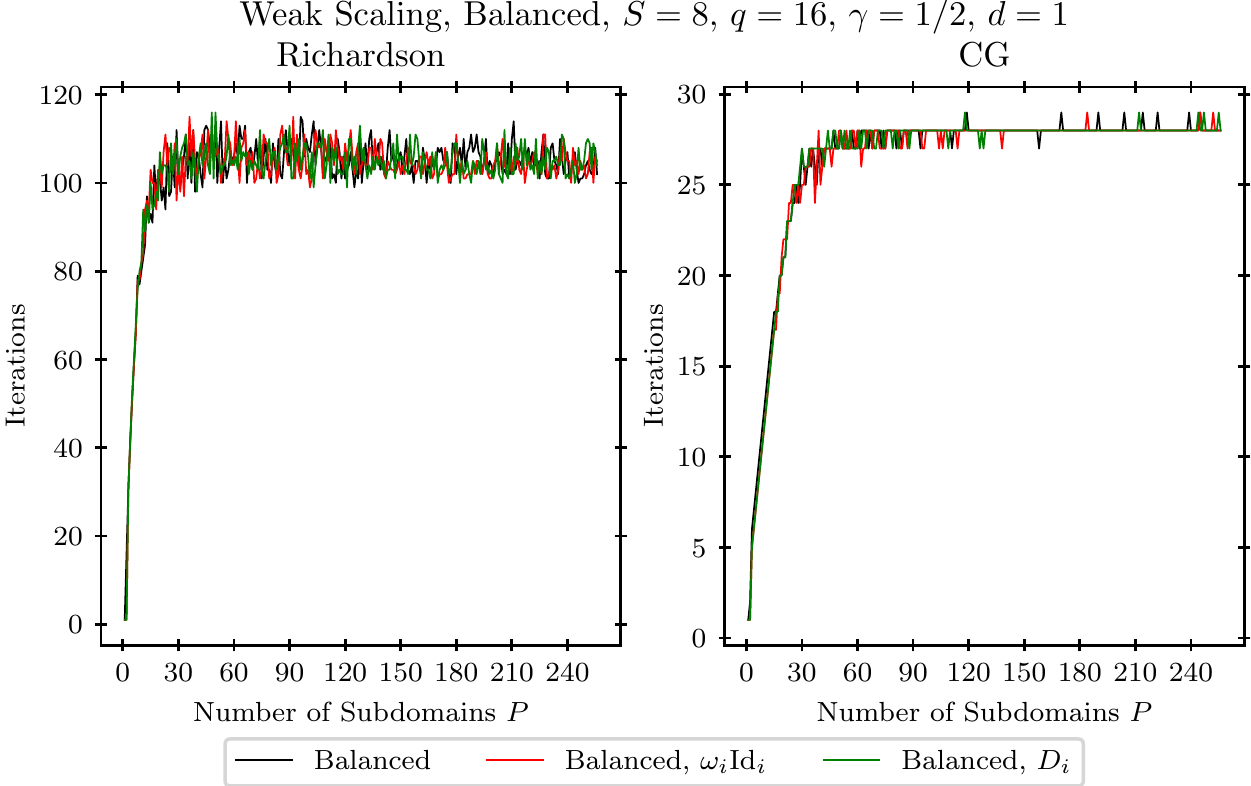}
	\caption{Weak scaling: Number of iterations $K$ versus number of subdomains for the three differently weighted versions of the {\em balanced} two-level additive Schwarz/Richardson iteration (left) and the associated preconditioned conjugate gradient method (right), $d=1$, $q=16$, $\gamma=0.5$.}
\label{compare2_it}
\end{figure}

Of course we observe the same behavior as in the two previous figures for $\rho^{ave}$, albeit now in a different representation which is due to the measured number of iterations. Moreover we again obtain weak scaling behavior in all cases, i.e. the necessary number of iterations stays constant for growing values of $P$. This constant depends on the respective splitting:
For the plain additive Schwarz/Richardson iteration (left) we see that a scaling with $\omega_i$ (and equally with $D_i$) reduces this constant compared to the no-scaling case, albeit a large number of iterations is still needed. Moreover the results get substantially improved by the balanced variants: All three scalings now give the same results and the weak scaleup constant is reduced by a factor of approximately seven for the unweighted case and by a factor of approximately six for the other two variants.
For the associated preconditioned conjugate gradient methods (right) we observe a further reduction of the necessary number of iterations. 
This reflects roughly the $\kappa$-versus-$\sqrt \kappa$ effect of the conjugate gradient method in its convergence rate.
For the balanced version, we additionally see a substantial improvement of the scaleup constant compared to the plain additive preconditioned conjugate gradient method by a factor of nearly one half.
Note here that for values of $\gamma$ that are not integer multiples of 1/2, the $D_i$-scaling does not lead to a symmetric operator, which renders a sound and robust convergence theoretically questionable and in further experiments gave considerably worse iteration numbers with oscillating behavior for the corresponding preconditioned conjugate gradient method in practice.

We conclude that the balanced variant of both the Schwarz/Richardson iteration and the associated preconditioned conjugate gradient method is substantially faster than the plain version. We also see that, for our choice $\gamma=1/2$, balancing eliminates the difference of the unscaled and the $\omega_i$-scaled (and $D_i$-scaled) cases. Moreover the preconditioned conjugate gradient version is nearly quadratically faster and gives good weak scaling constants. 
In further comparisons with the Nicolaides method (\ref{balan}) it was found that the balanced approach was again superior. Also, a comparison with the deflated variant (\ref{balan}) revealed that the balanced method was more robust.
Therefore, we will from now on focus on the optimally damped, {\em balanced} Schwarz iteration and the associated preconditioned conjugate gradient method. The type of damping we will choose, i.e. none at all, $\omega_i$ according to (\ref{omegai}) or $D_i$ according to (\ref{Di}) is still to be determined. We refrain from employing the $D_i$-weighting for general choices of $\gamma$, since in general this results in a non-symmetric operator for which our theory is not valid anymore. 
It now remains to study the behavior of the unweighted and the $\omega_i$-weighted algorithms in more detail.

So far, we kept the value of the overlap parameter $\gamma$ fixed. Now we vary $\gamma$ and consider its influence on the weak scaling behavior of our two algorithms in the balanced case. First we consider the one-dimensional situation, where we set $S=8$, $q=16$ and vary the number $P$ of subdomains. The resulting number of iterations for different values of $\gamma$ ranging from $1/5$ up to $5$ are shown in Figure \ref{weak_vary_gamma_1d} for the unweighted case and in Figure \ref{weak_vary_gamma_omegai_1d} for the $\omega_i$-weighted case.
\begin{figure}[ht]
\centering
	\includegraphics[width=0.82\textwidth,height=0.38\textheight]{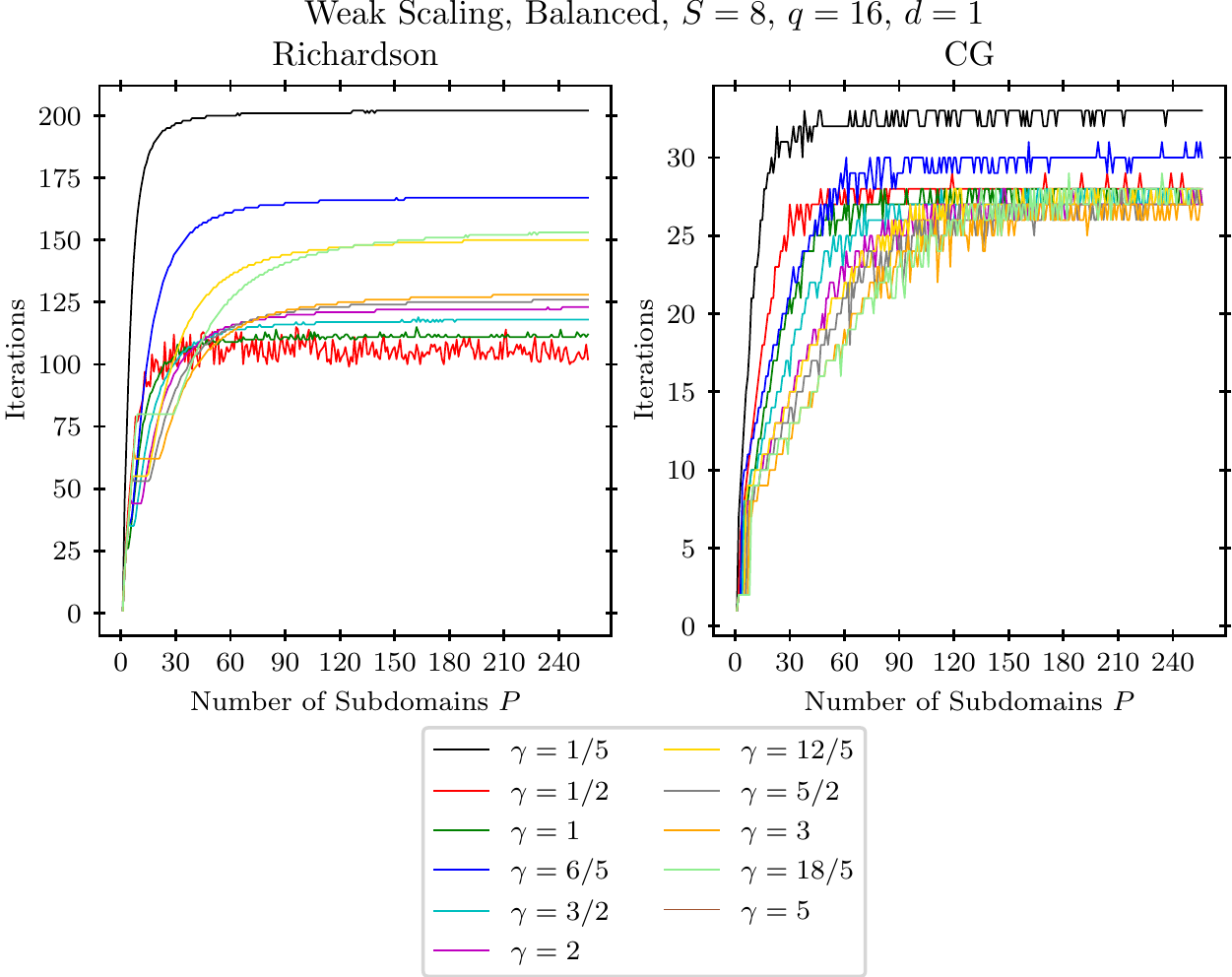}
	\caption{Weak scaling: Number of iterations versus number of subdomains for the unweighted balanced two-level additive Schwarz/Richardson iteration (left) and the associated preconditioned conjugate gradient method (right) for different values of $\gamma$ with varying $P$, $d=1$, $q=16$, $S=8$.}
	\label{weak_vary_gamma_1d}
\end{figure}
\begin{figure}[ht]
\centering
	\includegraphics[width=0.82\textwidth,height=0.38\textheight]{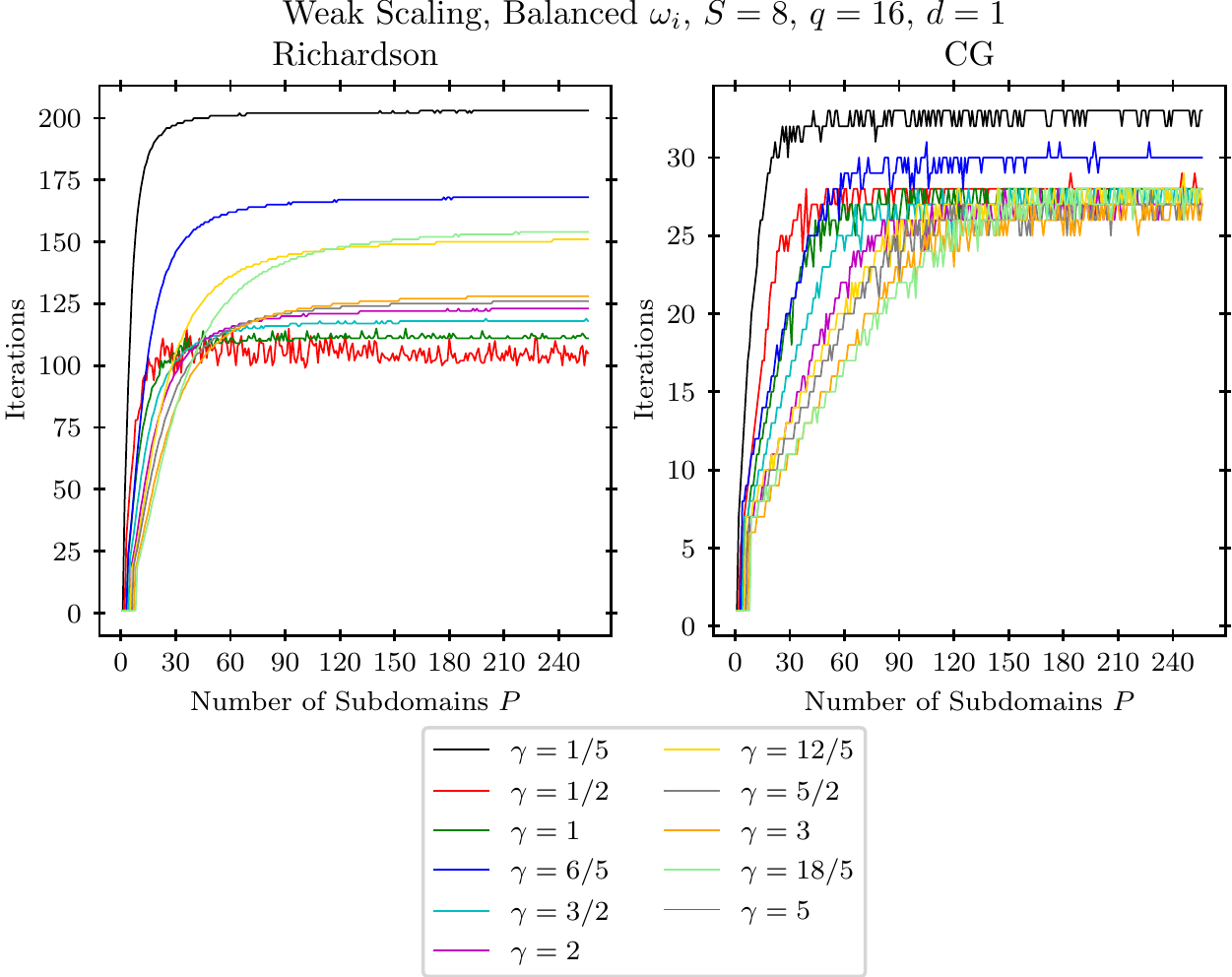}
	\caption{Weak scaling: Number of iterations versus number of subdomains for the $\omega_i$-weighted balanced two-level additive Schwarz/Richardson iteration (left) and the associated preconditioned conjugate gradient method (right) for different values of $\gamma$ with varying $P$, $d=1$, $q=16$, $S=8$.}
	\label{weak_vary_gamma_omegai_1d}
\end{figure}

Comparing the $\omega_i$-weighted case to the unweighted case, there is not much visible difference at all.
We again clearly see weak scaling behavior.
The scaling constant now depends on the respective value of $\gamma$. 
In the Richardson case, it is interesting to observe that, in any case, a constant number of iterations is quickly reached for rising values of $P$. 
Moreover, for a small overlap value of $\gamma=1/5$, it is quite bad. The number of iterations is seen to be optimal for $\gamma=1/2$ and then deteriorates for larger values of $\gamma$. 
Note at this point that, starting with $\gamma=1/2$, we observe a slight deterioration of the convergence curves and of the weak scaling constant, where this deterioration is {\em monotone} in $n$ if we restrict ourselves to values of $\gamma$ that are integer multiples of $0.5$, i.e. $\gamma=\frac 1 2 n, n \in \mathbb{N}_+$. The non-integer multiples give worse results.
In the conjugate gradient case, the scaling constant is reached increasingly later for rising values of $P$ (which is desirable). It is improved in a nearly monotonic way for rising values of $\gamma$. 
Furthermore the absolute value of necessary iterations is again much smaller than in the Richardson case.  

Next, we consider the three-dimensional case and again vary the value of $\gamma$. Then, in contrast to the one-dimensional situation, the Hilbert curve structure comes into play. The resulting iteration numbers for different values of $\gamma$ in the unweighted and the $\omega_i$-weighted case are shown in Figures \ref{weak_vary_gamma_3d} and \ref{weak_vary_gamma_omegai_3d}, respectively.
\begin{figure}[ht]
\centering
	\includegraphics[width=0.82\textwidth,height=0.38\textheight]{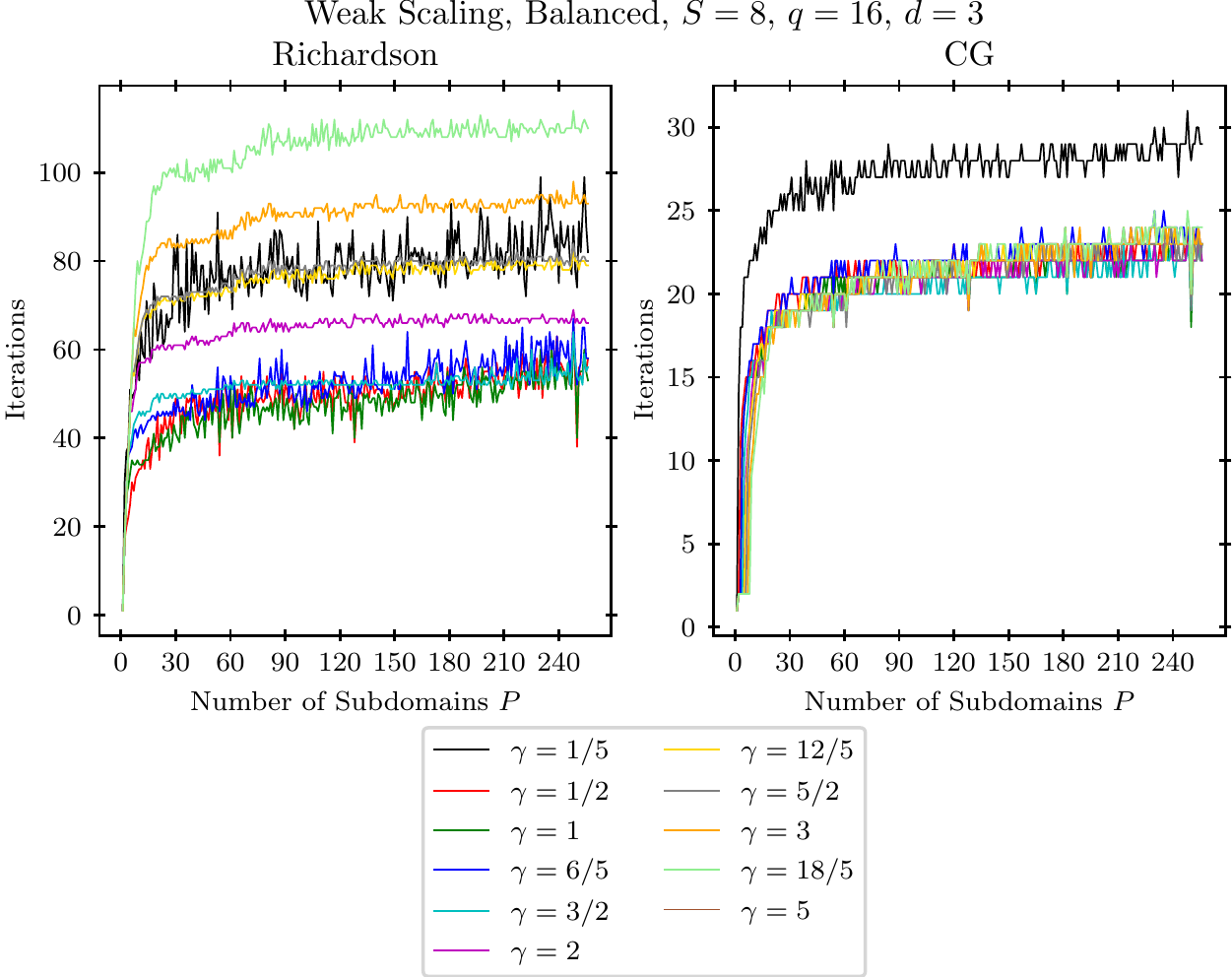}
	\caption{Weak scaling: Number of iterations versus number of subdomains for the unweighted balanced two-level additive Schwarz/Richardson iteration (left) and the associated preconditioned conjugate gradient method (right) for different values of $\gamma$ with varying $P$, $d=3$, $q=16$, $S=8$.}
	\label{weak_vary_gamma_3d}
\end{figure}
\begin{figure}[ht]
\centering
	\includegraphics[width=0.82\textwidth,height=0.38\textheight]{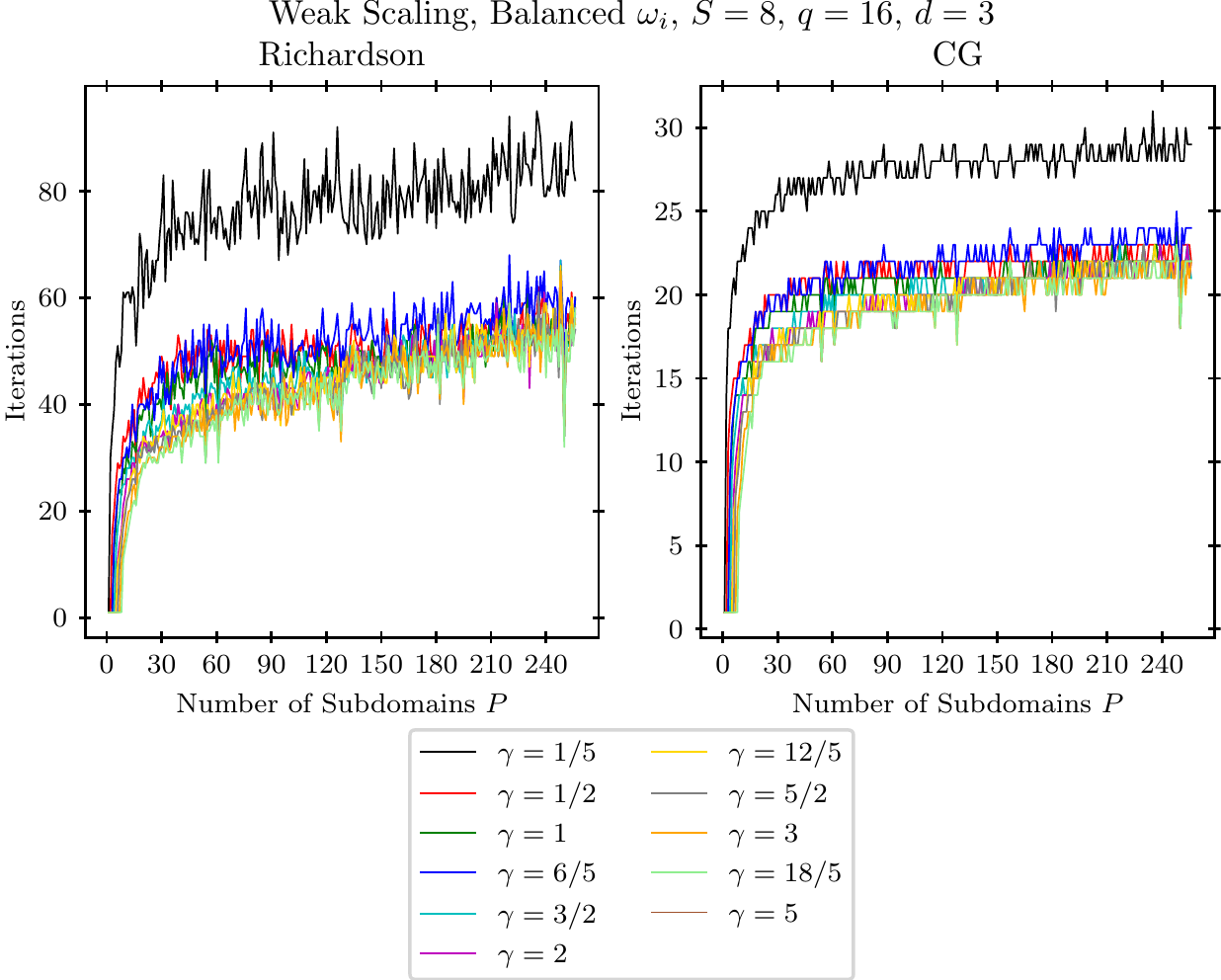}
	\caption{Weak scaling: Number of iterations versus number of subdomains for the $\omega_i$-weighted balanced two-level additive Schwarz/Richardson iteration (left) and the associated  preconditioned conjugate gradient method (right) for different values of $\gamma$ with varying $P$, $d=3$, $q=16$, $S=8$.}
	\label{weak_vary_gamma_omegai_3d}
\end{figure}

For the Richardson iteration we again see an analogous behavior for the weak scaling constant but with a much lower number of iterations compared to the one-dimensional case. Furthermore we observe that the $\omega_i$-scaling improves the convergence: Except for $\gamma=1/5$, all required numbers of iterations are now close together for both the values of $\gamma$ that are integer multiples of $1/2$ and the other values, and their respective number of iterations are in any case reduced to around $50$. 
This shows that $\omega_i$-weighting is able to deal with the decomposition based on the Hilbert curve, which appears for $d>1$, in a proper way.
In the unweighted conjugate gradient case, the curves are approximately the same for all values of $\gamma$, except for $\gamma=1/5$ which is too small again. 
They are successively improved for rising values of $\gamma$ in the $\omega_i$-weighted case, as is expected intuitively. Again, the conjugate gradient method is much faster than the Richardson scheme. Similar observations could be made for other dimensions.

Altogether, $\omega_i$-weighting stabilizes the iteration numbers against variations of $\gamma$ and improves the convergence behavior. In light of the larger costs involved for higher values of $\gamma$, a good choice for $\gamma$ is given by the value $1/2$ in both the Richardson and the conjugate gradient case. 
However, larger values of $\gamma$ may be needed to attain fault tolerance due to redundancy later on.
The $\omega_i$-weighted case then indeed results in slightly better convergence results for larger values of $\gamma$ and shows no deterioration, which the unweighted case does.    
From now on we therefore will focus on the optimally damped, $\omega_i$-weighted balanced Schwarz iteration and the associated preconditioned conjugate gradient method.
Moreover we  will restrict ourselves in this subsection to the value $\gamma=1/2$. 
In the next subsection, where we deal with fault tolerance, we will also employ values of $\gamma$ that are larger integer multiples of $1/2$. 

So far, we kept the weak scaling parameter $S$, i.e. the size $2^S$ of each subproblem, fixed and only varied the number $P$ of subdomains. But what happens if we also vary the subproblem size? For fixed $q=16$, the results are shown in Figure \ref{fixq}.
\begin{figure}[ht]
\centering
	\includegraphics[width=0.82\textwidth,height=0.33\textheight]{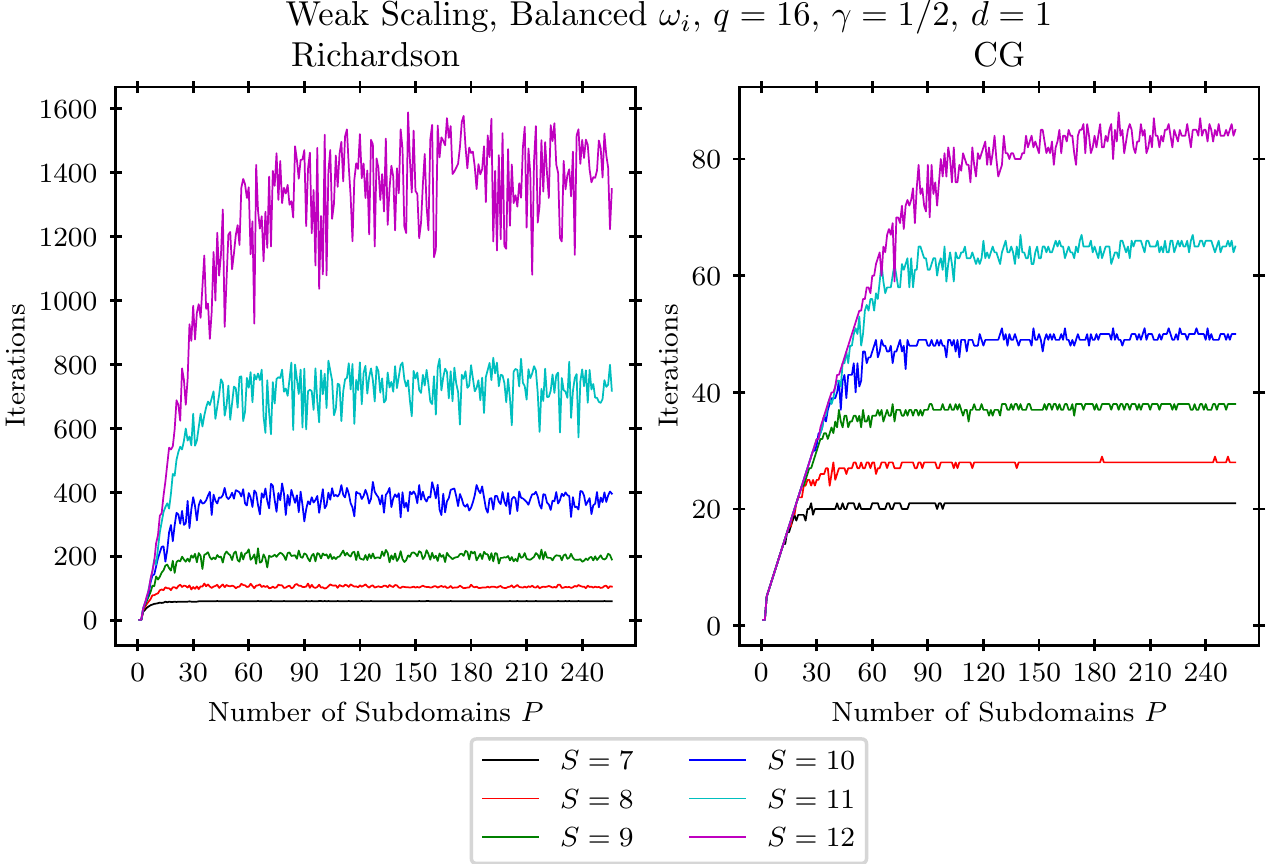}
	\caption{Weak scaling: Number of iterations versus number of subdomains for the $\omega_i$-weighted balanced two-level additive Schwarz/Richardson iteration (left) and the associated preconditioned conjugate gradient method (right) for different values of $S$, $d=1$, $q=16$, $\gamma=0.5$.}
\label{fixq}
\end{figure}

We again clearly see weak scaling behavior.
However the weak scaling constant now depends on the subproblem size, i.e. it grows with rising $S$. This holds for both the Richardson-type iteration and the conjugate gradient approach. This behavior stems from the fixed value of $q$ and thus the fixed size of the coarse scale problem for fixed $P$. In this case the difference between fine scale and coarse scale increases with growing $S$ and, consequently, the additive coarse scale correction is weakened compared to the fine scale.

We now let the coarse scale parameter $q$ be dependent on $S$, compare also the discussion leading to (\ref{qchoice}). To be precise, we set $q=2^{S-4}$, which gives a coarse problem of size $P \cdot 2^{S-4}$, where we vary the parameter $S$. This way, we double $q$ while doubling the subdomain size $2^S$, i.e. $N/P$.
The obtained results are shown in Figure \ref{varq}. 
\begin{figure}[ht]
\centering
	\includegraphics[width=0.82\textwidth,height=0.33\textheight]{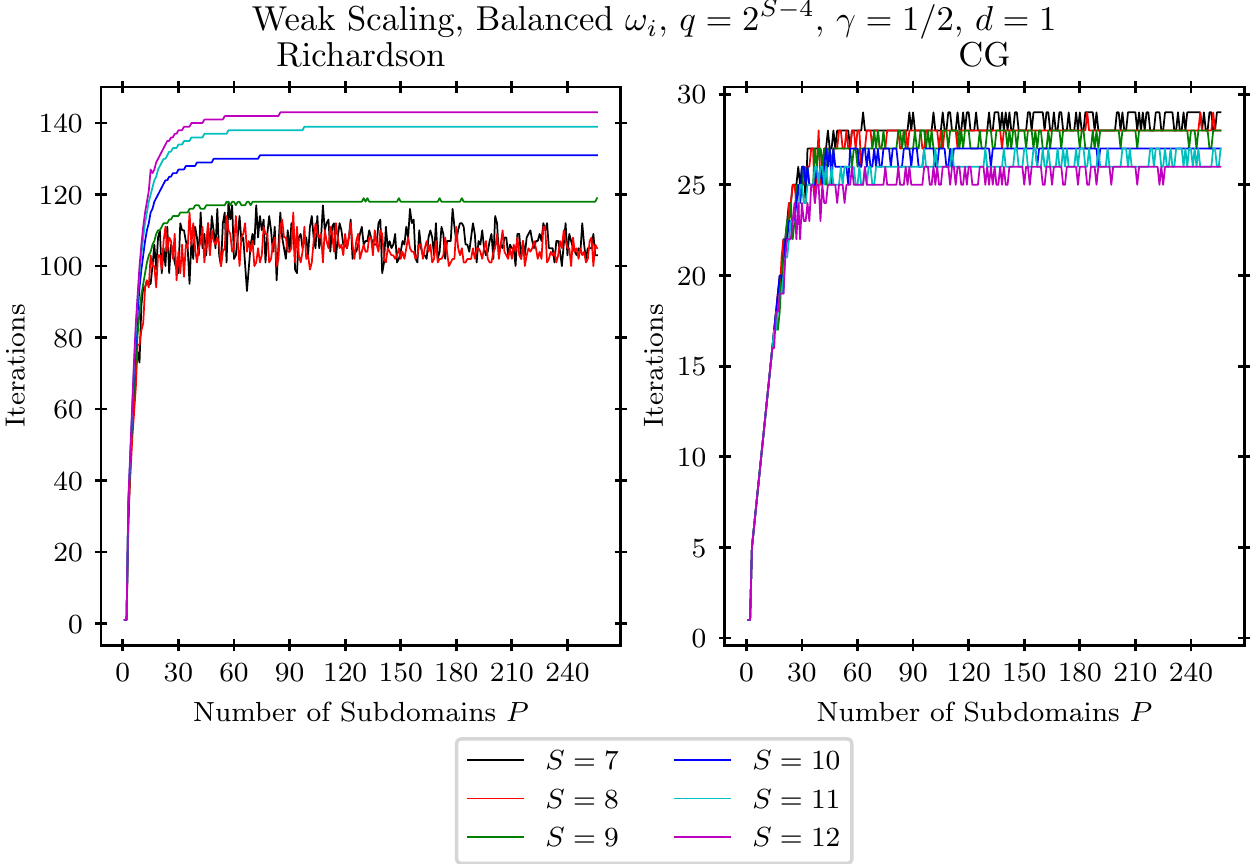}
	\caption{Weak scaling: Number of iterations versus number of subdomains for the $\omega_i$-weighted balanced two-level additive Schwarz/Richardson iteration (left) and the associated preconditioned conjugate gradient method (right) for different values $S$, $d=1$, $q=2^{S-4}$, $\gamma=0.5$.}
\label{varq}
\end{figure}

In all cases, we obtain substantially improved results compared to the fixed choice ${q=16}$ from Figure \ref{fixq}. This was to be expected since now the coarse scale correction is improved for rising $S$. We again observe an asymptotically constant number of iterations for growing values of $P$ and we obtain a weak scaling constant which is, compared to the fixed choice ${q=16}$, now only slightly growing with $S$ for the Richardson iteration and seems to approach a limit of about 145 for rising values of $S$. Moreover it is now completely independent of $S$ for the conjugate gradient approach, for which we need at most 29 iterations in all cases. This shows that $q$ should scale proportional to $N/P$. In further experiments we quadrupled $q$ while doubling $N/P$ and found that the weak scaling constant then even shrank with growing $S$. 

Note at this point that the size of the coarse scale problem is now $2^{-4} \cdot P \cdot 2^S$, while the size of each subdomain problem is $2^S$ (whereas the size of the overall problem is $N=P\cdot 2^S$).
Thus the cost of solving the coarse scale problem tends to dominate the overall cost with rising $P$, which is the price to pay for a more uniform convergence behavior. This calls for further parallelization of the coarse scale problem itself via our $P$ processors to remedy this issue. Thus, in contrast to the present implementation via the Bank-Holst paradigm where we redundantly keep the coarse grid problem on each processor (besides the associated subdomain problem), we should partition the coarse scale matrix to the $P$ processors and therefore solve the coarse scale problem in a parallel way.
This however will be future work.

Next, we consider the weak scaling behavior for varying dimensions $d=1,2,3,4,5,6$.
For the discretization we stick to the isotropic situation, i.e. we set
$${l} =(\lfloor (S+\log(P))/d\rfloor, \ldots, \lfloor (S+\log(P))/d \rfloor).$$ 
Thus the overall number of degrees of freedom is independent of $d$, since 
$$N \approx \prod_{j=1}^d 2^{(S+\log(P))/d} =2^{S+\log(P)}=P \cdot 2^S,$$ 
and the size of each subdomain is again approximately $2^S$ for all values of $P$. Furthermore we choose $\gamma=0.5$, $S=8$ and $q=2^{S-4}$ and consider only the $\omega_i$-weighted balanced methods. The resulting weak scaling behavior is shown in Figure \ref{dimvarq}.
\begin{figure}[ht]
\centering
	\includegraphics[width=0.82\textwidth,height=0.33\textheight]{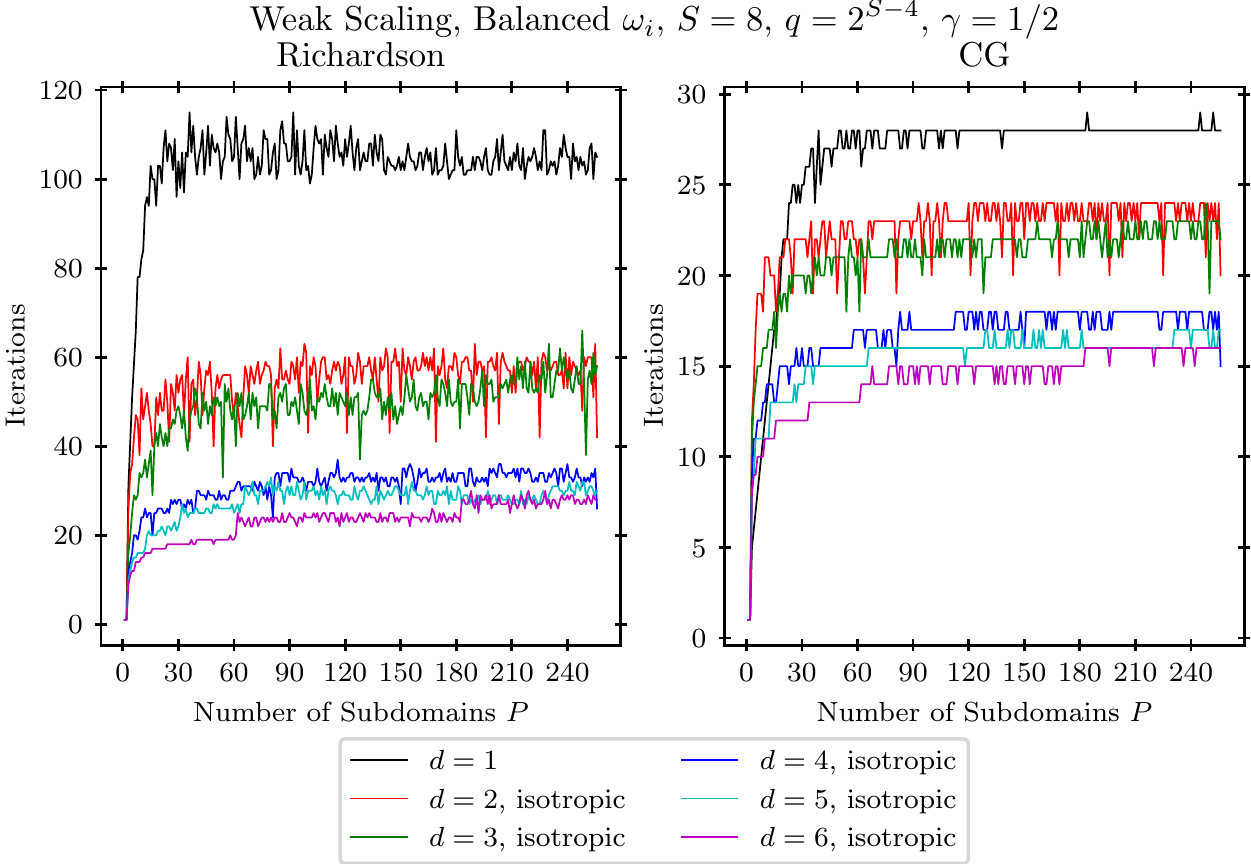}
	\caption{Weak scaling: Number of iterations versus number of subdomains for the $\omega_i$-weighted balanced two-level additive Schwarz/Richardson iteration (left) and the associated  preconditioned conjugate gradient method (right) for different dimensions $d$, $q=2^{S-4}$, $\gamma=0.5$.}
\label{dimvarq}
\end{figure}

We always obtain weak scaling behavior, where now the constant depends on $d$. But it improves for growing values of $d$ and we see that the one-dimensional case is indeed the most difficult one. This is due to the relative ''distance'' of
the fine scale to the coarse scale in the two-level domain decomposition method, which is largest for $d=1$  and decreases for larger values of $d$. Furthermore a stable limit of 26 and 16 iterations, respectively, is reached  for $d\geq 6$.
Such a behavior could be observed not only for the isotropic discretizations in $d$ dimensions but also for all the various anisotropic discretizations, which arise from (\ref{eq:combi}) in the sparse grid combination method. This becomes clear when we consider the simple case of the anisotropic discretization $l=(\lfloor S+\log(P)\rfloor,1,\ldots, 1)$ in $d$ dimensions: With our homogeneous Dirichlet boundary conditions we obtain the same fine scale finite difference discretization as for the case $d=1$ with differential operator $-\partial^2/\partial^2_{x_1} + 8 (d-1) \cdot I_{x_1}$. We then have an additional reaction term of size $8 (d-1),$ which merely improves the condition number compared to the purely one-dimensional Laplacian. Consequently, the one-dimensional convergence results impose an upper limit to the number of iterations needed for all the subproblems arising in the combination method.

Now let us shortly consider the strong scaling situation as well. There we have $N=2^L$, where the size of each subdomain is $2^L /P$, i.e. it decreases with growing values of $P$. Moreover values of $P$ larger than $2^L$ are not feasible. We consider again the one-dimensional situation, set $q=2^{L-12}$, $\gamma=0.5$, and vary the number $P$ of subdomains. The resulting strong scaling behavior is shown for $L=16,18,20$ in Figure \ref{strongfix}.
\begin{figure}[ht]
\centering
	\includegraphics[width=0.82\textwidth,height=0.30\textheight]{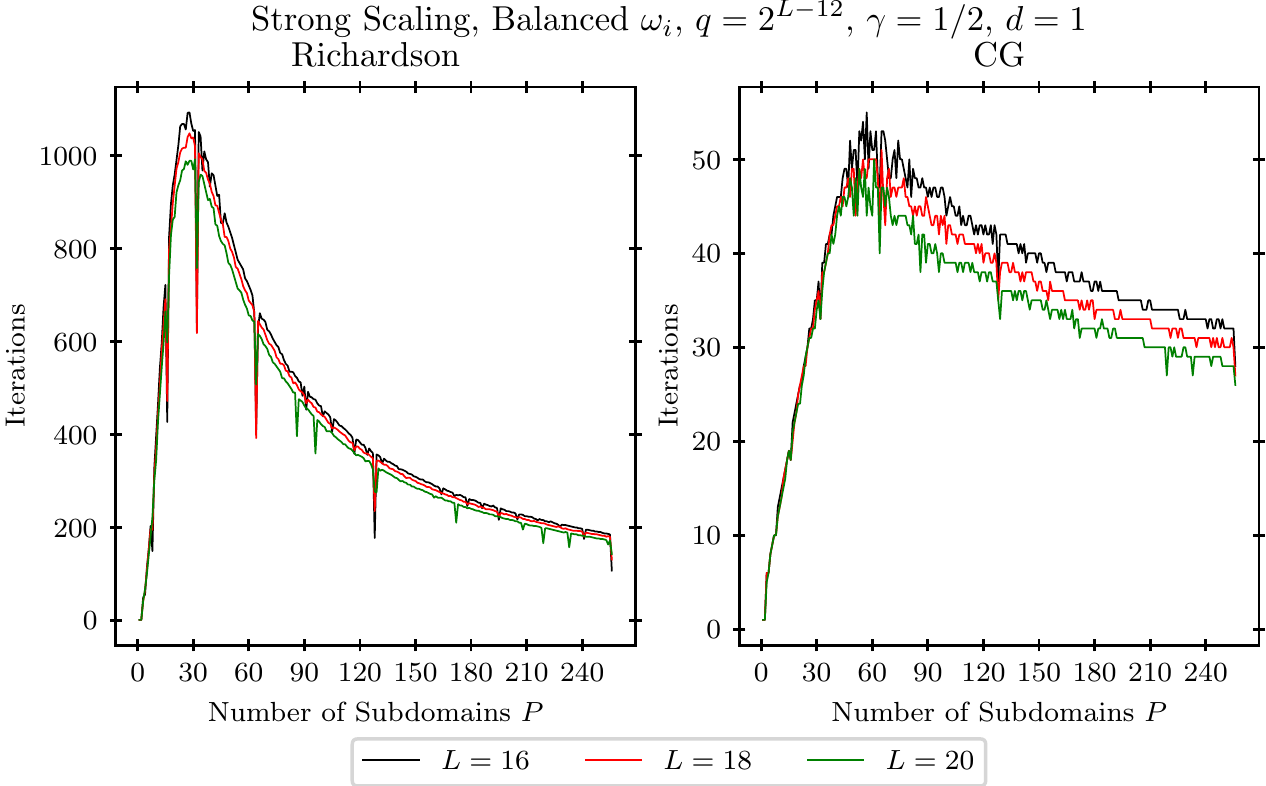}
	\caption{Strong scaling: Number of iterations versus number of subdomains for the $\omega_i$-weighted balanced two-level additive Schwarz/Richardson iteration (left) and the associated preconditioned conjugate gradient method (right), $d=1$, ${L=16,18,20}$, $q=2^{L-12}$, $\gamma=0.5$.}
\label{strongfix}
\end{figure}

We see that the necessary number of iterations first grows with rising values of $P$ and then, after its peak, steadily declines, as is expected. This is due to the fact that now the size of each subproblem shrinks with rising $P$, whereas the size $P \cdot 2^{L-12}$ of the coarse scale problem grows linearly with $P$. There is not much of a difference for the three curves $L=16$, $L=18$ and $L=20$ since  for each $P$ the coarse scale problem has the same relative distance to the fine grid for all values of $L$, i.e. $ 
2^{16}/(2^{16-12} \cdot P) = 2^{18}/(2^{18-12} \cdot P)=2^{20}/(2^{20-12} \cdot P)$. Note here that the downward peaks of the curves in Figure \ref{strongfix} correspond to the situation
where all subdomains are perfectly balanced, i.e. where $N \mod P=0$.  
Furthermore the parameter $q$ was chosen such that the case $P=256$ for $L=16,18,20$ in Figure \ref{strongfix} results in exactly the same situation as the case $P=256$ for $S=8,10,12$ in Figure \ref{varq}, respectively.

We finally consider strong scaling for the situation where $q= 2^{\lfloor \log(N/P)\rfloor-4}$, i.e. where the choice of $q$ is $P$-dependent and is chosen such that the relative distance between the coarse and the fine grid is roughly four levels for all $P$, as was the case in the weak scaling experiments in Figure \ref{varq}. The results are given in Figure \ref{strongvar}. Due to the involved round off there are now jumps at the points where $\log_2(P)$ is an integer. 
However we observe the expected steady decline of the necessary number of iterations in between these jump points. Note that, for this choice of $q$, the necessary numbers of iterations for the end points of these intervals are approximately the same.
\begin{figure}[ht]
\centering
	\includegraphics[width=0.82\textwidth,height=0.30\textheight]{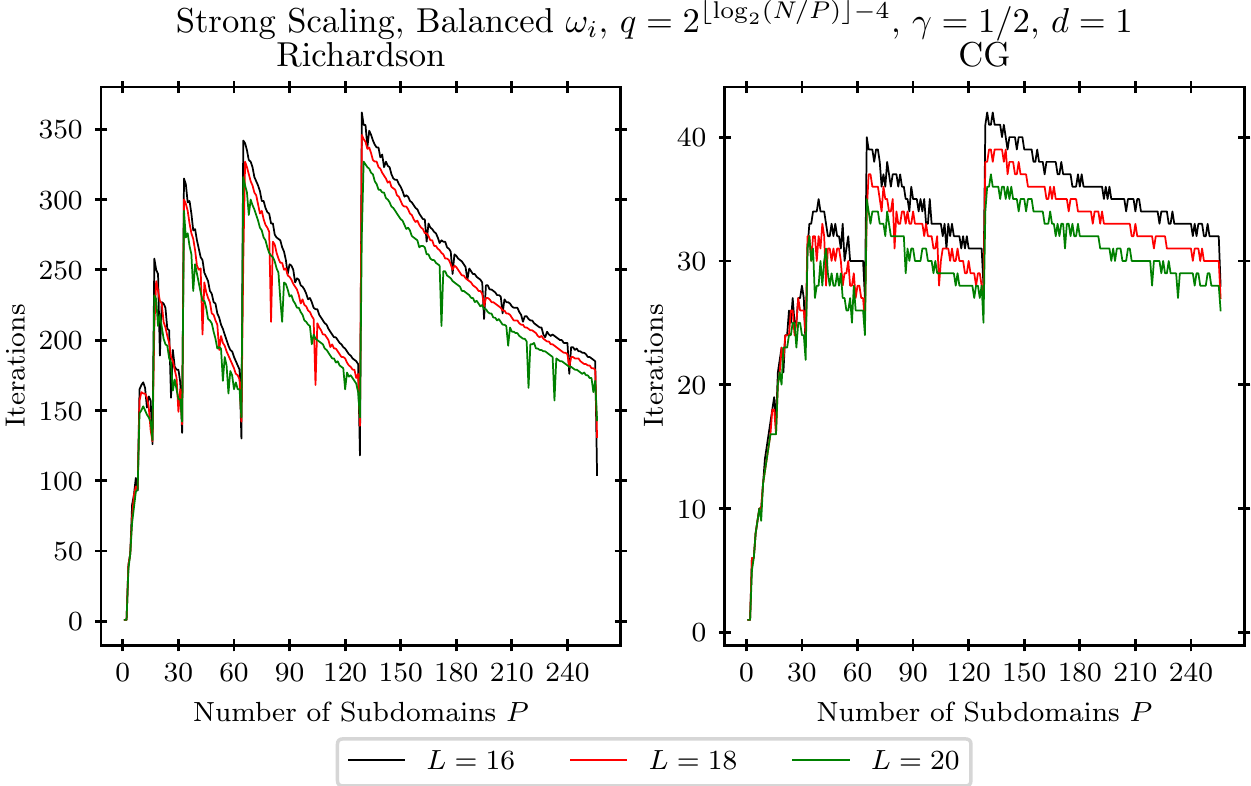}
	\caption{Strong scaling: Number of iterations versus number of subdomains for the $\omega_i$-weighted balanced two-level additive Schwarz (left) and the associated preconditioned conjugate gradient (right) method, $d=1$, $L=16,18,20$, ${q= 2^{\lfloor \log (N/P) \rfloor -4}}$, $\gamma=0.5$.}
\label{strongvar}
\end{figure}

\subsection{Fault tolerance results}\label{subft}
We now consider the results of our fault tolerance experiments. As already mentioned, we assume fault locality (in space and time) which is often done in the literature for analysis purposes. 
As a model for faults we employ the Bernoulli distribution to inject faults into our compute system, as described in detail in Subsection \ref{ssfault}.
This directly gives independent occurrences of faults per cycle. 
Furthermore we assume that a faulty processor will be instantly repaired and is available again in the next iteration, i.e. a processor only stays faulty for the iteration in which the fault was detected. We exploit a sufficiently large $\gamma$ of overlap and thus {\em redundancy} on the stored data in our algorithm to be able to recover lost data due to a faulty processor.
This assumption on the model adheres to the prerequisites of our Corollary \ref{cor1} and our Theorem \ref{theo2} and thus enables an estimate for the expectation of the squared error for our specific additive, two-level domain decomposition method based on space filling curves in the faulty situation. Now we consider the $\omega_i$-weighted balanced two-level additive Schwarz/Richardson iteration with damping parameter $\xi^*$ under faults. Moreover we are interested in the behavior of the associated preconditioned  conjugate gradient methods
under faults, even though we do not have a proper convergence theory for this case as for the simpler linear iteration. Note to this end that the energy norm based on the finite element matrix, for which our theory is valid, and the energy norm based on the finite difference matrix, which we use in practice, are, up to a scaling of $h^{-d/2}=2^{(dL)/2}$, spectrally equivalent due to the well-known $L_2$-stability of the piecewise $d$-linear finite element basis.

First we study the convergence behavior of these different methods. To this end, we consider the case $d=1$, we set $q=16$, $S=8$ and $P=100$. Thus the size $N$ of the overall problem is fixed to $100 \cdot 2^8$ with a size of $2^8$ for each subproblem. Now we study the error reduction properties of our fault-tolerant algorithms. To this end we consider, besides the case of no fault at all, the cases where the faults are randomly generated in each cycle and for each processor according to the Bernoulli distribution with different fault-probabilities $p_\text{fault}=0.01,0.02, 0.05, 0.1$ and $0.2$. This way, on average, one, two, five, ten and twenty processors fail in each iteration. Since the faults are randomly distributed, we run the algorithms ten times and show the resulting convergence curves for each run in transparent colors. We additionally show the averaged convergence curve in bold color. 
Note here also that, if the overlap parameter $\gamma$ is too small compared to the fault rate $p_\text{fault}$, a proper fault recovery might no longer be possible and, depending on the 
specific realization of random picks of faulty processors,
the respective run of the algorithm may no longer be able to recover and may fail. These failed runs are discarded in the computation of the average convergence rate.
For the overlap parameter $\gamma=0.5, 1,1.5,2$, the results for the $\omega_i$-weighted balanced two-level Richardson/Schwarz iteration and the correspondingly preconditioned conjugate gradient method are shown in Figure \ref{ft_bal}.  
\begin{figure}[htb!]
\centering
	\includegraphics[width=0.82\textwidth,height=0.82\textheight]{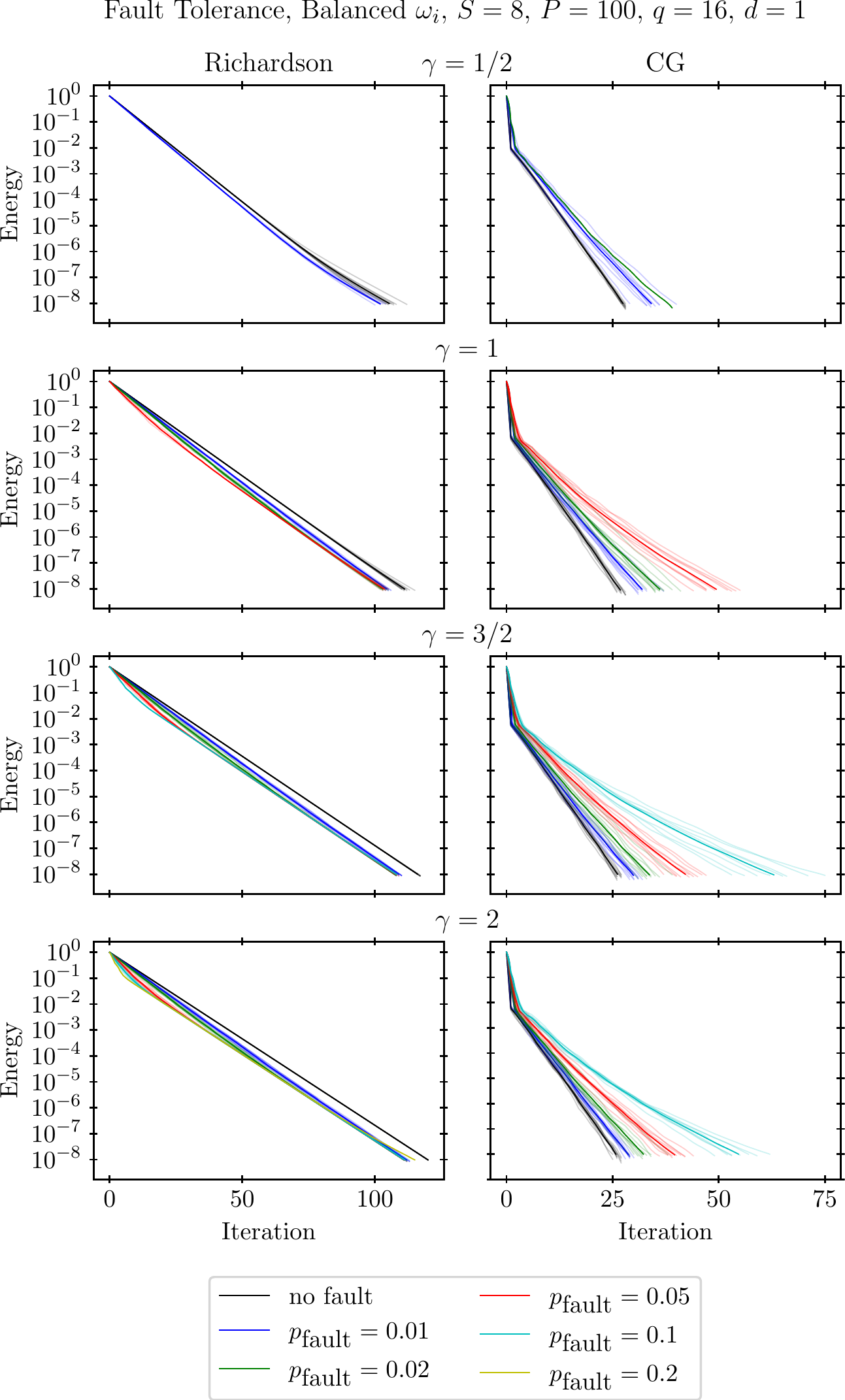}
	\caption{Number of iterations (single runs and mean) of the fault-tolerant version of the $\omega_i$-weighted balanced two-level additive Schwarz/Richardson iteration (left) and the associated preconditioned conjugate gradient method (right) for different fault probabilities and different overlaps $\gamma=0.5,1, 1.5, 2, q=16, d=1$.}
\label{ft_bal}
\end{figure}

We see that in the $\omega_i$-weighted balanced Richardson case there is a slight preasymptotic reduction of the initial error in the first iterations, which shortens but is more profound the larger the values of $p_\text{fault}$ and of $\gamma$ are. After that an asymptotic error reduction can be observed  
which is nearly the same as for the non-faulty case. A close inspection reveals that, depending on $p_\text{fault}$, this asymptotic rate becomes slightly larger in the end compared to the non-faulty case. This is however to be expected from our theory, compare also the bound for the (asymptotic) rate in Theorem \ref{theo2}.
The preasymptotic behavior stems from the fact that a moderate number of faults indeed restores a certain smoothing effect in the preasymptotic regime. 
This observed (small) preasymptotic smoothing effect in the faulty cases can best be explained by looking at all the fine grid points that belong to a failing processor. 
For example, the overlap $\gamma=2$ and the associated weighting $\omega_i I=\frac{1}{2\gamma+1}I=\frac{1}{5}I$ yields an (asymptotically) optimal damping parameter of $\xi_\textrm{opt}=2/(\lambda_\textrm{min} + \lambda_\textrm{max})\approx 1.86$ for the balanced preconditioner. Recall that $\omega_0$ is chosen as one here. 
Then the effective preconditioner, including damping, can be seen to be
\begin{equation}\label{wda}
\begin{aligned}
 \xi_\textrm{opt}C^{-1}_{(2),D, bal} &= \xi_\textrm{opt}(G^T C^{-1}_{(1), D} G + F) = \xi_\textrm{opt} (\frac{1}{5} G^T  C^{-1}_{(1)} G + F) \\
	&= \xi_\textrm{opt}\frac{1}{5} G^T  C^{-1}_{(1)} G + \xi_\textrm{opt}F  \approx 0.37 ~ G^T  C^{-1}_{(1)} G + 1.86 ~ F.
\end{aligned}
\end{equation}
On the other hand, an intuitive way of weighting the respective fine grid corrections in the iteration would be to average just over all fine grid contributions per point, which resembles a certain smoothing of higher frequency error components. For our example, this would yield a weighting factor of $\frac{1}{5}$, i.e.\ the associated preconditioner would then be $\frac{1}{5} G^T  C^{-1}_{(1)} G + F$.
A comparison of this preconditioner to the above effective preconditioner (\ref{wda}), which involves the asymptotically optimal dampening parameter $\xi_\textrm{opt}$, shows that 
the optimally damped version shifts the balance between fine and coarse level more towards the coarse side to achieve optimal convergence in the asymptotic regime. 
But this in turn yields worse smoothing properties in the preasymptotic regime compared to the dampening due to averaging, since the factor $0.37$ of the fine grid contribution does not match the fact that each point is overlapped by five fine subdomains.
However, if we now consider the faulty cases, the occurrence of a few faults brings the expected factor due to averaging closer to this factor $0.37$. 
For example, if one of five neighboring processors fails, one would expect an averaging factor of $\frac{1}{4}$, since now only $4=5-1$ processors contribute to the correction at all points belonging to the failing processor. 
In the case of two neighboring faults the expected weighting factor due to averaging of $\frac 1 3 $ in the points belonging to both failing processors  is now even closer to our asymptotically optimal fine grid weighting of $0.37$.
Therefore we see an improved initial smoothing behavior for a moderate number of neighboring faults compared to the non-faulty case simply due to the fact that 
the effective weighting parameter for the fine grid, i.e.\ $\xi_\textrm{opt}\cdot\omega_i$, 
more closely resembles 
the intuitive weighting factor due to averaging when a small number of neighboring faults is present.
Note however that, for larger values of $p_\text{fault}$, the faulty cases still produce asymptotic rates that are slightly worse than those of the non-faulty case.  
The somewhat better convergence of the faulty cases is therefore only a result of the better initial smoothing.
Hence one could potentially improve the overall iteration scheme by starting with the dampening parameter due to averaging and then (gradually) switching to the optimal damping parameter after a certain number of iterations. Further experiments showed that the non-faulty case $p_\textrm{fault} = 0$ then indeed exhibits better overall average convergence than the faulty cases and that the convergence for the faulty cases gradually deteriorates with rising values of $p_\textrm{fault}$. However, a proper switching between these two damping parameters is not straightforward and the convergence analysis of such a linear iterative method with dynamically changing damping is quite involved. This is future work.

For the associated conjugate gradient method the convergence behavior is different: We see that there is a sudden reduction of the initial
error in the first few iterations by a factor of about 100, which even grows slightly with rising values of $\gamma$. This holds for all values of $p_\text{fault}$. Here indeed the same smoothing effect due to averaging as in the faulty Richardson case comes into play. However, instead of the fixed $\xi_{opt}$, the conjugate gradient iteration determines the optimal step size in the preconditioned search direction in the first iteration step. This leads to a substantial smoothing effect 
in the first (few) iteration(s) and thus to the steep drop of the error as seen in Figure \ref{ft_bal} (right).
After that the asymptotic error reduction rate gradually sets in. It becomes monotonously worse for rising values $p_\text{fault}$, as is intuitively to be expected. 
Again, the necessary number of iterations is substantially smaller due to the $\kappa$-versus-$\sqrt \kappa$ effect of the conjugate gradient method compared to the simple linear Richardson iteration. The conjugate gradient approach needs approximately a quarter of the iterations of the Richardson approach in the non-faulty case. 
But it can be seen that the conjugate gradient method tends to lose its advantage for larger values of $p_\text{fault}$. For example, for $\gamma=2$ and $p_\text{fault}=0.1$, it merely needs approximately half the number of iterations compared to the Richardson approach.
Moreover recall that runs, for which the recovery failed 
due to a larger number of faults than can be repaired with the given values of $\gamma$,
are discarded in the computation of the average convergence rate.
This is the reason why for example for $\gamma=1/2$ in Figure \ref{ft_bal} only averaged convergence plots for $p_\text{fault}=0,0.01, 0.02$ could be displayed, whereas the curves for the larger values of $p_\text{fault}$ are missing. 
We furthermore see that for the larger overlap $\gamma=2$ also faults with larger fault probabilities of $p_\text{fault}$ up to 0.1 could be recovered from and are treated successfully. We finally see that, if $p_\text{fault}$ is larger than 0, we gradually obtain better results for larger overlap
values $\gamma$ as long as the runs are successful at all. For example, the number of iterations improves with $p_\text{fault}=0.05$  from 50 for $\gamma=1$ to 43 for $\gamma=1.5$ and to 37 for $\gamma=2$. In the case  $p_\text{fault}=0$ such a gain for rising values of $\gamma$ is still there, but it is substantially smaller. 
In further experiments with larger values
of $P$, like for example $P = 200$, we could observe analogous results with a further improvement in convergence for the faulty situations. This is due to the larger subdomain sizes and the exact subdomain solves, while now larger costs are involved of course. 

Altogether, we see that the occurrence of a moderate number of faults does not have much influence on the convergence rate of the fault-tolerant version of the balanced two-level additive Richardson iteration, as already predicted by our convergence theory in Subsection \ref{sec:thft}. We observe that, due to the error reduction in the preasymptotic regime, the necessary number of iterations even decreases compared to the non-faulty case.
Analogous observations could be made for the plain Richardson iteration.
Note that a similar behavior was already present in the numerical experiments of \cite{Griebel.Oswald:2019}, Figure 3, for the case of a failure rate of eight percent using a Weibull distribution.
For one, two, five, ten and twenty failing processors in each iteration, which is indeed already an extremely large number of failures regarding the short computation time and the employed number of 100 processors, the influence on the number of iterations 
is deteriorating slightly and is thus only marginal.

This is somewhat different when considering the associated preconditioned conjugate gradient method. Its convergence in the non-faulty case is substantially faster (i.e. it is dependent on $\sqrt \kappa$ instead of just $\kappa$), however we no longer have a sound convergence theory in the faulty situation. Indeed, the behavior of the conjugate gradient approach is much more sensitive to the occurrence of faults, as can be seen on the right side of the Figure  \ref{ft_bal}. Now the deviation from the non-faulty convergence rate grows substantially with rising values of $p_\text{fault}$. For example, for $\gamma=2$, the necessary number of iterations grows in the $\omega_i$-weighted balanced conjugate gradient case with $p_\text{fault}=0,0.01,0.02,0.05, 0.1$ with $25,28,31,37,54$. Note furthermore that this dependence on the fault probability decreases with rising values of the overlap parameter $\gamma$.
But, all in all, the fault-tolerant version of our conjugate gradient method still exhibits excellent convergence properties in practice.

Next let us consider the convergence behavior of our $\omega_i$-weighted balanced methods for a higher dimensional example. Here we focus on $d=6$ and again set $P=100$ and $S=8$. 
Thus the size $N$ of the overall problem is fixed to $100 \cdot 2^8$ with a size of $2^8$ for each subproblem. Now we study the error reduction property of our fault-tolerant algorithms and consider again, besides the case of no fault at all, the cases where the faults are randomly generated in each cycle for each processor according to a Bernoulli distribution. We again employ the different fault-probabilities $p_\text{fault}=0.01,0.02, 0.05, 0.1, 0.2$ where, on average, one, two, five, ten and twenty  processors fail in each cycle. For the overlap parameter $\gamma=0.5,1,1.5,2$, the results for the  $\omega_i$-weighted balanced two-level Richardson/Schwarz iteration and the correspondingly preconditioned conjugate gradient method are shown in Figure \ref{ft_add_bal_d=6}.    
\begin{figure}[htb!]
\centering
	\includegraphics[width=0.82\textwidth,height=0.82\textheight]{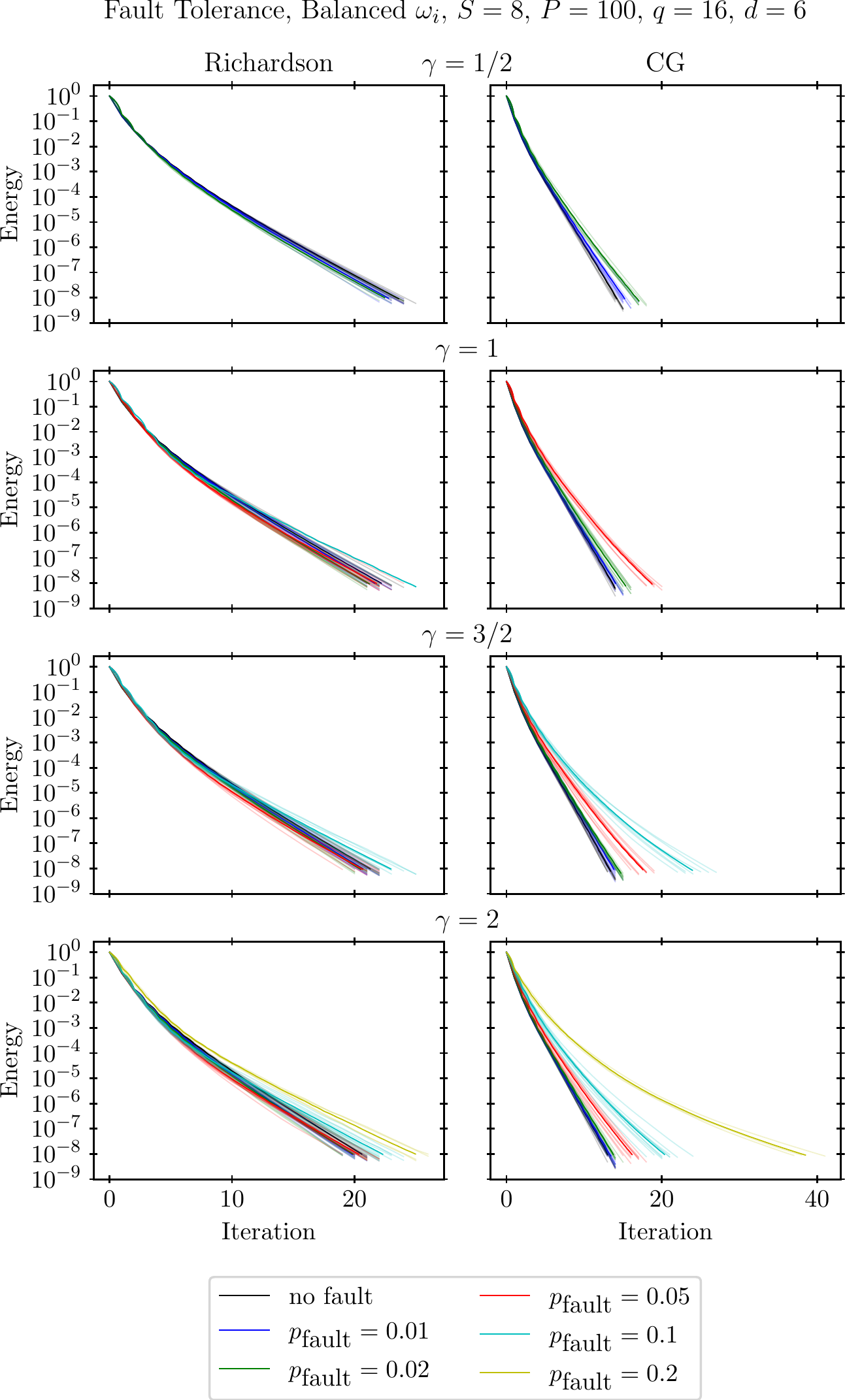}
	\caption{Number of iterations (single runs and mean) of the fault-tolerant version of the  $\omega_i$-weighted balanced two-level additive Schwarz/Richardson iteration (left) and the associated preconditioned conjugate gradient method (right) for different fault probabilities and different overlaps $\gamma=0.5, 1, 1.5, 2, q=16, d=6$.}
\label{ft_add_bal_d=6}
\end{figure}

We now observe that, for the Richardson case, the preasymptotic error reduction is substantially longer but smoothly changes to the asymptotic convergence, which is in contrast to the one-dimensional case. Furthermore we see that we obtain nearly the same convergence curves for all the different values of $p_\text{fault}$ for a fixed value of $\gamma$. Moreover we now find that the number of iterations stays the same or slightly worsens for larger values of $p_\text{fault}$  if $\gamma$ increases.
Again we observe that there are no results for large values of $p_\text{fault}$  but small overlap $\gamma$, which stems from the failure of the recovery process due to insufficient redundancy.
Consequently, $\gamma$ should be chosen small but nevertheless still sufficiently large to guarantee the necessary redundancy with respect to the fault rate $p_\text{fault}$. 
This holds for the associated conjugate gradient method as well. Moreover, for $p_\text{fault}=0$, the number of iterations is not very sensitive to $\gamma$, it even marginally improves. For larger values of $p_\text{fault}$ it monotonously improves with rising values of $\gamma$. 
The corresponding number of iterations is still smaller compared to the Richardson case, possibly due to the $\sqrt{\kappa}$-versus-$\kappa$ effect, but the advantage of the conjugate gradient method now tends to be not as profound anymore. Especially for larger fault rates $p_\text{fault}$ and larger values of $\gamma$ the conjugate gradient methods tends to lose its  superiority over the Richardson iteration and can even be inferior, as the example $\gamma=2, p_\text{fault}=0.2$ shows: Now the Richardson approach only needs 25 iterations, whereas the corresponding conjugate gradient method needs 40 iterations. In any case, depending on the choice of the value of $p_\text{fault}$, the value of $\gamma$ must be large enough to allow for a proper recovery after faults.

Finally, we consider the weak scaleup behavior of our $\omega_i$-weighted balanced algorithms under faults. 
We select a substantial overlap of $\gamma=3$ and set $d=1,q=16,S=8$. Moreover we allow the values $p_\text{fault} =0, 0.01, 0.02,0.05, 0.1,0.2$. This corresponds to $p_\text{fault} \cdot P$ faults in each cycle on average for varying numbers of $P$. The resulting number of iterations are shown in Figure \ref{weakscale_ft_it_gamma=3}. The results for the average convergence rate $\rho^{ave}$ and the asymptotic convergence rate $\rho^{asy}$ are shown in Figure \ref{weakscale_ft_rate_gamma=3}.
\begin{figure}[htb!]
\centering
	\includegraphics[width=0.82\textwidth,height=0.36\textheight]{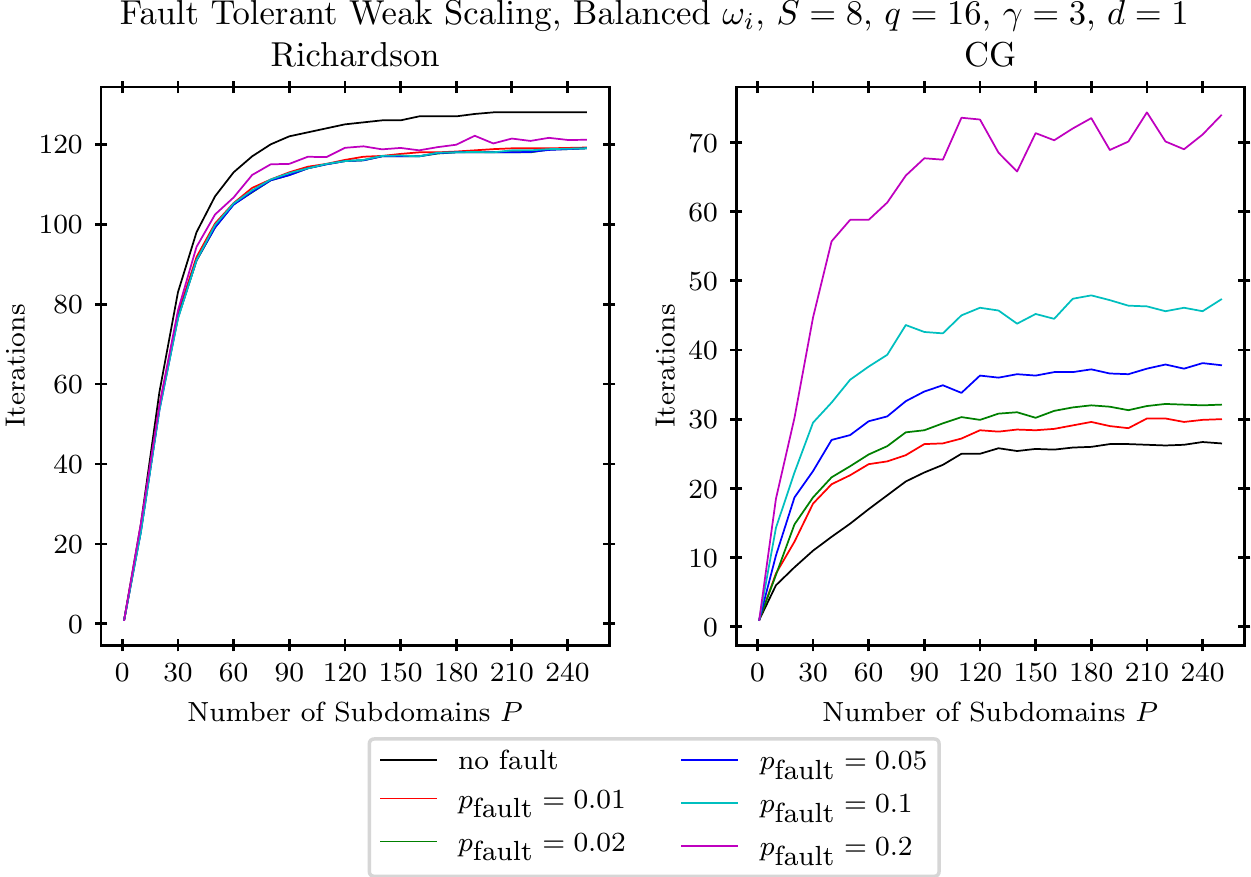}
	\caption{Number of iterations (mean) of the fault-tolerant version of the $\omega_i$-weighted balanced two-level additive Schwarz/Richardson iteration (left) and the associated conjugate gradient method (right) for different fault probabilities, $\gamma=3,S=8,q=16, d=1$.}
\label{weakscale_ft_it_gamma=3}
\end{figure}
\begin{figure}[htb!]
\centering
	\includegraphics[width=0.82\textwidth,height=0.36\textheight]{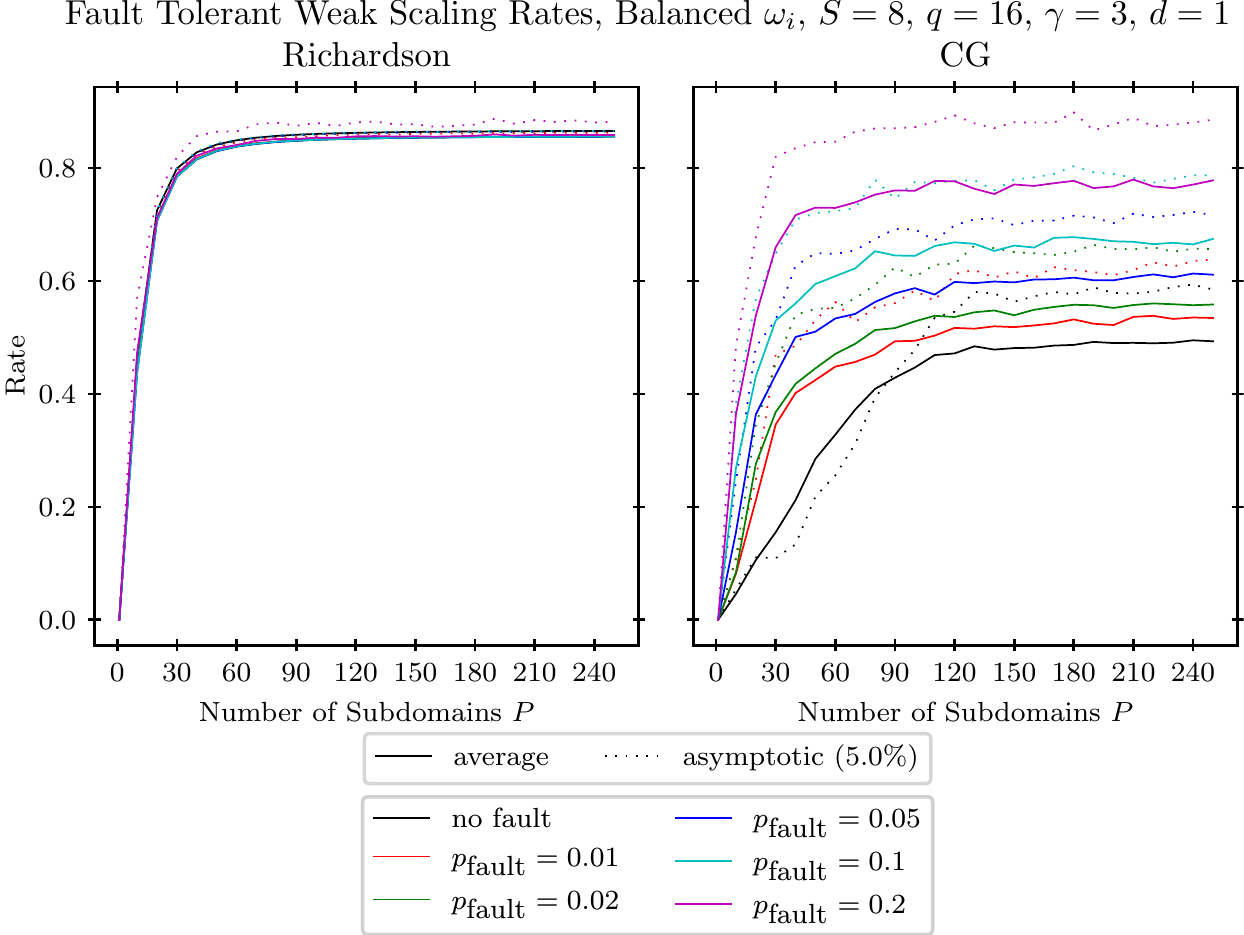}
	\caption{Average and asymptotic rates (mean) of the fault-tolerant version of the $\omega_i$-weighted balanced two-level additive Schwarz/Richardson iteration (left) and the associated conjugate gradient method (right) for different fault probabilities, ${\gamma=3,}$ $S=8,q=16, d=1$.}
\label{weakscale_ft_rate_gamma=3}
\end{figure}

We see that the weak scaling constant associated to the iterations, i.e. the respective number of iterations for rising numbers of $P$,
first (for $p_\text{fault}=0.01,0.02,0.05,0.1$) marginally improves in the Richardson case and then slightly deteriorates again ($p_\text{fault}=0.2$), which can be explained by the preasymptotic behavior.
In the conjugate gradient case it monot\-o\-nously deteriorates for rising values of $p_\text{fault}$ as intuitively expected, but the corresponding iteration numbers are much better than those for the Richardson iteration. An analogous behavior can be observed for the average convergence rate $\rho^{ave}$.
In addition the corresponding weak scaling constant for the asymptotic convergence rate $\rho^{asy}$ behaves similarly and is only worse by a value of about 0.02 and about 0.1 than for $\rho^{ave}$ in the Richardson and conjugate gradient case, respectively.

At last, for the six-dimensional case, the weak scaleup behavior of our $\omega_i$-weighted balanced algorithms under faults, i.e. the number of iterations and the average and the asymptotic convergence rate, are shown in the Figures \ref{weakscale_ft_it_gamma=3_d=6} and \ref{weakscale_ft_rate_gamma=3_d=6}.
\begin{figure}[htb!]
\centering
	\includegraphics[width=0.82\textwidth,height=0.36\textheight]{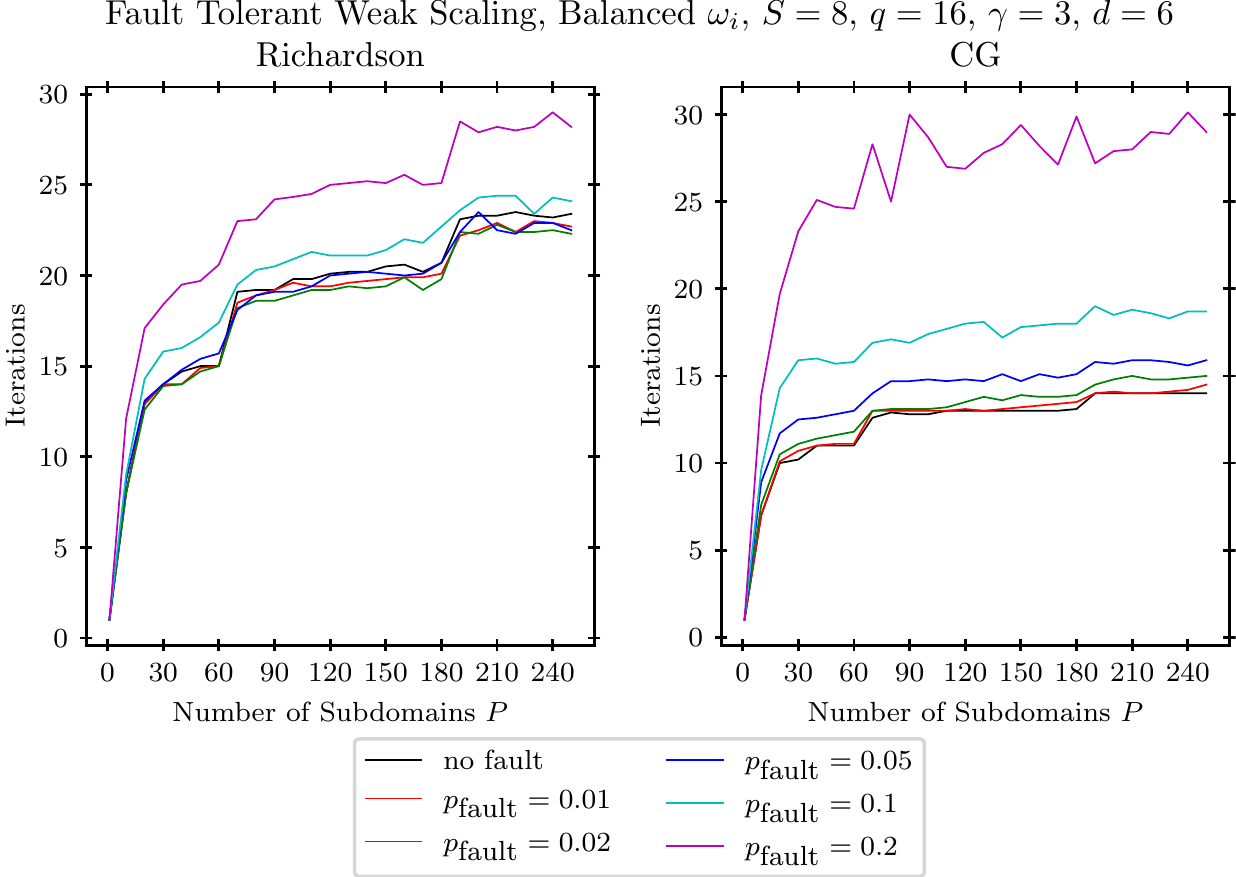}
	\caption{Number of iterations (mean) of the fault-tolerant version of the $\omega_i$-weighted balanced two-level additive Schwarz/Richardson iteration (left) and the associated conjugate gradient method (right) for different fault probabilities, ${\gamma=3,}$ $S=8,q=16, d=6$.}
\label{weakscale_ft_it_gamma=3_d=6}
\end{figure}
\begin{figure}[htb!]
\centering
	\includegraphics[width=0.82\textwidth,height=0.36\textheight]{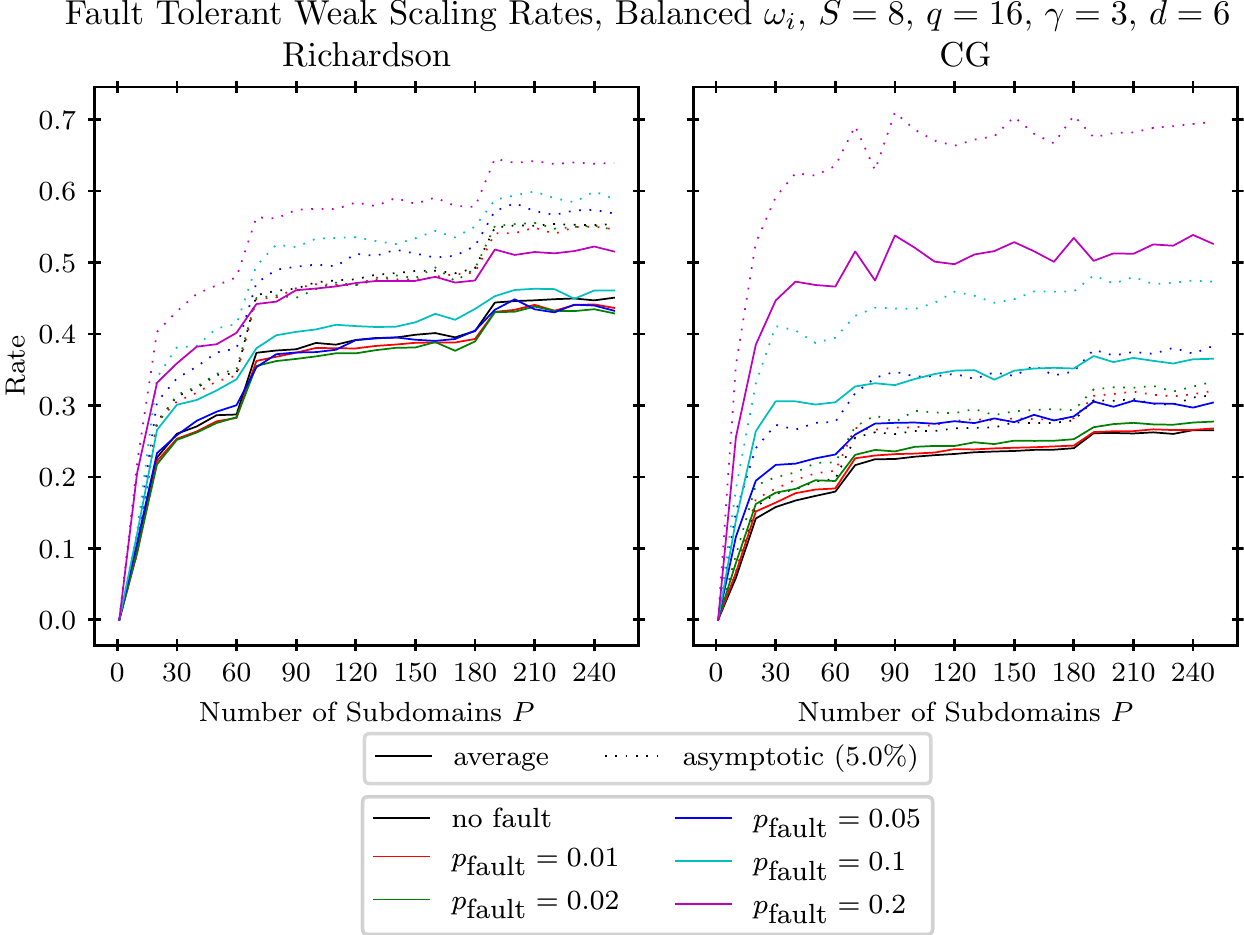}
	\caption{Average rate and asymptotic rate (mean) of the fault-tolerant version of the $\omega_i$-weighted balanced two-level additive Schwarz/Richardson iteration (left) and the associated conjugate gradient method (right) for different fault probabilities, $\gamma=3,S=8,q=16, d=6$.}
\label{weakscale_ft_rate_gamma=3_d=6}
\end{figure}

We see that the algorithm exhibits good weak scaling properties also for the six-dimensional case. For rising values of $P$, the number of iterations stays at around 23 for the balanced Richardson iteration and is nearly independent of the fault rate (except of the case $p_\text{fault}=0.2 $ which needs 28 iterations).
Weak scaling can also be observed for the corresponding conjugate gradient method.
Now however the weak scaling constant is dependent on the fault rate $p_\text{fault}$ and increases for rising numbers of faults, albeit with still less iterations than in the Richardson case (except for the case $p_\text{fault}=0.2 $). The situation is the same for the average rate and the asymptotic rate is basically the same. 

Altogether we conclude that, for a moderate number of faults and a sufficiently large choice of $\gamma$, our fault-tolerant algorithms work and scale well with rising numbers of $P$ and associated rising number $p_\text{fault} \cdot P$ of faults. Thus we obtain fault-tolerance also in the weak scaling regime.

\section{Discussion and conclusion}\label{sec:conclude}
In the present article we proposed and studied a simple additive two-level overlapping domain decomposition method for $d$-dimensional elliptic PDEs, which relies on a space-filling curve approach to create equally sized subproblems and to determine an overlap based on the one-dimensional ordering of the space-filling curve.  Moreover we employed agglomeration and a purely algebraic Galerkin discretization in the construction of the coarse problem to avoid $d$-dimensional geometric information. This resulted in a parallel subspace correction iteration and an associated parallel preconditioner for the conjugate gradient method which can be employed for any value of $P$ and in any dimension $d$. 
It exhibits the same asymptotic convergence rate as a conventional geometric domain decomposition method, i.e.
 $\kappa =O(1+\frac H \delta)$ with $\delta= c h$, compare  \cite[Theorem 3.13]{toselliwidlund04}  and (\ref{Hh}), but allows for more flexibility and simplicity with respect to $N$, $P$ and $d$ due to the space-filling curve construction.
 Moreover it turned out that its balanced version according to (\ref{bala}), which, compared to the plain version, involves two additional matrix-vector-multiplications, three additional vector-vector-additions and one additional solve of the coarse scale problem, was clearly superior and is therefore the recommended approach. 
To gain fault-tolerance
we stored on each processor, besides the data of the associated subproblem, a copy of the coarse problem and also the data of a fixed number of (partial) neighboring subproblems with respect to the one-dimensional ordering of the subproblems induced by the space-filling curve. This redundancy then allows to restore necessary data if 
  processors fail during the computation and to therefore restore convergence in the case of faults. We showed theoretically and practically
that the convergence rate of such a linear iteration method for a moderate number of faults stays approximately the same in expectation and only its order constant deteriorates slightly due to the faults. This was also observed in practice for the associated preconditioned conjugate gradient method. Moreover weak scaling could be achieved in any situation.
Altogether, we obtained a fault-tolerant, parallel and efficient domain decomposition method based on space-filling curves, which is especially suited for higher-dimensional elliptic problems. 
	
Note here that we employed an exact subdomain solver in our domain decomposition method. It can be replaced by a faster but iterative and thus somewhat inexact method, like for example algebraic multigrid. This needs to be explored in more detail in the future.

Moreover our algorithmic approach possesses a range of parameters, which influence each other and need to be chosen properly. The  determination of an optimal set of parameters is a difficult task. Nevertheless some general guidelines could be derived: One key parameter is the number of degrees of freedom for the coarse grid. Here it became clear that the coarse grid size should be chosen relative to the fine grid size. The distance between the two levels controls the balance between the required number of iterations on the one hand and the cost to solve the coarse system in each iteration on the other hand. Fortunately, as the dimension increases, the size of the coarse grid can be decreased, which reduces the cost for each iteration. 
Furthermore, as seen from Lemma \ref{lemweig}, it is sensible to choose another key parameter of our scheme, namely the overlap $\gamma$, as an integer multiple of $\frac 1 2$, i.e. $\gamma=\frac 1 2 n, n \in \mathbb{N}_+$. This way, the different possible scalings via $D_i$ and $\omega_i$ coincide and boil down to a constant scaling with $ 1/ (2 \gamma+1)$ and symmetry of the associated operator is guaranteed, at least for the Laplacian and diffusion operators with constant coefficients. Furthermore the choice of $\gamma$ as an integer multiple of $\frac 1 2 $ is also natural for a uniform redundancy of stored data, which is advantageous for the recovery and reconstruction in the case of faults.
Moreover we learned from our convergence studies and from our fault tolerance experiments that $\gamma$ can be chosen to be quite small for any real world setting. For example, Figures \ref{ft_bal} and \ref{weakscale_ft_rate_gamma=3_d=6} showed that $\gamma=0.5$ or $\gamma=1$ still produce excellent convergence for both considered types of solvers for fault rates of up to 2\% and 5\% per iteration/cycle, respectively. These fault rates are indeed way beyond those of any existing high performance compute system, where merely one or very few faults currently occur on average each day. They are also beyond the expected fault rates for future exascale compute systems.
For practical purposes this means that we will see no deterioration of the convergence rates at all, as long as the data recovery process is executed properly. This is guaranteed due to the overlap and the redundant parallel storage of current data and local parts of the iterate 
which allows the algorithm to recover after a fault occurs. 
In this respect, iterative methods are fault-forgiving and the iteration can continue dynamically.
Moreover our aim was to derive a fault-tolerant parallel domain decomposition solver, which can be later employed for the subproblems in the combination method. There we encounter short run times, since the numbers of degrees of freedom of all the arising subproblems are similar to the one-dimensional case, i.e. of size $\approx 2^L$ only, compare (\ref{eq:combi}).
But this is good news: The fault-repair mechanism is rarely needed for real-life fault probabilities within the short run times of our solver for the small subproblems arising in the combination method.

We next intend to employ our domain decomposition method as the inner, parallel and fault-tolerant solver for the subproblems arising in the sparse grid combination method. Let us recall the sparse grid combination formula from (\ref{eq:combi}), i.e.
\begin{equation}\label{eq:combi2}
	u({ x}) \approx u^{(c)}_{L}({ x}):=
	\sum_{i=0}^{d-1} (-1)^i \left( \begin{array}{c} {d-1}\\{i}\end{array}\right) \sum_{|{l}|_1=L+(d-1)-i} u_{l}({x}).
\end{equation}
Here the different subproblems can be computed in parallel independently of each other.
Moreover, for layer $i$, we have 
$$ \left(  \begin{array}{c} {L+d-2-i}\\{d-1}\end{array}\right)$$
	different subproblems, where each subproblem has approximately the same number $N(d,i)= \prod_{j=1}^d (2^{l_j}+1)=O(2^{L+d-1-i})$ of degrees of freedom. We now employ an additional step of parallelization by means of the domain decomposition treatment of each of these subproblems. To this end, we use 
\begin{equation}\label{Pchoice}
P:= \hat P \cdot 2^{d-1-i}
\end{equation}
subdomains and thus processors for each subproblem on layer $i$. This choice of a $d$- and $i$-dependent $P$ via a universal $\hat P$ in (\ref{Pchoice}) results in a subdomain size of roughly $N(d,i)/P = 2^{L+d-1-i}/P =  2^{L+d-1-i}/(\hat P \cdot 2^{d-1-i}) = 2^L/\hat P$, which is independent of $d$ and $i$. Then, in our elliptic situation and except for the coarse scale problems, only small subdomain problems appear. But, depending on $d, L$ and $\hat P$, there can be millions of these subdomains for the combination method. To be precise,  the amount of fine level subdomains for the overall set of subproblems in the combination method (\ref{eq:combi}) is
$$\frac {\hat P}{(d-1)!} \sum_{k=0}^{d-1} \left( 2^{d-1-k} \prod_{i=1}^{d-1} (L+d-1-k-i)\right),
$$
where each occurring subdomain problem has approximately equal size $2^L/\hat P$.

As an example, consider the goal of an overall discretization error of $10^{-12}$. Then ${h\approx 10^{-6}}$ and $L=20$. Moreover we choose $\hat P=2^{10}$. For this case, Table \ref{NSPr} gives the number of subdomain problems arising in the combination method if our domain decomposition approach is employed with $P$ given by (\ref{Pchoice}) for each subproblem of layer $i$.
\begin{table}[ht]
\centering
	\begin{tabular}{c|c|c|c|c|c|c}
$d=$& 1& 2 & 3& 4& 5& 6\\
\midrule
& $1 \cdot \hat P$ & $59 \cdot  \hat P$ & $1.391 \cdot  \hat P$ & $20.889 \cdot  \hat P$ & $ 237.706 \cdot  \hat P$ & $ 1.754.744 \cdot  \hat P$\\
\end{tabular}
\caption{Overall number of subdomain problems, each one of size $2^L/\hat P$, in the combination method with domain decomposition of each subproblem.} 
\label{NSPr}
\end{table}

We see that we obtain a large number of subdomain problems to be solved in a doubly parallel way (one level of parallelism stems from the combination formula itself, the other stems from the domain decomposition of each subproblem of the combination method into subdomain problems). This will allow the use of extremely large parallel compute systems, i.e. the larger $d$ is the larger the parallel system is that can be employed in a meaningful way. Furthermore the use of the fault-tolerant domain decomposition method as the inner solver for the subproblems in the combination method results in a fault-tolerant and parallel solver for the combination method in a natural way. There, the fault-repair mechanism is provided on the fine domain decomposition level and not just on the coarse subproblem level of the combination method, as it was previously done in \cite{Harding14, Pflueger.Bungartz.Griebel.ea.2014, Ali2016ComplexSA, ober17, Rentrop.Griebel.2020}.

For our simple elliptic Poisson problem we encounter very short computing times in the 10 second range for each of the subproblems arising in the combination method, and therefore a fault occurs extremely rarely during its run time. The true value of our fault-tolerant approach will rather become apparent in time-dependent parabolic problems with long time horizons and long computation times, where our elliptic DDM-solver is employed in each time step of an implicit time-discretization method within the combination method. 
This however is future work.

\subsection*{Acknowledgment}
Michael Griebel was supported by a research grant 
of the {\em Bayerische Akademie der Wissenschaften} for the project {\em Stochastic subspace correction as a fault-tolerant sparse grid combination method on massive parallel compute systems for high-dimensional parametric diffusion problems}. The support of the Leibniz-Rechenzentrum in Garching, the support of Prof. Dr. Dieter Kranzlm\"uller, LRZ Garching, LMU M\"unchen, Institut f\"ur Informatik, and the support of Prof. Dr. Hans-Joachim Bungartz, TU M\"unchen, Institut f\"ur Informatik,
is greatly acknowledged. Michael Griebel thanks the Institut f\"ur Informatik of the LMU M\"unchen and the Leibniz-Rechenzentrum of the Bayerische Akademie der Wissenschaften for their hospitality.


\end{document}